\documentclass[11pt,a4paper]{amsart}

\usepackage[left=2.8cm, top=3cm,bottom=3cm,right=2.8cm]{geometry}

\usepackage{mathtools,amssymb,amsthm,mathrsfs,calc,graphicx,stmaryrd,xcolor,dsfont,pgfplots,bbm,float,array,epsfig,color}
\usepackage[numbers,square]{natbib}
\usepackage[british]{babel}
\usepackage{amsfonts,amsmath,amsthm,amssymb}
\usepackage{verbatim}
\usepackage{multirow}
\usepackage{tikz}
\usepackage{bbm}
\usepackage[shortlabels]{enumitem}
\usepackage{subcaption}
\usepackage[hidelinks]{hyperref}

\numberwithin{equation}{section}
\numberwithin{figure}{section}

\newtheorem{thm}{Theorem}[section]
\newtheorem{lem}[thm]{Lemma}
\newtheorem{cor}[thm]{Corollary}
\newtheorem{prop}[thm]{Proposition}

\theoremstyle{remark}
\newtheorem{rem}[thm]{Remark}
\newtheorem{example}[thm]{Example}

\theoremstyle{definition}

\numberwithin{equation}{section}

\newcommand{\tod}{\overset{{\rm d}}{\longrightarrow}}
\newcommand{\topr}[1][]{\overset{\mathbb{P}^{#1}}{\longrightarrow}}

\renewcommand{\P}{\mathbb{P}}
\newcommand{\E}{\mathbb{E}}
\newcommand{\Ev}{\vec{\E}}
\newcommand{\F}{\mathcal{F}}
\newcommand{\G}{\mathbb{G}}
\newcommand{\M}{\mathcal{M}}
\newcommand{\N}{\mathbb{N}}
\newcommand{\R}{\mathbb{R}}
\newcommand{\Z}{\mathbb{Z}}
\newcommand{\C}{\mathbb{C}}
\newcommand{\Var}{\mathbb{V}{\rm{ar}}}
\newcommand{\Varv}{\vec{\mathbb{V}}{\rm{ar}}}
\newcommand{\A}{\mathcal{A}}
\newcommand{\Pow}{\mathcal{P}}
\newcommand{\Tree}{\mathcal{T}}
\newcommand{\UHTree}{\mathcal{I}}
\newcommand{\Lap}{\mathcal{L}}
\newcommand{\Gen}{\mathcal{G}}
\newcommand{\cZ}{\mathcal{Z}}
\newcommand{\Av}{\vec{A}}

\newcommand{\cC}{\mathcal{C}}
\newcommand{\Cc}{\mathcal{C}}
\newcommand{\cG}{\mathcal{G}}
\newcommand{\cN}{\mathcal{N}}
\newcommand{\Surv}{\mathcal{S}}
\newcommand{\dLambda}{{\partial\Lambda}}
\newcommand{\intLambda}{{\Lambda \setminus \dLambda}}
\renewcommand{\Im}{{\rm Im}}

\newcommand{\eps}{\varepsilon}
\newcommand{\imag}{\mathrm{i}}
\DeclareMathOperator{\1}{\mathbbm{1}}

\newcommand{\I}{\mathbf{I}}
\newcommand{\Id}{\mathbf{I}_d}
\newcommand{\Ip}{\mathbf{I}_p}
\newcommand{\Ik}{\mathbf{I}_k}
\newcommand{\Ikl}{\mathbf{I}_{k_\lambda}}
\newcommand{\In}{\mathbf{I}_n}

\newcommand{\bb}{\mathbf{b}}

\newcommand{\be}{\mathbf{e}}
\newcommand{\bu}{\mathbf{u}}
\newcommand{\bv}{\mathbf{v}}
\newcommand{\bw}{\mathbf{w}}

\newcommand{\bq}{\mathbf{q}}
\newcommand{\bxi}{\boldsymbol{\xi}}
\newcommand{\bmu}{\boldsymbol{\mu}}
\newcommand{\bV}{\mathbf{V}}
\renewcommand{\Re}{{\rm Re}}
\newcommand{\trans}{\intercal}
\newcommand{\dd}{\mathrm{d}}

\newcommand{\stirfst}[2]{\genfrac{[}{]}{0pt}{}{#1}{#2}}
\newcommand{\stirscd}[2]{\genfrac{\{}{\}}{0pt}{}{#1}{#2}}

\definecolor{ala}{RGB}{10,100,30}

\begin{document}

\title[Asymptotic fluctuations in multi-type CMJ processes]{Asymptotic fluctuations in supercritical multi-type Crump--Mode--Jagers processes}

\authors{
	Konrad~Kolesko\footnote{Wroc{\l}aw University of Science and Technology, Poland. Email: konrad.kolesko@pwr.edu.pl}
	\and
	Alicja~Ko\l{}odziejska\footnote{University of Gie\ss en, Germany and University of Wroc\l{}aw, Poland. Email: alicja.kolodziejska@uni-giessen.de, alicja.kolodziejska@math.uni.wroc.pl}
	\and
	Matthias~Meiners\footnote{University of Gie\ss en, Germany. Email: matthias.meiners@math.uni-giessen.de}
}

\begin{abstract}
	We consider an irreducible, supercritical multi-type general branching process (Crump--Mode--Jagers process) with finite type space, counted with a random characteristic. We prove that, under certain second moment assumptions, a central limit theorem holds for the process. Our result extends the single-type central limit theorem obtained recently in [Ann.\ Probab.\ 52 (2024), no.\ 4, 1538-1606] and unifies various limiting results for specific branching processes.
\end{abstract}

\date{}
\maketitle
{\footnotesize \noindent \textbf{Keywords:} {Asymptotic fluctuations, central limit theorem, multi-type general branching process, random characteristic, Laplace transform, Nerman’s martingale} \\
\textbf{MSC2020 subject classifications:} primary 60J80; secondary 60F05, 60G44.}
\thispagestyle{empty}

\tableofcontents

\section{Introduction}		\label{sec:Intro}

In the present paper, we continue the research from \cite{Iksanov+al:2024} and extend the central limit theorem
for the single-type Crump--Mode--Jagers (CMJ) branching process counted with random characteristic
proven there to the irreducible multi-type model with finite type space.

\subsection{The multi-type Crump--Mode--Jagers model}	\label{subsec:CMJ model}

CMJ processes, also called general branching processes,
are special branching processes that can be used to model the evolution of populations.
In the CMJ model, the population consists of individuals,
each of whom produces a random number of offspring independently of all other individuals in the current and previous generations.
The offspring, in turn, reproduce according to the same mechanism, giving rise to a branching structure.
The CMJ model adopts a natural and general perspective on reproduction:
individuals may be born at various ages during the lifetime of a potential parent.
Each birth event may produce one or several offspring.
Formally, the reproduction of an individual takes place at the points of a random point process,
translated by the individual’s birth time.
Moreover, individual life spans are allowed to be random and need not be independent of the reproduction process of the individual.
The general branching process counted with a random characteristic at time $t$ is defined as the sum over all individuals ever born,
where each individual’s contribution is given by a random characteristic that may depend on its entire life history
(e.g., age at $t$, lifespan). This unified formulation simultaneously captures a wide range of quantities of interest,
including the total number of births up to time $t$, the population size at time $t$,
and the number of individuals alive at time $t$ whose age is below a given threshold $a>0$.

Multi-type CMJ processes further assign a type to each individual,
with that individual’s life span, reproduction and random characteristic potentially depending on its type.
This provides even greater flexibility in modeling.
A precise description of the CMJ model is given in Section \ref{sec:setup} below.

The general branching process was introduced independently in the late 1960s by Ryan \cite{Ryan:1968},
by Crump and Mode \cite{Crump+Mode:1968}, and by Jagers \cite{Jagers:1969} in order to unify and extend existing branching models.
In particular, the model encompasses the Bienaym\'e--Galton--Watson, Yule, continuous-time Markov branching,
Sevastyanov, and Bellman--Harris processes.

Today, the applications of general branching processes extend far beyond their original purpose of modeling biological populations
and can be found in fields as diverse as biology, ecology, physics, and computer science.
Several monographs \cite{Harris:1963,Asmussen+Hering:1983,Athreya+Ney:1972,Haccou+al:2007,Jagers:1975,Kimmel+Axelrod:2015}
offer comprehensive introductions to the field.

We do not aim to provide a comprehensive survey of all applications of the multi-type CMJ process here;
instead, we highlight representative examples, including cell proliferation and differentiation \cite{Nordon+al:2011},
modeling cancer cell growth \cite{Durrett:2015,Foo+Leder:2013,Sehl+al:2011}.

Overall, a particularly important area of application is provided by multi-type continuous-time Markov branching processes,
which are contained in the CMJ model as a special case as briefly noted above.
This is especially relevant because they are connected to urn models via the Athreya--Karlin embedding \cite{Athreya+Karlin:1968},
so that any model that can be coupled to an urn model can also be studied using multi-type Markov branching processes,
and therefore by means of CMJ processes.
This connection has been studied in depth in \cite{Janson:2004}, which also presents various applications.
That paper serves as a comprehensive reference for the literature up to 2004.
We revisit
and extend results on urn models and related coupled models,
including in particular balanced and cyclic P\'olya urns \cite{Mueller+Neininger:2018},
the elephant random walk \cite{Baur+Bertoin:2016,Guerin+al:2025}, and $m$-ary search trees \cite{Chauvin+al:2014}.
Detailed analyses of specific models are provided in Section \ref{sec:applications} below.

\subsection{Laws of large numbers and central limit theorems for the CMJ model}	\label{subsec:LLN and CLTs}
The natural starting point for studying the asymptotic behavior of a supercritical multi-type CMJ process
is the weak and strong law of large numbers.
In the Malthusian setting, which we introduce in detail in Section \ref{subsec:assumptions},
there exists an $\alpha > 0$, referred to as the Malthusian parameter,
such that the general branching process counted with random characteristic $\varphi$,
which we denote by $(\cZ_t^\varphi)_{t \geq 0}$, satisfies
\begin{equation}	\label{eq:LLN}
e^{-\alpha t} \cZ_t^\varphi \to c W	\quad	\text{as } t \to \infty
\end{equation}
for some constant $c \geq 0$ and a finite random variable $W \geq 0$
provided that $\varphi$ satisfies appropriate conditions.

The laws of large numbers were first established by Nerman \cite{Nerman:1979,Nerman:1981}
in the non-lattice case, for both the single-type model and the case of finitely many types.
For the latter case, these laws of large numbers were later reproved \cite{Iksanov+Meiners:2015}
through a reduction to the single-type case, in a paper primarily concerned with the rate of convergence in the laws of large numbers.
Earlier results for special, mainly Markovian, cases exist; for a historical overview, see the introduction of \cite{Nerman:1981}.
The law of large numbers for multi-type CMJ processes with an abstract type space
is due to Jagers \cite{Jagers:1989}.

In both the single-type and multi-type settings, the law of large numbers
is closely linked to the convergence of the intrinsic martingale, the limit of which is the random variable $W$ from \eqref{eq:LLN}.
In the case of an abstract type space, it is studied in \cite{Jagers:1989} as well as in \cite{Olofsson:2009}
using the spinal decomposition.

Central limit theorems for CMJ processes
have first been addressed in the special case of Markov branching.
In the setting of a multi-type continuous-time Markov branching process with a finite type space,
where reproduction occurs only at death,
Athreya \cite{Athreya:1969a,Athreya:1969b} established a central limit theorem
and Janson \cite{Janson:2004} proved a functional central limit theorem.
Asmussen and Hering \cite[Section VIII.3]{Asmussen+Hering:1983} established results
concerning the asymptotic fluctuations of multi-type Markov branching processes with general type spaces.
In principle, these also encompass the single-type general branching process,
which may be regarded as a Markov process where an individual’s type at time~$t$ is given by its complete life history up to that time.
However, because this type space is vast, the assumptions of \cite{Asmussen+Hering:1983} are usually violated,
except in special cases such as the Galton--Watson process.

Beyond the Markov case, there are mainly results for the single-type case.
In \cite{Iksanov+al:2024},
in the setting where the Malthusian parameter $\alpha > 0$ exists and the intrinsic martingale is uniformly integrable,
a central limit theorem for single-type CMJ process counted with random characteristic is established. More precisely,
under suitable as\-sump\-tions---including the finiteness of second moments---it is shown that there exist $k \in \N_0$ and a function $H(t)$,
given as a finite random linear combination of functions of the form $t^j e^{\lambda t}$ with $\alpha/2 \leq \Re(\lambda) < \alpha$,
such that
\begin{equation}	\label{eq:CLT CMJ}
\frac{\cZ_t^{\varphi} - a e^{\alpha t} W - H(t)}{\sqrt{t^k e^{\alpha t}}}
\end{equation}
converges in distribution to a normal random variable with random variance.
This result unifies and extends various central limit theorem-type results for specific branching processes.
Before doing so, however, it should be noted that this result has several predecessors and, by now, also a number of successors.

Jagers and Nerman \cite{Jagers+Nerman:1984b} proved a limit theorem establishing convergence in distribution for CMJ processes.
However, the conditions in the cited paper can be challenging to verify, even for relatively simple characteristics.
Janson \cite{Janson:2018} analyzed the asymptotic fluctuations of single-type supercritical CMJ processes in the lattice case.
Janson and Nei\-nin\-ger \cite{Janson+Neininger:2008}
investigated Kolmogorov’s conservative fragmentation model, recast it in the framework of general branching processes,
and established a second-order convergence result for the number of individuals born up to time $t$
in a CMJ process, relying on restrictive assumptions on the reproduction point process
that stem from the fragmentation model.
Henry \cite{Henry:2017} proved a central limit theorem for supercritical binary homogeneous CMJ processes.
A further related contribution is the work of Charmoy, Croydon, and Hambly \cite{Charmoy+al:2017},
who analyze fluctuations of the eigenvalue counting function associated with certain random fractals.
This problem can likewise be treated via a central limit theorem for a suitably chosen centered CMJ process,
which they prove.
However, in their setting the random characteristic may depend on an individual’s progeny,
placing the model outside the scope of the framework described above.

The method developed in \cite{Iksanov+al:2024} was adapted by Kolesko and Sava-Huss \cite{Kolesko+Sava-Huss:2023}
to establish a limit theorem in the spirit of \eqref{eq:CLT CMJ} for Galton--Watson processes with finitely many types.  
In this setting, the terms $e^{\lambda t}$ that appear in $H(t)$ in \eqref{eq:CLT CMJ} can be written as $\gamma^{\,t}$,  
where the $\gamma$ are the eigenvalues of the mean offspring matrix.
The eigenvalues $\gamma \geq \sqrt{\rho}$ then arise in an expansion of $\cZ_t^\varphi$,
where $\rho = e^\alpha$ denotes the dominant eigenvalue.

Further, the proof technique established in \cite{Iksanov+al:2024} was adapted in \cite{Berzunza+Connor:2025}
to also cover random characteristics $\varphi$ that may depend on an individual’s life history as well as on its descendant processes up to a fixed generation in order to make the result applicable in the study of fringe trees \cite{Holmgren+Janson:2017}.

Another related paper in the single-type case is the recent preprint by Daly and Wade \cite{Daly+Wade:2025}
in which the authors investigate Kakutani’s random interval-splitting process and provide rates of convergence
in the central limit theorem of the longest subinterval after $n$ splits. This can be translated
into a Berry--Esseen-type convergence result for the number of individuals born up to time $t$ in a CMJ process
with a specific (binary) offspring point process.

\section{Setup and main results}\label{sec:setup}

\subsection{General notation}

Throughout the paper, we will denote the set of strictly positive natural numbers by $\N$
and write $\N_0 \coloneqq \N \cup \{0\}$. For $n \in \N$, we write $[n] \coloneqq \{1,\dots,n\}$.

For $\mathbb{K} = \R,\C$ we denote by $\mathbb{K}^n$ and $\mathbb{K}^{1\times n}$
the spaces of $n$-dimensional column and row vectors, respectively.
For any $n\in\N$ and $k \in [n]$, the $k$th standard base column vector of $\mathbb{K}^{n}$ is denoted by $\be_{k,n}$.
In most cases we shall write $\be_k$ with the dimension being implicit.
For any vector $x$, we denote the Euclidean norm of $x$ by $\|x\|$.

By $\M_{n\times m}(\mathbb{K})$ we denote the space of $n \times m$ matrices over $\mathbb{K}$. Every $\M_{n\times n}(\mathbb{K})$ is equipped with the operator norm $\| \cdot \|$ induced by the Euclidean norm.
Observe that $\|\cdot\|$ is dominated by the Hilbert--Schmidt norm, denoted by $\|\cdot\|_{\mathrm{HS}}$.

Unless otherwise stated, for a matrix-valued function $f$ and $\gamma\in\N$,
we will denote by $f^\gamma$ the matrix-valued function with $(i,j)$-th entry given by $(\be_i^\trans f(t) \be_j)^\gamma$,
that is, we apply the power entry-wise.
Accordingly, for a matrix-valued random function $\varphi$,
we denote the matrix with entries $\E[\be_i^\trans \varphi(t) \be_j]$ by $\E[\varphi](t)$
and the matrix with entries $\E[|\be_i^\trans \varphi(t) \be_j - \E[\be_i^\trans \varphi(t)\be_j]|^2]$ by $\Var[\varphi](t)$.

\subsection{Multi-type general branching process}\label{sec:CMJ definition}
In this section we formally introduce the multi-type Crump-Mode-Jagers (CMJ) branching process counted with random characteristic.
We set the type space to be $[p] = \{1,\dots,p\}$ for some fixed $p \in \N$. The process starts with a single individual, the ancestor, born at time $0$. The type of the ancestor, denoted $\tau(\varnothing)$, may be chosen randomly according to any probability distribution on $[p]$. The ancestor produces offspring according to the \emph{reproduction (point) process}
$$\bxi = (\xi^{ij})_{i,j\in[p]} = \begin{pmatrix}
	\xi^{11} & \xi^{12} & \dots & \xi^{1p} \\
	\vdots &&&\vdots \\
	\xi^{p1} & \xi^{p2} & \dots & \xi^{pp}
\end{pmatrix},$$
where for every $i,j\in[p]$,
$$	\xi^{ij} = \sum_{k=1}^{N^{ij}} \delta_{X^{ij}_k}	$$
is a point process on $[0,\infty)$. Here, the $N^{ij} = \xi^{ij}([0,\infty))$, $i,j \in [p]$ are random variables taking values in $\N_0 \cup \{\infty\}$.
Moreover, $\bxi$ is assumed to be independent of $\tau(\varnothing)$.
If the ancestor's type is $i$, then its $k$th child of type $j$ is born at time $X^{ij}_k$ for $k=1,\ldots,N^{ij}$.
Let $N^i \coloneqq \sum_{j=1}^p N^{ij}$ and
\begin{equation}\label{eq:repr-process-flat}
\xi^{i} \coloneqq \sum_{j=1}^p \sum_{k=1}^{N^{ij}} \delta_j \otimes \delta_{X^{ij}_k}
\eqqcolon \sum_{k=1}^{N^i} \delta_{(\tau^i_k, X^{i}_k)}
\end{equation}
for the process encoding all birth times and types of the children of the ancestor in the case $\tau(\varnothing) = i$.
We assume without loss of generality that $0 \leq X^i_1 \leq X^i_2 \leq \ldots$
For completeness, we write $X^i_k = \infty$ when $k > N^i$.
The children of the ancestor form the first generation of the process.
Each of them then reproduces according to an independent
copy of $\bxi$, and so on.

Formally, the individuals are indexed by the Ulam-Harris tree $\UHTree = \bigcup_{n\in\N_0} \N^n$,
with $\N^0 = \{\varnothing\}$ being the singleton that contains only the ancestor.
We use the abbreviation $u = u_1 \dots u_n$ for a tuple $u = (u_1, \dots, u_n) \in \N^n$ and call $|u| \coloneqq n$
the \emph{generation} of $u$.
In this context, $u_1\dots u_n$ encodes the genealogy of $u$: $u_1$ is the $u_1$th child of the ancestor,
$u_1 u_2$ is the $u_2$th child of $u_1$, etc.
For two tuples $u = u_1\dots u_n$ and $v = v_1\dots v_m$,
we write $uv$ for the concatenation $u_1\dots u_n v_1\dots v_m$.
More generally, $u$ is a \emph{descendant} of $v$ (denoted $v \preceq u$), if $u = vw$ for some $w \in \UHTree$.
The individual $v$ is then a \emph{progenitor} of $u$.
We write $v \prec u$ if $v \preceq u$ and $v \neq u$,
and call $v$ a \emph{strict} progenitor of $u$.

Let $(\Omega_u, \A_u, \mathrm{P}_u)$ be a probability space,
which we call the \emph{life space} of the individual $u$, $u \in \UHTree$.
We assume that the $(\Omega_u, \A_u, \mathrm{P}_u)$, $u \in \UHTree$
are all identical
and that $\bxi$ is defined on $(\Omega_\varnothing,\A_\varnothing,\mathrm{P}_\varnothing)$.
Define, for $i \in [p]$,
\begin{equation*}
(\Omega, \F, \P^i)
= \bigg([p] \times \bigtimes_{u \in \UHTree} \Omega_u, \Pow([p]) \otimes \bigotimes_{u \in \UHTree} \A_u,
\delta_i \otimes \bigotimes_{u \in [p]}\mathrm{P}_u\bigg)
\end{equation*}
where $\Pow([p])$ is the power set of $[p]$.
We adopt the convention that $\P$ denotes an arbitrary convex combination of the measures $\P^i$,
and that $\E$ denotes the corresponding expectation.
For $u \in \UHTree$, we denote by $\pi_u: \Omega \to \Omega_u$ the projection from $\Omega$ onto $\Omega_u$
and write $\bxi_u$ for $\bxi \circ \pi_u$.
Somewhat abusively, we also write $\bxi$ for $\bxi_\varnothing$ when no confusion can arise.

Quantities derived from $\bxi_u$ are denoted with a subscript $u$.
In particular, for each $i\in[p]$,
\begin{equation}	\label{eq:repr-process-flat u}
\xi^{i}_u = \sum_{k=1}^{N^i} \delta_{(\tau^i_{uk}, X^{i}_{uk})}
\end{equation} 
 is a copy of the process given in~\eqref{eq:repr-process-flat}
encoding types and birth times of the children of $u$ if $u$'s type is $i$.
We denote by $\tau(u)$ and $S(u)$ \emph{type} and the \emph{birth time} of $u$, respectively.
These are defined as follows: $\tau(\varnothing)$ is the projection from $\Omega$ onto $[p]$
and $S(\varnothing) = 0$. Furthermore, for $u \in \UHTree$ and $k \in \N$, recursively,
\begin{equation*}
	\tau(uk) \coloneqq \tau^{\tau(u)}_{uk}
\quad	\text{and}	\quad
	S(uk) \coloneqq S(u) + X^{\tau(u)}_{uk}.
\end{equation*}
We write $\Tree = \{u \in \UHTree: S(u)<\infty\}$ for the random subtree consisting of all individuals who are ever born.

Finally, assume there is a vector-valued function $\varphi: \Omega_\varnothing \times \R \to \R^p$,
where it should be recalled that $\R^p$ is considered as the space of $p$-dimensional column vectors
with canonical basis $\be_i$, $i \in [p]$.
We consider characteristics with c\`adl\`ag paths,
i.e., for every $i\in[p]$, the $i$th coordinate of $\varphi$ given by $t \mapsto \be_i^\trans \varphi(t)$
has right-continuous paths with existing left limits.
We write $\varphi_u \coloneqq \varphi \circ \pi_u$, $u \in \UHTree$.
As before, we permit the notation $\varphi = \varphi_\varnothing$ whenever the meaning is clear from context.
In general, $\bxi_u$ and $\varphi_u$ may be dependent.
However, since the type of the individual depends only on its strict progenitors, $\tau(u)$ is independent of $(\bxi_u, \varphi_u)$.
The Crump--Mode--Jagers process counted with characteristic $\varphi$ is the process $(\cZ_t^\varphi)_{t\geq 0}$ given by
\begin{equation}\label{eq:def CMJ}
	\cZ_t^\varphi = \sum_{u\in\Tree} \be_{\tau(u)}^\trans \varphi_u(t - S(u)).
\end{equation}
Whenever considering a random variable $X$ that depends on the type of the ancestor, we write $\Ev[X]$ for the vector in $\R^p$ (or $\C^p$)
whose $i$th entry is given by $\E^i[X]$ provided it is well-defined.
Similarly, if $X$ is a row vector of random variables, $\Ev[X]$ is a matrix whose $i$th row is $\E^i[X]$, $i\in[p]$.
In particular, the mean of $\cZ_t^\varphi$, denoted $m_t^\varphi = \Ev[\cZ_t^\varphi]$,
is the vector whose $i$th entry is given by $\E^i[\cZ_t^\varphi]$.
Similarly, $\Varv[\cZ_t^\varphi] \in \R^p$ is the column vector
with entries $\Var^i[\cZ_t^\varphi] \coloneqq \E^i[|\cZ_t^\varphi - \E^i[\cZ_t^\varphi]|^2]$.
Note that, since the characteristic $\varphi$ and the reproduction process $\bxi$ do not depend on the ancestor's type,
we keep the notation $\E[\varphi], \E[\bxi]$ for their (vector- or matrix-valued) expectations.

Sometimes it is convenient to consider CMJ processes counted with matrix-valued characteristics. Since their definition uses the Kronecker product, we postpone it, as well as the definition of $\Ev$ in this setting, until Section~\ref{sec:matrices}.

\begin{example}	\label{Exa:model with lifespan}
To illustrate the purpose of introducing random characteristics which may depend on $\bxi$,
consider a model in which the ancestor has a random lifespan $\zeta$.
The number of individuals alive at time $t$ is then given by
	\begin{equation*}
		\sum_{u \in \Tree} \1_{\{S(u) \leq t < S(u) + \zeta_u\}} = \sum_{u \in \Tree} \be_{\tau(u)}^\trans \varphi_u(t - S(u))
	\end{equation*}
for $\varphi(t) = \sum_{i=1}^p \be_i \1_{[0,\zeta)}(t)$.
One may also weight the particles alive depending on their type
by considering a characteristic $\varphi(t) = \bw \1_{[0,\zeta)}(t)$ for some vector $\bw \in \R^p$.
We refer the reader to Section~\ref{sec:applications} for more examples.
\end{example}

\begin{rem}
In general, the infinite sum in~\eqref{eq:def CMJ} need not be well defined.
In Section~\ref{sec:existence}, we provide several conditions in terms of $\bxi$ and $\varphi$
which are sufficient for the existence of the process $(\cZ_t^\varphi)_{t\geq 0}$.
\end{rem}

\subsection{Assumptions}	\label{subsec:assumptions}
	
In this section we present the assumptions under which our main result will be stated.
They fall into two categories: assumptions on the reproduction process $\bxi$
and assumptions on the characteristic $\varphi$.

To state the former, denote by $\bmu$ the \emph{intensity measure} of $\bxi$, i.e., the matrix of measures given by
\begin{equation*}
	\bmu = (\mu^{ij})_{i,j \in [p]} \coloneqq (\E[\xi^{ij}])_{i,j\in[p]} = \E[\bxi].
\end{equation*}
We call $\bmu$ $d$-lattice if $d > 0$ is the maximal real number
for which there exists a measurable \emph{shift function} $\gamma : [p] \to [0,d)$, $i \mapsto \gamma_i$,
such that for each $i,j\in[p]$,
the measure $\mu^{ij}$ is concentrated on $d \Z + \gamma_i - \gamma_j$.
If no such $d$ exists, then we call $\bmu$ non-lattice.

Throughout the paper, we assume that the intensity measure $\bmu$ is either non-lattice or $1$-lattice with zero shift.
In the non-lattice case we set $\G \coloneqq \R$ and denote by $\ell$ the Lebesgue measure on $\R$,
while in the lattice case, we set $\G \coloneqq \Z$ and $\ell$ denotes the counting measure on $\Z$.

Further, let $\Lap\bmu$ be the Laplace transform of $\bmu$, applied entry-wise, that is,
\begin{equation*}
	\Lap\bmu(z) = \int e^{-zx} \bmu(\dd x) = \begin{pmatrix}
		\int e^{-zx} \mu^{11}(\dd x) & \int e^{-zx} \mu^{12}(\dd x) & \dots & \int e^{-zx} \mu^{1p}(\dd x) \\
		\vdots &&&\vdots \\
		\int e^{-zx} \mu^{p1}(\dd x) & \int e^{-zx} \mu^{p2}(\dd x) & \dots & \int e^{-zx} \mu^{pp}(\dd x)
	\end{pmatrix}.
\end{equation*}
Recall that a nonnegative matrix is called \emph{irreducible} if it has a power that has positive (possibly infinite) entries only.

Our second standing assumption is that $\Lap\bmu(0)$ is irreducible.

The following assumption is essential for the law of large numbers and, subsequently, also for the central limit theorem.
\begin{enumerate}[label={\bf(A\arabic*)}, resume]
	\item
		There exists the Malthusian parameter $\alpha>0$, i.e., there is an $\alpha > 0$
		such that $\Lap\bmu(\alpha)$ has finite entries only and $1$ is the dominating (Perron--Frobenius) eigenvalue of $\Lap\bmu(\alpha)$.
		Moreover, $-(\Lap\bmu)'(\alpha)\coloneqq \int x e^{-\alpha x} \bmu(\dd x)$ is a finite nonnegative matrix
		with at least one strictly positive entry. \label{assumpt:Malthusian alpha}
\end{enumerate}
Notice that~\ref{assumpt:Malthusian alpha} implies supercriticality and, consequently, also a positive survival probability for the process.
More precisely, if $\Lap\bmu(\alpha)$ has finite entries only, then so does $\Lap\bmu(\theta)$ for every $\theta > \alpha$.
Since $(\Lap\bmu)'(\alpha)$ has a strictly negative entry, the associated Laplace transform, $\Lap \mu^{ij}$ say,
is strictly decreasing in a right neighborhood of $\alpha$, which means that $\mu^{ij}((0,\infty))>0$ and, consequently,
$\Lap \mu^{ij}$ is strictly decreasing on $\{\theta \in [0,\infty): \Lap \mu^{ij}(\theta) < \infty\}$.
This means that either $\Lap\mu(0)$ has an infinite entry, or has finite entries only and Perron-Frobenius eigenvalue $\rho > 1$.
Theorem II.7.1 in \cite{Harris:1963} yields that the survival probability $\P^i(\Surv) > 0$
for every $i \in [p]$ where
\begin{equation}	\label{eq:survival set}
\Surv \coloneqq \{\#\{u \in \Tree: |u|=n\} \to \infty \text{ as } n \to \infty\}.
\end{equation}
Note that in our framework the extinction--explosion dichotomy is implicit, so that, almost surely,
survival occurs only on the explosion set.

Under assumption~\ref{assumpt:Malthusian alpha},
let $\bu^\trans = (\bu_1,\ldots,\bu_p) \in \R^{1 \times p}$ and $\bv = (\bv_1,\ldots,\bv_p)^\trans \in \R^p$
denote left and right eigenvectors of $\Lap\bmu(\alpha)$, respectively.
By the Perron--Frobenius theorem, we can choose $\bu, \bv$ with all entries strictly positive.
By an appropriate rescaling, we can arrange that
\begin{equation}\label{eq:u-v}
	\bu^\trans \bv = 1 = \| \bv \|_1
\end{equation}
where $\|\cdot\|_1$ denotes the $1$-norm on $\R^p$.
Define
\begin{equation}\label{def:beta}
	\beta = \bu^\trans (-(\Lap\bmu)'(\alpha)) \bv.
\end{equation}
Notice that $\beta > 0$ since $\bu$ and $\bv$ have positive entries only and $-(\Lap\bmu)'(\alpha)$ is nonnegative
with at least one strictly positive entry.
The next assumption implies that the Laplace transform of $\bmu$ is finite in the half-plane $\{z \in \C : \Re(z) \geq \vartheta\}$
for some $\vartheta < \alpha/2$.
\begin{enumerate}[label={\bf(A\arabic*)}, resume]
	\item There exists $\vartheta \in (0,\alpha/2)$ such that $\Lap\bmu(\vartheta)$ has finite entries only. \label{assumpt:LaplaceDomain}
\end{enumerate}
Moreover, the asymptotic behavior of the general branching process is closely connected to the roots of the equation
\begin{equation}	\label{eq:LapChar}
	\det(\Ip - \Lap\bmu(z)) = 0
\end{equation}
where here and throughout, $\Ip$ denotes the $p \times p$ identity matrix.
With assumptions~\ref{assumpt:Malthusian alpha} and~\ref{assumpt:LaplaceDomain} in force,
in the non-lattice case, let $\Lambda \coloneqq \{z \in \C: z \text{ solves }~\eqref{eq:LapChar} \text{ and } \Re(z) \geq \frac\alpha2\}$.
In the lattice case, let
$\Lambda \coloneqq \{z \in \C: z \text{ solves }~\eqref{eq:LapChar},\; \Re(z) \geq \frac\alpha2 \text{ and }- \pi < \Im(z) \leq \pi\}$.
In both cases, denote $\dLambda = \{\lambda \in \Lambda : \Re(\lambda) = \frac\alpha2\}$.
We will refer to the set $\dLambda$ as the \emph{critical line}.
Throughout the paper, we will assume that
\begin{enumerate}[resume, label={\bf(A\arabic*)}]
	\item \label{assumpt:LambdaFin}
	$\Lambda$ is a finite set.
\end{enumerate}
Next, we state a moment assumption on the reproduction process.
Each root of the equation~\eqref{eq:LapChar} is a pole of the meromorphic matrix-valued function $z \mapsto (\Ip - \Lap\bmu(z))^{-1}$, see Section~\ref{sec:mean-and-mtgs} for details.
Let $k^*$ denote the maximum of the orders of poles that lie on the critical line.
If no such pole exists, set $k^*=1/2$.
\begin{enumerate}[resume, label={\bf(A\arabic*)}]
	\item For every $i,j \in [p]$, the random variable \label{assumpt:hreprMoments}
	\begin{equation*}
		\int (1 + t)^{k^*-\frac12} e^{-\frac\alpha2t} \, \xi^{ij}(\dd t)
		= \sum_{k=1}^{N^{ij}} \big(1 + X^{ij}_k\big)^{\!k^*-\frac12} e^{-\frac\alpha2 X^{ij}_k}
	\end{equation*} 
	has finite second moment.
\end{enumerate}

We now turn to assumptions concerning the random characteristic $\varphi$.
We say that a matrix-valued function $f = (f_{ij})_{i\in[n], j\in[k]} : \R \to \M_{n\times k}(\C)$ is directly Riemann integrable
if each of its entries $f_{ij}$ is directly Riemann integrable.
For the definition and a discussion of the basic properties of direct Riemann integrability, we refer to \cite[Section 3.10.1]{Resnick:1992}.
\begin{enumerate}[resume, label={\bf(A\arabic*)}]
	\item \label{assumpt:dRimean}
	$\varphi(t) \in L^1(\P)$ for any $t\in \R$ and $t \mapsto e^{-\alpha t} \E[\varphi](t)$
	is directly Riemann integrable.
	\item \label{assumpt:dRivar} 
	$\varphi(t) \in L^2(\P)$ for any $t\in \R$ and $t \mapsto e^{-\alpha t} \Var[\varphi](t)$
	is directly Riemann integrable.
	\item	\label{assumpt:local ui}
	For any $t \in \R$ there exists an $\varepsilon > 0$ such that
	the family $(|\varphi(x)|^2)_{|x-t| \leq \varepsilon}$ is uniformly integrable
	with respect to $\P$.
\end{enumerate}
We note that~\ref{assumpt:dRimean} is needed for the law of large numbers,
which serves as a prerequisite for the central limit theorem, whereas
\ref{assumpt:dRivar} is required for the central limit theorem.
For later reference, also note that~\ref{assumpt:local ui} is needed to ensure that condition~\ref{assumpt:dRivar}
is closed under linear combinations. More precisely, if $\varphi$ and $\psi$ are characteristics satisfying
\ref{assumpt:dRivar} and~\ref{assumpt:local ui}, then every linear combination $\alpha \varphi + \beta \psi$, where $\alpha,\beta \in \C$,
also satisfies~\ref{assumpt:dRivar} and~\ref{assumpt:local ui}.

Proposition 2.6 in \cite{Iksanov+al:2024} provides sufficient conditions for~\ref{assumpt:dRimean} through~\ref{assumpt:local ui} to hold
in the single-type case.
Next, we state the corresponding multi-type version.
To this end, for a function $f:\R^p \to \R$, $f=(f_1,\ldots,f_p)^\trans$ let $f^* = (f_1^*,\ldots,f_p^*)^\trans$
where $f_i^*(t) \coloneqq \sup_{|x-t| \leq 1} |f_i(x)|$, $i \in [p]$.

\begin{prop}	\label{prop:char}
	\begin{enumerate}[\normalfont(i), wide] 
		\item If $f : \R \to \R$ is c\`adl\`ag and $\int f^*(x) \, \dd x < \infty$,
		then $f$ is directly Riemann integrable. Conversely, if $f: \R \to \R$ is directly Riemann integrable, then so is $f^*$.
		\item If~\ref{assumpt:Malthusian alpha} holds and if, for a random characteristic $\varphi$, the matrix
		\begin{equation*}
		\int \E[\varphi^*](x) e^{-\alpha x} \, \dd x
		\end{equation*}
		has finite entries only, then~\ref{assumpt:dRimean} holds for $\varphi$.
		\item \label{item:dRivar conds} If~\ref{assumpt:Malthusian alpha} holds and if, for a random characteristic $\varphi$, the matrix
		\begin{equation*}
		\int \E[(\varphi^*)^2](x) e^{-\alpha x} \, \dd x
		\end{equation*}
		has finite entries only, then~\ref{assumpt:dRivar} and~\ref{assumpt:local ui} hold for $\varphi$.
		Moreover, if the characteristic $\varphi$ is supported on $[0,\infty)$  then~\ref{assumpt:dRimean} holds as well.
	\end{enumerate} 
\end{prop}
\begin{proof}
This follows from an entry-wise application of \cite[Proposition 2.6]{Iksanov+al:2024}.
Under the assumption \ref{item:dRivar conds}, if $\varphi$ is supported on the half-line, then the finiteness of $	\int \E[\varphi^*](x) e^{-\alpha x} \, \dd x$ follows from the Cauchy--Schwartz inequality.
\end{proof}

\begin{rem}\label{rem:charAssumpts}
Note that for a matrix-valued characteristic $\varphi$, each entry of $\varphi^2(t)$ is bounded by $\|\varphi(t)\|_{\mathrm{HS}}^2$,
where $\|\cdot\|_{\mathrm{HS}}$ denotes the Hilbert--Schmidt norm. In particular, if
\begin{equation*}
\int e^{-\alpha x} \E\Big[\sup_{|x-t|\leq1} \|\varphi(t)\|_{\mathrm{HS}}^2\Big] \, \dd x < \infty,
\end{equation*}
then, by Proposition~\ref{prop:char}\ref{item:dRivar conds}, the characteristic $\varphi$ satisfies~\ref{assumpt:dRivar}
and~\ref{assumpt:local ui}. If it is supported on the positive half-line, then~\ref{assumpt:dRimean} holds as well.
\end{rem}

\subsection{Main results: preliminaries}

In order to state our main theorem,
we recall some known results concerning the mean of $\cZ_t^\varphi$
and certain martingales connected to it.
We refer the reader to Section~\ref{sec:mean-and-mtgs} for more details.

For any $\lambda \in \Lambda$, the meromorphic function $z \mapsto (\Ip-\Lap\bmu(z))^{-1}$ may be represented in some neighborhood of $\lambda$ in the form
\begin{equation*}
(\Ip - \Lap\bmu(z))^{-1} = \sum_{j = 1}^{k_\lambda}(z-\lambda)^{-j} A_{\lambda, j} + H_\lambda(z),
\end{equation*}
where $k_\lambda \in \N$ does not exceed the multiplicity of $\lambda$ as a root of~\eqref{eq:LapChar},
$A_{\lambda,j}$ are $p\times p$ matrices over $\C$ and $H_\lambda$ is matrix-valued,
the entries being holomorphic functions.
Moreover, for any $j \in[k_\lambda]$,
\begin{equation}\label{eq:matrixRec}
A_{\lambda,j} = \sum_{l=j}^{k_\lambda} \frac{1}{(l-j)!} (\Lap\bmu)^{(l-j)}(\lambda) A_{\lambda,l} = \sum_{l=j}^{k_\lambda} \frac1{(l-j)!}A_{\lambda,l} (\Lap\bmu)^{(l-j)}(\lambda),
\end{equation}
where $(\Lap\bmu)^{(n)}$ denotes the $n$th derivative of $\Lap\bmu$, applied entry-wise.
It follows from \cite{Kolesko+al:2025} that in the given situation, if~\ref{assumpt:Malthusian alpha} through~\ref{assumpt:hreprMoments}
are in force and subject to additional assumptions on the characteristic $\varphi$,
the mean of the process $(\cZ_t^\varphi)_{t\geq 0}$ satisfies the asymptotic relation
\begin{equation}	\label{eq:meanAs1}
m_t^\varphi
= \sum_{\lambda \in \Lambda} e^{\lambda t} \sum_{k=0}^{k_\lambda-1} A_{\lambda,k+1} \int_\G \E[\varphi](s)\frac{(t-s)^{k}}{k!} e^{-\lambda s}
\, \ell(\dd s) + r(t),
\end{equation}
where
\begin{equation}	\label{eq:rasympt}
r(t) = O\big(e^{\frac\alpha2t}/(1+t^2) \big) \quad \text{as } t\to\infty,\ t\in\G.
\end{equation}
Note that, regrouping the terms,~\eqref{eq:meanAs1} may be rewritten as
\begin{equation}\label{eq:meanAsymptotics-vecs}
m_t^\varphi = \sum_{\lambda\in\Lambda} e^{\lambda t}\sum_{j=0}^{k_\lambda-1} t^j \bb_{\lambda,j,\varphi} + r(t)
\end{equation}
for vectors $\bb_{\lambda,j,\varphi} \in \C^p$ given by
\begin{equation}\label{def:binomialVecs}
\bb_{\lambda,j,\varphi} = \frac1{j!}\sum_{k=j}^{k_\lambda-1} \frac{1}{(k-j)!} \int_\G (-s)^{k-j} e^{-\lambda s} A_{\lambda,k+1} \E[\varphi](s) \, \ell(\dd s).
\end{equation}
For $t\geq 0$, let
\begin{equation}\label{def:comingGen}
	\cC_t = \{ uj \in \Tree : S(u) \leq t < S(uj) \}
\end{equation}
be the \emph{coming generation at time $t$}. For $\lambda \in \intLambda$ and $i \in [k_\lambda]$, put
\begin{equation*}
	W_t(\lambda,i) = \sum_{u\in\cC_t} e^{-\lambda S(u)} \sum_{j=i}^{k_\lambda} \frac{(-S(u))^{j-i}}{(j-i)!} \be_{\tau(u)}^\trans A_{\lambda,j} 
\end{equation*}
and let $W_t(\lambda)$ be the $k_\lambda \times p$ block matrix defined via
\begin{equation*}
	W_t(\lambda) = \begin{pmatrix}
	W_t(\lambda,1) \\
	W_t(\lambda,2) \\
	\vdots \\
	W_t(\lambda,k_\lambda)
	\end{pmatrix}.
\end{equation*}
By \cite[Theorem 2.1]{Kolesko+al:2025b}, the process $(W_t(\lambda))_{t \geq 0}$ is an entry-wise complex-valued martingale.
Moreover, if~\ref{assumpt:hreprMoments} holds, Theorem 2.3 in \cite{Kolesko+al:2025b}
(and, alternatively, Lemma~\ref{lem:mtgsasCMJ} in the present paper) applies
and yields that $(W_t(\lambda))_{t \geq 0}$ converges almost surely and in $L^2$ to a random matrix
which we denote by $W(\lambda)$. In particular, for each $i \in [k_\lambda]$,
$W_t(\lambda,i) \to W(\lambda,i)$ almost surely as $t\to\infty$, where $W(\lambda,i)$ is the $i$th row of $W(\lambda)$.
Moreover, $\Ev[W(\lambda,j)] = A_{\lambda,j}$ with the convention introduced shortly after equation~\eqref{eq:def CMJ},
i.\,e.\ $\Ev[W(\lambda,j)]$ is the $p\times p$ matrix with $i$th row given by $\E^i[W(\lambda,j)]$.

The martingale connected to $\lambda = \alpha$ is of special importance. Let, for $t \geq 0$,
\begin{equation}\label{def:Nermanmtg}
	W_t = \sum_{u \in \cC_t} \bv_{\tau(u)} e^{-\alpha S(u)},
\end{equation}
where $\bv_i$ is the $i$th entry of $\bv$, the normed right eigenvector of $\Lap\bmu(\alpha)$.
It is known that $(W_t)_{t\geq 0}$ is a positive martingale, called \emph{Nerman's martingale},
and thus converges almost surely to a limit $W$. Nerman's martingale gives the first-order asymptotics of the Crump--Mode--Jagers process. The following result in the non-lattice case may be found e.\,g.\ in \cite{Iksanov+Meiners:2015}. It can be easily checked that the same reasoning can be applied to the lattice case.
\begin{prop}
	\label{thm:asconv}
	Assume~\ref{assumpt:Malthusian alpha} through~\ref{assumpt:hreprMoments} hold and let $\varphi$ be a characteristic with c\`adl\`ag paths satisfying~\ref{assumpt:dRimean}. Then
\begin{equation*}
e^{-\alpha t} \cZ^{\varphi}_t \xrightarrow[t\in\G]{t\to\infty} \frac{W}\beta \int_\R e^{-\alpha s} \bu^\trans \E[\varphi](s) \, \ell(\dd s) \qquad \text{in }\P.
\end{equation*}
\end{prop}
Furthermore, under assumptions~\ref{assumpt:Malthusian alpha} through~\ref{assumpt:hreprMoments}, $k_\alpha = 1$ and $W(\alpha,1) = \frac1\beta W \bu^\trans$,
in particular, each entry of the random vector $W(\alpha,1)$ is a scalar multiple of $W$, 
see Remark~\ref{rem:NermanviaA} below.

\begin{rem}
In some publications, such as \cite{Iksanov+Meiners:2015},
Nerman's martingale is defined in a slightly different way. Namely, for any $i\in[p]$,
\begin{equation*}
W^{(i)}_t = \sum_{u \in \cC_t} \frac{\bv_{\tau(u)}}{\bv_i} e^{-\alpha S(u)}
\end{equation*}
constitutes a unit-mean martingale under $\P^i$.
The almost sure convergence is then stated in Theorem 2.4 of \cite{Iksanov+Meiners:2015} as follows:
For any $i\in[p]$,
\begin{equation*}
e^{-\alpha t} \cZ^{\varphi}_t
\xrightarrow[t\in \G]{t\to\infty} \frac{\bv_i W^{(i)}}\beta \int_\R e^{-\alpha s} \bu^\trans \E[\varphi](s) \, \ell(\dd s) \qquad \P^i\text{-a.\,s.,}
\end{equation*}
where $W^{(i)}$ is the $\P^i$-almost sure limit of $(W^{(i)}_t)_{t\geq0}$.
However, since this statement
can be unified by avoiding a scaling of the martingale that depends on the type of the ancestor,
we will throughout the paper consider Nerman's martingale as defined in~\eqref{def:Nermanmtg}.	
\end{rem}

\begin{rem}\label{rem:conjugates}
Observe that if $\lambda \in \Lambda$, then also $\overline{\lambda} \in \Lambda$ and $k_\lambda = k_{\overline{\lambda}}$,
where $\overline{\lambda}$ denotes complex conjugate of $\lambda$. It follows from the definition of the matrices $A_{\lambda,j}$ that $A_{\overline{\lambda},j} = \overline{A_{\lambda,j}}$ and thus $\bb_{\overline{\lambda},j,\varphi} = \overline{\bb_{\lambda,j,\varphi}}$
for a real-valued characteristic $\varphi$ such that the integral in~\eqref{def:binomialVecs} converges absolutely.
In particular, the sum in~\eqref{eq:meanAs1} is real-valued. Moreover, the martingale limits satisfy $W(\overline{\lambda},i) = \overline{W(\lambda,i)}$ almost surely for every $i\in[k_\lambda]$.
\end{rem}

\subsection{Statement of the main results}
We are now ready to state our main result. Put
\begin{align*}
H_\Lambda(t)
&= \1_{[0,\infty)}(t) \sum_{\lambda \in \intLambda} e^{\lambda t} \sum_{k=0}^{k_\lambda-1} W(\lambda,k+1) \int_\G \E[\varphi](s)\frac{(t-s)^{k}}{k!} e^{-\lambda s} \, \ell(\dd s), \\
H_\dLambda(t) 
&= \1_{[0,\infty)}(t) \sum_{\lambda \in \dLambda} e^{\lambda t} \sum_{k=0}^{k_\lambda-1} \be_{\tau(\varnothing)}^\trans A_{\lambda,k+1} \int_\G \E[\varphi](s)\frac{(t-s)^{k}}{k!} e^{-\lambda s} \, \ell(\dd s).
\end{align*}
Note that $H_\Lambda, H_\dLambda$ are real-valued by Remark~\ref{rem:conjugates}. Further, recall the vectors $\bb_{\lambda,j,\varphi}$ given by~\eqref{def:binomialVecs}. For $\lambda\in\dLambda$ and $k =0,\dots,k_\lambda-1$, let
\begin{equation*}
Y(\lambda,k) = \sum_{m=0}^{k_\lambda-k-1} \binom{k+m}{k} \int e^{-\lambda x} (-x)^m \, \bxi(\dd x) \bb_{\lambda,k+m,\varphi}.
\end{equation*}
Observe that under~\ref{assumpt:hreprMoments}, each of the vectors $\Var[Y(\lambda,k)]$ has finite entries only. Let
\begin{equation*}
	n = \max\{k: \be_i^\trans \Var[Y(\lambda,k)] \neq 0 \text{ for some } \lambda\in\dLambda, i\in[p]\}
\end{equation*}
and put $n=-1$ if such $k$ does not exist. In the case $n\geq0$, let
\begin{equation*}
	\rho_n^2 = \sum_{\lambda\in\dLambda} \bu^\trans \Var[Y(\lambda,n)].
\end{equation*}
Note that $\rho_n^2>0$ since all entries of $\bu$ are strictly positive.

\begin{thm}\label{thm:main}
Assume~\ref{assumpt:Malthusian alpha} through~\ref{assumpt:hreprMoments} hold,
and let $\varphi$ be characteristic with c\`adl\`ag paths satisfying~\ref{assumpt:dRimean} through~\ref{assumpt:local ui}
for which~\eqref{eq:meanAs1} holds. Let $\mathcal{N}$ be a standard normal variable independent of Nerman's martingale limit $W$.
	\begin{enumerate}[\normalfont(i), wide]
		\item If $n=-1$, then
		\begin{equation*}
			e^{-\frac\alpha2t} \big(\cZ_t^\varphi - H_\Lambda(t) - H_\dLambda(t) \big) \tod \sigma \Big(\frac{W}\beta\Big)^{\!1/2} \mathcal{N}
			\quad	\text{as } t \to \infty
		\end{equation*}
		where
		\begin{equation*}
			\sigma^2 = \int_\G e^{-\alpha t} \bu^\trans \Var[\bxi*f^\varphi + \varphi](t) \, \ell(\dd t)
		\end{equation*}
		for the function
		\begin{equation}\label{eq:hvarphi}
			\begin{split}
				f^\varphi(t) 
				&= r(t)\1_{[0,\infty)}(t) - \!\!\!\sum_{\lambda\in\intLambda} e^{\lambda t} \sum_{k=0}^{k_\lambda-1} A_{\lambda,k+1} \int_\G \E[\varphi](s)\frac{(t-s)^k}{k!} e^{-\lambda s} \, \ell(\dd s) \1_{(-\infty,0)}(t).
			\end{split}
		\end{equation}
		Moreover, $\sigma=0$
		if and only if $\cZ_t^{\varphi} = H_\Lambda(t) + H_\dLambda(t) + \be_{\tau(\varnothing)}^\trans r(t)$ a.\,s.\ for any $t \geq 0$.
		\item If $n\geq0$, then
		\begin{equation*}
			\frac{e^{-\frac\alpha2 t}}{\sqrt{t^{2n+1}}} \big(\cZ_t^\varphi - H_\Lambda(t) - H_\dLambda(t) \big) \tod \frac{\rho_n}{\sqrt{2n+1}} \Big(\frac{W}\beta\Big)^{\!1/2} \mathcal{N}\quad	\text{as } t \to \infty.
		\end{equation*}
	\end{enumerate}
\end{thm}

\begin{rem}
The convolution $\bxi * f^\varphi$ in the definition of $\sigma$ is to be understood in the matrix-multiplication manner,
namely, $\bxi * f^\varphi$ is the column vector with $i$th entry
\begin{equation*}
\sum_{j=1}^p \xi^{ij} * f^\varphi_j = \sum_{j=1}^p f^\varphi_j * \xi^{ij}.
\end{equation*}
Here, of course, $f^\varphi_j$ is the $j$th component of $f^\varphi$.
\end{rem}

\begin{rem}	\label{rem:sigma}
In the non-lattice setting, the formula for $\sigma$ in case $n=-1$ may be further transformed using Plancherel's theorem into
\begin{equation}
\sigma^2
= \frac{1}{2\pi} \int_{\frac\alpha2 -i\infty}^{\frac\alpha2 + i\infty} \bu^\trans \Var \big[ \Lap\varphi(z) + \Lap\bxi(z)\Lap f^\varphi(z) \big] |dz|
\end{equation}
by repeating the argument used in \cite[Remark 2.18]{Iksanov+al:2024}.
Moreover, if $\varphi$ is supported on $[0,\infty)$, $\dLambda$ is empty and all the roots in $\Lambda$ are simple, then
\begin{equation}
\Lap f^\varphi(z) = (I - \Lap\bmu(z))^{-1} \Lap(\E[\varphi])(z)
\end{equation}
for $z$ in some neighborhood of $\alpha/2 + i\R$.
\end{rem}

\begin{rem}
	Let us note that the convergence in distribution claimed in Theorem~\ref{thm:main} is in fact stable (see \cite{Aldous+Eagelson:1978} for a precise definition) and that the random variable $\mathcal{N}$ is independent of~$\F$. This can be seen by repeating the remark at the end of the proof of Theorem 6.3 in~\cite{Iksanov+al:2024}. In particular, in the case  (i) of Theorem~\ref{thm:main},
	\begin{equation*}
		\frac{1}{\sqrt{e^{\alpha t}W}} \big(\cZ_t^\varphi - H_\Lambda(t) - H_\dLambda(t) \big) \tod \sigma \beta^{-1/2} \mathcal{N}
		\quad	\text{as } t \to \infty\quad \text{ conditionally given }\Surv.
	\end{equation*}
\end{rem}

\subsection{Further organization of the paper} In the remainder of this section, we provide additional notation which will be used throughout the paper. In Section~\ref{sec:applications},
we present applications of Theorem~\ref{thm:main}.
In Section~\ref{sec:MRW}, we gather known facts concerning the Markov random walk associated with the CMJ process.
The well-definiteness of~\eqref{eq:def CMJ} is addressed in Section~\ref{sec:existence}.
In Section~\ref{sec:mean-and-mtgs}, recent results on the mean expansion and martingales are gathered, together with several new results. In Section~\ref{sec:core}, we present auxiliary results concerning convergence in distribution of CMJ processes and prove Theorem~\ref{thm:main}.

\subsection{Matrix notation}\label{sec:matrices}

Recall that for matrices $A \in \M_{n\times m}(\C)$ and $B \in \M_{k\times l} (\C)$,
the Kronecker product $A \otimes B \in \M_{nk \times ml}(\C)$ is a block matrix of the form
\begin{equation*}
	A\otimes B = \begin{pmatrix}
		a_{11} B & a_{12} B & \dots & a_{1m} B \\
		a_{21} B & a_{22} B & \dots & a_{2m} B \\
		\vdots & & & \vdots \\
		a_{n1} B & a_{n2} B & \dots & a_{nm} B
	\end{pmatrix}.
\end{equation*}
The Kronecker product is a tensor product. In particular, it is bilinear and associative.

We will sometimes consider a CMJ process counted with a matrix-valued characteristic.
For any $k,l,n\in\N$ and a characteristic $\varphi: \Omega_\varnothing \times \R \to \M_{k\times l}(\C) \otimes \M_{p\times n}(\C)$,
we define $(\cZ_t^\varphi)_{t\geq 0}$ as the $\M_{k\times l}(\C) \otimes \M_{1\times n}(\C)$-valued process given by
\begin{equation*}
\cZ_t^\varphi = \sum_{u\in\Tree} (\Ik \otimes \be_{\tau(u)}^\trans) \varphi_u(t-S(u)),
\end{equation*}
where $\Ik \in \M_{k\times k}(\C)$ is the identity matrix.
If $\varphi = \phi \otimes \psi$ for $\phi : \Omega_\varnothing \times \R \to \M_{k\times l}(\C)$ and $\psi : \Omega_\varnothing \times \R \to \M_{p\times n}(\C)$, $\cZ_t^\varphi$ is a random block matrix of the form
\begin{equation*}
	\cZ_t^\varphi = \begin{pmatrix}
		\sum_{u\in\Tree} \phi_{11}(t - S(u)) \be_{\tau(u)}^\trans \psi(t - S(u)) & \dots & \sum_{u\in\Tree} \phi_{1l}(t - S(u)) \be_{\tau(u)}^\trans \psi(t - S(u)) \\
		\vdots & & \vdots \\
		\sum_{u\in\Tree} \phi_{k1}(t - S(u)) \be_{\tau(u)}^\trans \psi(t - S(u)) & \dots & \sum_{u\in\Tree} \phi_{kl}(t - S(u)) \be_{\tau(u)}^\trans \psi(t - S(u))
		\end{pmatrix}.
\end{equation*}
That is, $\cZ^\varphi_t$ may be seen as a matrix of scalar-valued CMJ processes counted with characteristics $\phi_{ij}\psi_{\cdot h}$ for $i\in[k], j\in[l], h\in[n]$ and $\psi_{\cdot h}$ denoting $h$th column of $\psi$.

Observe that for any $i\in[k],j\in[l]$ the block $Y^{(i,j)}_t \coloneqq \sum_{u\in\Tree}\phi_{ij}(t - S(u)) \be_{\tau(u)}^\trans \psi(t - S(u))$
is a random row vector. Using the convention introduced in Section~\ref{sec:CMJ definition},
we may define $\Ev[\cZ_t^\varphi] \in \M_{k\times l}(\C) \otimes \M_{p\times n}(\C)$ as the block matrix
\begin{equation}\label{eq:mean of matrix-valued CMJ}
	\Ev[\cZ_t^\varphi] = \begin{pmatrix}
		\Ev\big[Y_t^{(1,1)}\big] & \dots & \Ev\big[Y_t^{(1,l)}\big] \\
		\vdots & & \vdots \\
		\Ev\big[Y_t^{(k,1)}\big] & \dots & \Ev\big[Y_t^{(k,l)}\big]
	\end{pmatrix}.
\end{equation}
In other words, the first $p$ rows of $\Ev[\cZ_t^\varphi]$ contain the expectations of the first row of $\cZ_t^\varphi$ under all possible types of the ancestor, the next $p$ rows of $\Ev[\cZ_t^\varphi]$ contain the expectations of the second row of $\cZ_t^\varphi$ etc. Note that when $k>1$, $\Ev[\cZ_t^\varphi]$ is \emph{not} the block matrix $(\E^1[\cZ_t^\varphi], \dots, \E^p[\cZ_t^\varphi])^\trans$.
Although this convention may seem peculiar at first, it is natural
when dealing with the martingale characteristics introduced in Section~\ref{sec:mean-and-mtgs}.

Note for later use that for a characteristic $\varphi$ as above and deterministic matrices $A \otimes \I_p \in \M_{1\times k}(\C) \otimes \M_{p \times p}(\C)$, $B \otimes C \in \M_{l\times 1}(\C) \otimes \M_{n \times 1}(\C)$, the process given by
\begin{equation*}
	(A \otimes \Ip) \cZ_t^\varphi (B \otimes C) = \cZ_t^{\varphi'}
\end{equation*}
is a standard, scalar-valued CMJ-process counted with characteristic
\begin{equation*}
	\varphi' \coloneqq (A \otimes \Ip) \varphi (B \otimes C) : \Omega_\varnothing \times \R \to  \C^p.
\end{equation*}

We will repeatedly use the Kronecker product
since it allows to simplify formulae and calculations involving the mean expansion and martingales connected to CMJ processes.
To the same end, we define the \emph{exponential matrices}. For any $\lambda\in\C$, $t\in \R$ and $k \in \N$,
let $\exp(\lambda, t, k)$ be the $k\times k$ matrix of the form
\begin{equation}\label{eq:defExp1}
	\exp(\lambda, t, k) \coloneqq e^{\lambda t}\begin{pmatrix}
		1		& t			&	\frac{t^2}{2}	& \dots		& \frac{t^{k-1}}{(k-1)!} \\
		0		& 1			& 	t				& \dots		& \frac{t^{k-2}}{(k-2)!} \\
		0		& 0			& 	1				& \ddots	& \vdots \\
		\vdots	& 			& \ddots			& \ddots	& t \\
		0 		&  			& 	\dots			& 			& 1
	\end{pmatrix},
\end{equation}
i.e., $\exp(\lambda, t,k)$ is an upper-triangular matrix with value $e^{\lambda t} \frac{t^{j}}{j!}$ on $j$th diagonal.
Equivalently,
\begin{equation}\label{eq:defExp2}
	\exp(\lambda, t, k) = \exp\big(t(\lambda \Ik + N_k)\big),
\end{equation}
where $\Ik$ is $k\times k$ identity matrix and
\begin{equation*}
	N_k \coloneqq \begin{pmatrix}
		0 		& 1 & 0			& \dots 	& 0 \\
		0 		& 0 & 1			& 		 	& 0 \\
		\vdots 	& 	& \ddots	& \ddots 	& \vdots \\
		\vdots	&	&			& \ddots	& 1 \\
		0 		& 	& \dots 	&			& 0	 
	\end{pmatrix}
\end{equation*}
is a nilpotent $k\times k$ matrix with value $1$ on the superdiagonal and $0$ everywhere else. In particular, for every $s,t\in\R$,
\begin{equation}\label{eq:expMult}
	\exp(\lambda, t, k) \exp(\lambda, s, k) = \exp(\lambda, t+s, k).
\end{equation}
Moreover, the operator norm of the exponential matrix may be estimated by its Hilbert--Schmidt norm, which leads to
\begin{equation*}
	\|  \exp(\lambda,t,k)\| \leq \|  \exp(\lambda,t,k) \|_{\mathrm{HS}} \leq C' (1+|t|)^{k-1} e^{\Re(\lambda)t} \leq C e^{\Re(\lambda)t + \delta|t|}
\end{equation*}
for any $\delta > 0$, a constant $C'$ depending on $k$ and a constant $C$ depending on $k,\delta$. More generally, for every $n\in\N$,
\begin{equation}\label{eq:expNorm}
	\begin{split}
	\| \exp(\lambda,t,k) \otimes \In \| &\leq \| \exp(\lambda,t,k) \otimes \In \|_{\mathrm{HS}} = \sqrt{n} \|\exp(\lambda,t,k) \|_{\mathrm{HS}} \\
	&\leq \sqrt{n} C'(1+|t|)^{k-1} e^{\Re(\lambda)t} \leq \sqrt{n} C e^{\Re(\lambda)t + \delta|t|}.
	\end{split}
\end{equation}

\section{Applications}\label{sec:applications}

\subsection{P\'olya urns and continuous-time Markov branching processes}\label{sec:Polya}

In the generalized P\'olya urn model,
we consider an urn containing initially some (possibly random) configuration of balls of $p \in \N$ different types.
At each step, a ball is drawn uniformly at random from the urn. The sampled ball is then either returned to the urn or removed from it. Subsequently,  $B_{ij}$ balls of type $j$ are added to the urn for each $j\in[p]$.
Here, $B_{ij} \in \N_0$ for $i,j \in [p]$.
The vectors $(B_{i1},\ldots,B_{ip})$ may also be random, sampled independently for each ball from some fixed distribution
that may, of course, depend on $i$.
For the analysis of the model it is convenient to consider its continuous-time version
obtained via the Athreya--Karlin embedding \cite{Athreya+Karlin:1968}.
We also refer to \cite{Janson:2004}, where this connection was studied in depth.
The Athreya--Karlin embedding works as follows:
Each ball is equipped with an exponential clock
and picked from the urn when the clock rings.
This gives rise to a branching process where each ball is represented by an individual. In the case where the sampling is performed with replacement, the reproduction point process is given by 
\begin{equation*}
\xi^{ij} = \sum_{n\in\N} B_{ij}(n) \delta_{T_n},	\quad	i,j \in [p]
\end{equation*}
where $\xi = \sum_{n\in\N} \delta_{T_n}$ is a homogeneous Poisson point process on $[0,\infty)$ with rate $1$,
independent of the i.i.d.\ matrices $ (B_{ij}(n))_{i,j\in[p]}$, $n\in\N$.
If in the original model balls are removed and not returned to the urn, we have
\begin{equation*}
\xi^{ij} = B_{ij}(1) \delta_{T_1},	\quad	i,j \in [p].
\end{equation*}
The with-replacement model may be viewed as a special case of the without-replacement model in which every removed ball is instantaneously replaced by a new ball of the same type.
The corresponding random variables $B_{ij}$ are then replaced by $B_{ij}+\1_{\{i=j\}}$.

Before providing specific examples, let us give some remarks on the correspondence between our setting and classical results on P\'olya urns.
Here, for simplicity, we confine ourselves to the case where each drawn ball is returned to the urn.
The reproduction intensity then is $\bmu([0,t]) = B t$ where $B = (\E[B_{ij}])_{i,j\in[p]}$ is the mean replacement matrix.
We assume $B$ to be irreducible. We have
\begin{equation*}
\Lap\bmu(z) = \frac1zB,	\quad	z \not = 0.
\end{equation*}
Observe that
\begin{equation*}
\det(\Ip - \Lap\bmu(z)) = z^{-p}\det(z\Ip - B) = z^{-p}\chi_B(z)
\end{equation*} 
where $\chi_B$ is the characteristic polynomial of the matrix $B$. As a consequence,
elements of the set $\Lambda$ are eigenvalues of $B$
and the Malthusian parameter $\alpha$ is the dominant (Perron--Frobenius) eigenvalue of $B$. The vectors $\bu^\trans,\bv$ are left and right eigenvectors of $B$ corresponding to $\alpha$ and we have
\begin{equation*}
\beta = \bu^\trans \frac1{\alpha^2} B \bv = \frac1\alpha.
\end{equation*}

Moreover, all the matrices $A_{\lambda,k}$ may be expressed in terms of generalized eigenvectors of $B$. Let $\lambda_1, \dots, \lambda_n$ be the eigenvalues of $B$ and consider the Jordan decomposition $B = P D P^{-1}$, where
\begin{equation*}
D = \begin{pmatrix}
J(\lambda_1, 1) \\
& \ddots \\
& & J(\lambda_1,m_{\lambda_1}) \\
&&& J(\lambda_2,1) \\
&&&& \ddots \\
&&&&&J(\lambda_2, m_{\lambda_2}) \\
&&&&&&\ddots
\end{pmatrix}.
\end{equation*}
Here, for each $j\in[n]$, $m_{\lambda_j}$ is the geometric multiplicity of $\lambda_j$ and the matrices $J_{\lambda_j, i}, i\in[m_{\lambda_j}]$ are the corresponding Jordan blocks with dimensions $d_{j,i}, i\in[m_{\lambda_j}]$. The columns $\bv(1), \dots, \bv(p)$ of $P$ are the generalized right eigenvectors of $B$ and the rows $\bu^\trans(1),\dots,\bu^\trans(p)$ of $P^{-1}$ are the generalized left eigenvectors. We may assume without loss of generality that $\lambda_1 = \alpha$ and $\bv(1) = \bv$. Note that, since $P^{-1}P = \Ip$, the left and right eigenvectors satisfy
\begin{equation*}
\bu(j)^\trans \bv(i) = \1_{\{i=j\}} \quad \text{for } i,j\in[p].
\end{equation*}
In particular, $\bu(1) = \bu$.

The configuration of the balls present in the urn at time $t\geq 0$ is given by $\cZ_t^\varphi$ for the matrix-valued characteristic
\begin{equation*}
\varphi(t) = \Ip \1_{[0,\infty)}(t).
\end{equation*}
Since $(\cZ_t^\varphi)_{t\geq 0}$ is a Markov process with generator $B$, we have
\begin{equation*}
m_t^\varphi = e^{tB} = P \begin{pmatrix}
e^{tJ(\lambda_1, 1)} \\
& \ddots \\
&&e^{tJ(\lambda_n, m_{\lambda_n})}
\end{pmatrix} P^{-1}.
\end{equation*}
Observe that
$$ e^{tJ(\lambda_j, i)} = \exp(\lambda_j, t, d_{j,i}). $$

On the other hand,
\begin{equation*}
(\Ip - \Lap\mu(z))^{-1} = P z (z\Ip - D)^{-1} P^{-1}
\end{equation*}
and the inverse $(z\Ip - D)^{-1}$ is a block matrix, with diagonal consisting of matrices
\begin{equation}\label{eq:Polya-inverse of Jordan}
(z \I_{d_{j,l}} - J(\lambda_j, l))^{-1} =  \begin{pmatrix}
(z-{\lambda_j})^{-1} & (z-{\lambda_j})^{-2} &\cdots & (z-{\lambda_j})^{-d_{j,l}} \\
& \ddots & \ddots \\
& &  (z-{\lambda_j})^{-1} & (z-{\lambda_j})^{-2} \\
&&& (z-{\lambda_j})^{-1}
\end{pmatrix}.
\end{equation}
for $j\in[n], l\in[m_{\lambda_j}]$. Note that this implies that for an eigenvalue $\lambda\neq 0$, $k_{\lambda}$ is the size of the largest Jordan block corresponding to $\lambda$. For $\lambda=0$ it is one less. Extending the definition of matrices $A_{\lambda,k}$ to eigenvalues lying not necessarily in the set $\Lambda$, we obtain, for $i \in [n], k\in[k_{\lambda_i}]$,
\begin{equation*}
\begin{split}
A_{\lambda_i, k} &= \frac1{(k_{\lambda_i}-k)!} \frac{\dd^{k_{\lambda_i}-k}}{\dd z^{k_{\lambda_i}-k}} \Big|_{z={\lambda_i}} \Big( (z-{\lambda_i})^{k_{\lambda_i}} \big(\Ip - \Lap\bmu(z)\big)^{-1} \Big) \\
&= \frac1{(k_{\lambda_i}-k)!} P \frac{\dd^{k_{\lambda_i}-k}}{\dd z^{k_{\lambda_i}-k}} \Big|_{z={\lambda_i}} \Big( z(z-{\lambda_i})^{k_{\lambda_i}} \big(z\Ip - D\big)^{-1} \Big) P^{-1},
\end{split}
\end{equation*}
where the derivative is applied entry-wise. Observe that, in view of \eqref{eq:Polya-inverse of Jordan}, the derivative at $z = \lambda_i$ may be non-zero only in blocks corresponding to $\lambda_i$. In other words, the matrix $A_{\lambda_i,k}$ acts essentially only on the eigenspace corresponding to $\lambda_i$, since the subspace spanned by generalized eigenvectors corresponding to other eigenvalues is contained in its kernel. Let $N_d^{(k)}$, for $0\leq k < d$, be a $d\times d$ matrix with value $1$ on $k$th diagonal and $0$ everywhere else, and a zero matrix for $k\geq d$ and $k<0$. We have
\begin{equation*}
\frac{\dd^{k_{\lambda_i}-k}}{\dd z^{k_{\lambda_i}-k}} \Big|_{z=\lambda_i} \Big( z(z-{\lambda_i})^{k_{\lambda_i}} \big(z\Ip - J({\lambda_i}, l)\big)^{-1} \Big) = (k_{\lambda_i}-k)! \Big(N^{(k)}_{d_{i,l}} + \lambda N^{(k-1)}_{d_{i,l}} \Big),
\end{equation*}
therefore
\begin{equation}\label{eq:PolyaA}
A_{\lambda_i,k} = P \begin{pmatrix}
0 \\
& \ddots \\
&& N^{(k)}_{d_{i,1}} + \lambda_i N^{(k-1)}_{d_{i,1}} \\
&&& \ddots \\
&&&& N^{(k)}_{d_{i,k_i}} + \lambda_i N^{(k-1)}_{d_{i,k_i}} \\
&&&&&\ddots \\
&&&&&&0
\end{pmatrix} P^{-1}.
\end{equation}
In particular, if for some $i\in[n]$ the algebraic and geometric multiplicity of $\lambda_i$ are equal, then $k_{\lambda_i} = 1$ and, with $E(\lambda_i) = \{v \in \C^p : Bv = \lambda_i v\}$ denoting the corresponding eigenspace,
\begin{equation*}
A_{\lambda_i,1} = \lambda_i \sum_{k : \bv(k) \in E(\lambda_i) } \bv(k) \bu(k)^\trans.
\end{equation*} 

Further, the mean $m_t^\varphi$ may be written in terms of the matrices $A_{\lambda,k}$. That is, if $0$ is not an eigenvalue of $B$,
\begin{equation}\label{eq:Polya-mean}
m_t^\varphi 
= \sum_{i=1}^n e^{\lambda_i t} \sum_{k=1}^{k_{\lambda_i}} A_{\lambda_i,k} \sum_{m=0}^{k-1} \frac{t^m}{m!} (-1)^{k-1-m} \lambda_i^{m-k}.
\end{equation}
Indeed, we have, for any $i\in[n], l\in[m_{\lambda_i}]$, a telescoping sum
\begin{equation*}
\begin{split}
e^{\lambda_i t}&\sum_{k=1}^{k_{\lambda_i}} \big( N^{(k)}_{d_{i,l}} + \lambda_i N^{(k-1)}_{d_{i,l}} \big) \sum_{m=0}^{k-1} \frac{t^m}{m!} (-1)^{k-1-m} \lambda_i^{m-k} \\
&= \sum_{m=0}^{k_{\lambda_i}-1} e^{\lambda_i t} \frac{t^m}{m!} \Big( \sum_{k=m+2}^{k_{\lambda_i} + 1} \frac{(-1)^{k-2-m}}{\lambda_i^{k-m-1}} N^{(k-1)}_{d_{i,l}} + \sum_{k=m+1}^{k_{\lambda_i}} \frac{(-1)^{k-1-m}}{\lambda_i^{k-m}} \lambda_i N^{(k-1)}_{d_{i,l}} \Big) \\
&= \sum_{m=0}^{k_{\lambda_i}-1} e^{\lambda_i t} \frac{t^m}{m!} N_{d_{i,l}}^{(m)} = \exp(\lambda_i, t, d_{i,l}) = e^{t J(\lambda_i, l)},
\end{split}
\end{equation*}
which, in view of \eqref{eq:PolyaA}, implies \eqref{eq:Polya-mean}.

If $0$ is an eigenvalue of $B$, then it gives rise to a slightly different term: the block matrix consisting of exponents of Jordan blocks corresponding to $0$ may be written as
\begin{equation*}
\sum_{k=0}^{k_0} \frac{t^k}{k!} A_{0,k},
\end{equation*}
where $A_{0,0}$ is as in \eqref{eq:PolyaA} (i.e.\ identity on the generalized eigenspace corresponding to $0$).

For positive $\lambda_i$, the coefficient in \eqref{eq:Polya-mean} equals
$$ \sum_{m=0}^{k-1} \frac{t^m}{m!} (-1)^{k-1-m} \lambda_i^{m-k} = \int_0^\infty \frac{(t-s)^{k-1}}{(k-1)!} e^{-\lambda_i s} \dd s,$$
i.e.\ this representation of the mean is consistent with \eqref{eq:meanAs1}.

Note for future use that \eqref{eq:matrixRec} in this case reads
\begin{equation}\label{eq:Polya-matrixRec}
A_{\lambda,j} = \sum_{l=j}^{k_\lambda} \frac{(-1)^{l-j}}{\lambda^{l-j+1}}BA_{\lambda,l} = \sum_{l=j}^{k_\lambda} \frac{(-1)^{l-j}}{\lambda^{l-j+1}}A_{\lambda,l}B.
\end{equation}

In order to obtain limit theorems for $\cZ_t^\varphi$, one may apply Theorem \ref{thm:main} to the vector-valued characteristics
\begin{equation}\label{eq:PolyaChar}
\varphi_{\mathbf{q}}(t) = \mathbf{q} \1_{[0,\infty)}(t) = \varphi(t) \mathbf{q}
\end{equation}
for any $\mathbf{q}\in\R^p$ and infer the limit theorem for the configuration vector $\cZ^\varphi_t$ using the Cram\'er--Wold device.
Since $\cZ_t^{\varphi_\mathbf{q}} = \cZ_t^{\varphi \mathbf{q}} = \cZ_t^\varphi \mathbf{q}$,
\begin{equation*}
m_t^{\varphi_\mathbf{q}} = m_t^{\varphi} \mathbf{q},	\quad	t \geq 0.
\end{equation*}

The discrete-time analogue of Theorem \ref{thm:main} for urn schemes was provided in
\cite{Mueller:2019} under the additional assumption that $B$ is diagonalisable and the replacement is deterministic.
In the general setting, the functional central limit theorem was proved in \cite{Janson:2004} in the case where $\intLambda$ contains only the Malthusian parameter $\alpha$. We are not aware of a comprehensive result that covers Gaussian fluctuations for the urn schemes
in the case of $\Lambda$ containing multiple roots of arbitrary multiplicity.
Theorem \ref{thm:main} provides such a result in the continuous-time setting as long as $B$ is irreducible.

\subsubsection{Balanced urns}
The urn model is called \emph{balanced} if the number of balls added at each step is constant. That is, there exists $b \in \N$ such that for every $i\in[p]$, $\sum_{j\in[p]} B_{ij} = b$ almost surely. Much attention in the study of balanced urns is directed at the distribution of the martingale limits that contribute to the asymptotic behaviour of the configuration vector when the urn is \emph{large}, i.e.\ $|\intLambda| > 1$; see, for example, \cite{Chauvin+al:2011, Mailler:2018}. In this section, we consider a simple example of a balanced urn that illustrates the use of Theorem \ref{thm:main}.

Fix $b \geq 5$ and consider a three-colour urn with deterministic replacement rule given by
\begin{equation*}
B = \begin{pmatrix}
b-2 & 1 & 1 \\
0 & b-2 & 2 \\
4 & 0 & b-4
\end{pmatrix}.
\end{equation*}
Observe that $B$ has two eigenvalues
\begin{equation*}
\alpha = b, \quad \lambda = b-4.
\end{equation*}
We have $k_\alpha = 1, k_\lambda = 2$ and, by \eqref{eq:PolyaA},
\begin{align*}
A_{\alpha,1} &= \frac{b}{4}
\begin{pmatrix}
2 & 1 & 1 \\
2 & 1 & 1 \\
2 & 1 & 1
\end{pmatrix}, \\
A_{\lambda,2} &= (b - 4)
\begin{pmatrix}
0 & 0 & 0 \\
-2 & 1 & 1 \\
2 & -1 & -1
\end{pmatrix}, \quad
A_{\lambda,1} = \frac14
\begin{pmatrix}
2b - 8 & 4-b & 4 - b \\
-2b & 3b - 8 & 8 - b \\
16 - 2b & -b & 3b-16
\end{pmatrix}.
\end{align*}
Furthermore,
\begin{equation*}
\bu = \frac34\begin{pmatrix}2 \\ 1 \\ 1 \end{pmatrix}, \quad \bv = \frac13 \begin{pmatrix} 1 \\ 1 \\ 1 \end{pmatrix}, \quad \beta = \frac1b.
\end{equation*}
In particular, Nerman's martingale is given by
\begin{equation*}
W_t = \frac13 \sum_{u\in\cC_t} e^{-\alpha S(u)}
\end{equation*}
and we have
\begin{align*}
W_t(\alpha) &= \frac{b}4 \sum_{u\in\Cc_t} e^{-\alpha S(u)} \begin{pmatrix} 2 & 1 & 1 \end{pmatrix} = b W_t \bu^\trans,\\ 
W_t(\lambda,1) &= \sum_{u\in\Cc_t} e^{-\lambda S(u)} \big(\be_{\tau(u)}^\trans A_{\lambda,1} - S(u) \be_{\tau(u)}^\trans A_{\lambda,2} \big) \\
&= \sum_{u\in\Cc_t} e^{-\lambda S(u)} \Big(\be_{\tau(u)}^\trans A_{\lambda,1} -(b-4)S(u) \big(\1_{\{\tau(u)=3\}} - \1_{\{\tau(u)=2\}} \big) \begin{pmatrix} 2 & -1 & -1 \end{pmatrix}\Big),\\ 
W_t(\lambda,2) 
&= (b-4) \sum_{u\in\Cc_t} e^{-\lambda S(u)} (\1_{\{\tau(u) = 3\}} - \1_{\{\tau(u) = 2\}}) \begin{pmatrix} 2 & -1 & -1 \end{pmatrix}.
\end{align*}
Finally, for the characteristic defined in \eqref{eq:PolyaChar}, we have
\begin{equation*}
m_t^{\varphi_\bq} = \bigg(e^{\alpha t} \frac1\alpha A_{\alpha,1} + e^{\lambda t}\Big( \frac1\lambda A_{\lambda,1} + \frac{\lambda t - 1}{\lambda^2} A_{\lambda,2} \Big)\bigg)\1_{[0,\infty)}(t) \bq.
\end{equation*}
The scaling in the central limit theorem depends on whether or not $\lambda > \frac\alpha2$, i.e.~$b > 8$.

\begin{enumerate}[\normalfont(i), wide]
	\item Small urn: $b < 8$. In this case $\Lambda = \{\alpha\}$ and $\dLambda = \varnothing$. We have
	\begin{equation*}
	H_\Lambda^\bq(t) = e^{\alpha t} W(\alpha) \int_0^\infty \bq e^{-\alpha s} \dd s = e^{\alpha t} \frac1\alpha W(\alpha) \bq = e^{b t} W \bu^\trans \bq, 
	\end{equation*}
	where $W$ is the limit of Nerman's martingale. By Theorem \ref{thm:main},
	\begin{equation*}
	e^{-\frac{b t}2} \Big( \cZ_t^{\varphi_\bq} - e^{b t} W \bu^\trans \bq \Big) \tod \sigma_\bq (b W)^{1/2} \mathcal{N},
	\end{equation*}
	where
	\begin{equation*}
	\sigma^2_\bq = \int_\R e^{-\alpha t} \bu^\trans \Var[B\xi * f^\bq](t) \dd t
	\end{equation*}
	for the function
	\begin{equation*}
	f^\bq(t) = e^{\lambda t}\Big( \frac1\lambda A_{\lambda,1} \bq + \frac{\lambda t - 1}{\lambda^2} A_{\lambda,2} \bq \Big)\1_{[0,\infty)}(t) - e^{\alpha t} \frac1\alpha A_{\alpha,1} \bq \1_{(-\infty,0)}(t).
	\end{equation*}
	Fix $t\in\R$. Observe that by \eqref{eq:Polya-matrixRec}
	\begin{equation*}
	B f^\bq(t) = e^{\lambda t} \big( A_{\lambda,1} \bq + t A_{\lambda,2} \bq \big)1_{[0,\infty)}(t) - e^{\alpha t} A_{\alpha,1} \bq \1_{(-\infty,0)}(t),
	\end{equation*}
	therefore $i$th coordinate of $B\xi*f^\bq(t)$ is given by
	\begin{equation*}
	\be_i^\trans (B\xi * f^\bq)(t) = \int_{-\infty}^t e^{\lambda(t-s)} \big(\be_i^\trans A_{\lambda,1} \bq + (t-s) \be_i^\trans A_{\lambda,2} \bq \big) \xi(\dd s) - \int_t^\infty e^{\alpha(t-s)} \be_i^\trans A_{\alpha,1} \bq \xi(\dd s).
	\end{equation*}
	Since $\xi$ is a homogeneous Poisson point process, we have
	\begin{multline*}
	\be_i^\trans\Var[B\xi * f^\bq](t) =  \frac{e^{2\alpha t}}{2\alpha} (\be_i^\trans A_{\alpha,1}\bq)^2\1_{(-\infty, 0)}(t) \\
	+ \bigg(\int_{0}^t e^{2\lambda(t-s)} (\be_i^\trans A_{\lambda,1}\bq + (t-s) \be_i^\trans A_{\lambda,2} \bq )^2 \dd s + \int_{t}^\infty e^{2\alpha(t-s)} (\be_i^\trans A_{\alpha,1} \bq)^2 \dd s\bigg)\1_{[0,\infty)}(t). 
	\end{multline*}
	Evaluating the integrals, we get
	\begin{equation*}
	\begin{split}
	\sigma_\bq^2 &=
	\frac1{\alpha^2} \sum_{i=1}^3 \bu_i (\be_i^\trans A_{\alpha,1} \bq)^2  +\frac{1}{\alpha(\alpha-2\lambda)} \sum_{i=1}^3 \bu_i (\be_i^\trans A_{\lambda,1} \bq)^2 \\
	& + \frac{2}{\alpha(\alpha-2\lambda)^2} \sum_{i=1}^3 \bu_i (\be_i^\trans A_{\lambda,1} \bq)(\be_i^\trans A_{\lambda,2} \bq)  + \frac2{\alpha(\alpha-2\lambda)^3} \sum_{i=1}^3 \bu_i (\be_i^\trans A_{\lambda,2} \bq)^2\\
	& = \frac1{\alpha} \bq^\trans \Sigma \bq
	\end{split}
	\end{equation*}
	for the matrix
	\begin{equation*}
	\begin{split}
	\Sigma &= \frac1{\alpha}  A_{\alpha,1}^\trans \mathrm{diag}(\bu) A_{\alpha,1}
	+ \frac1{\alpha-2\lambda} A_{\lambda,1}^\trans \mathrm{diag}(\bu) A_{\lambda,1} + \frac{2}{(\alpha-2\lambda)^3} A_{\lambda,2}^\trans \mathrm{diag}(\bu) A_{\lambda,2} \\
	&+ \frac{1}{(\alpha-2\lambda)^2} \Big( A_{\lambda,1}^\trans \mathrm{diag}(\bu) A_{\lambda,2} + A_{\lambda,2}^\trans \mathrm{diag}(\bu) A_{\lambda,1} \Big),
	\end{split}
	\end{equation*}
	where $\mathrm{diag}(\bu)$ is a $3\times3$ matrix with entries of $\bu$ on the diagonal and $0$ everywhere else. Since $\bq$ is arbitrary, the Cram\'er--Wold device implies
	\begin{equation*}
	e^{-\frac{b t}2} \Big( \cZ_t^\varphi - e^{b t} W \bu^\trans \Big) \tod  W^{1/2} \mathcal{N}(0,\Sigma).
	\end{equation*}
	\item Large urn: $b > 8$. In this case $\Lambda = \{\alpha,\lambda\}$ and we have
	\begin{equation*}
	\begin{split}
	H_\Lambda^\bq(t) &= e^{b t} W \bu^\trans \bq + e^{\lambda t} \bigg(\frac{W(\lambda,1) + t W(\lambda,2)}{\lambda} - \frac{W(\lambda,2)}{\lambda^2} \bigg)\bq.
	\end{split}		
	\end{equation*}
	By Theorem \ref{thm:main},
	\begin{equation*}
	e^{-\frac{b t}2} \Big( \cZ_t^{\varphi_\bq} - H_\Lambda^\bq(t) \Big) \tod \sigma_\bq (b W)^{1/2} \mathcal{N}
	\end{equation*}
	where the function $f^\bq$ is of the form
	\begin{equation*}
	f^\bq(t) = -\bigg(e^{\lambda t}\Big( \frac1\lambda A_{\lambda,1} \bq + \frac{\lambda t - 1}{\lambda^2} A_{\lambda,2} \bq \Big) + e^{\alpha t} \frac1\alpha A_{\alpha,1} \bq\bigg) \1_{(-\infty,0)}(t).
	\end{equation*}
	Proceeding as in the previous case, we conclude
	\begin{equation*}
	e^{-\frac{b t}2} \bigg( \cZ_t^{\varphi} - e^{b t} W \bu^\trans - e^{(b-4) t} \bigg(\frac{W(\lambda,1) + t W(\lambda,2)}{b-4} - \frac{W(\lambda,2)}{(b-4)^2} \bigg) \bigg) \tod  W^{1/2} \mathcal{N}(0,\Sigma)
	\end{equation*}
	for the covariance matrix
	\begin{equation*}
	\begin{split}
	\Sigma &= \frac1\alpha  A_{\alpha,1}^\trans \mathrm{diag}(\bu) A_{\alpha,1} + \frac4{\lambda(2\lambda-\alpha)} A_{\lambda,1}^\trans \mathrm{diag}(\bu) A_{\lambda,1} + \frac2{(2\lambda-\alpha)^3} A_{\lambda,2}^\trans \mathrm{diag}(\bu) A_{\lambda,2} \\
	&+ \frac1\lambda \Big(A_{\alpha,1}^\trans \mathrm{diag}(\bu) A_{\lambda,1} + A_{\lambda,1}^\trans \mathrm{diag}(\bu) A_{\alpha,1}\Big)
	- \frac1{\lambda^2} \Big(A_{\alpha,1}^\trans \mathrm{diag}(\bu) A_{\lambda,2} + A_{\lambda,2}^\trans \mathrm{diag}(\bu) A_{\alpha,1}\Big) \\
	& - \frac1{(2\lambda-\alpha)^2} \Big(A_{\lambda,1}^\trans \mathrm{diag}(\bu) A_{\lambda,2} + A_{\lambda,2}^\trans \mathrm{diag}(\bu) A_{\lambda,1}\Big).
	\end{split}
	\end{equation*}
	\item Critical case: $b = 8$. In this case $\lambda = 4$ lies on the critical line. Note that
	\begin{align*}
	\bb_{\lambda,0,\varphi_\bq} &= \frac1\lambda A_{\lambda,1} \bq - \frac1{\lambda^2} A_{\lambda,2} \bq,\\
	\bb_{\lambda,1,\varphi_\bq} &= \frac1\lambda A_{\lambda,2}\bq
	\end{align*}
	and with the help of \eqref{eq:Polya-matrixRec} we get
	\begin{align*}
	Y^\bq(\lambda,0) &= \int_\R e^{-\lambda x} B\Big(\frac1\lambda A_{\lambda,1} \bq - \frac1{\lambda^2} A_{\lambda,2} - x\frac1\lambda A_{\lambda,2} \bq \Big) \xi(\dd x) \\
	&= \int_\R e^{-\lambda x} (A_{
		\lambda,1} \bq - x A_{\lambda,2} \bq) \xi(\dd x),\\
	Y^\bq(\lambda,1) &= \int_\R e^{-\lambda x} B \frac1\lambda A_{\lambda,2} \bq \xi(\dd x) = \int_\R e^{-\lambda x} A_{\lambda,2} \bq \xi(\dd x).  
	\end{align*}
	In particular,
	\begin{align*}
	\Var[Y^\bq(\lambda,1)] &= \frac12\begin{pmatrix}
	0 \\
	1 \\
	1
	\end{pmatrix} \big( 2\bq_1 - \bq_2 - \bq_3 \big)^2,
	\end{align*}
	where $\bq_i$ denotes $i$th entry of $\bq$. That is,
	\begin{equation*}
	(\rho_1^\bq)^2 = \frac34 \big( 2\bq_1 - \bq_2 - \bq_3 \big)^2
	\end{equation*}
	is positive as long as $\bq$ is not orthogonal to the vector $\begin{pmatrix}2&-1&-1\end{pmatrix}^\trans$, which is the left eigenvector of $B$ corresponding to $\lambda$. In this case
	\begin{equation*}
	H_\dLambda^\bq(t) = e^{\lambda t} \bigg( \frac1\lambda  \be_{\tau(\varnothing)}^\trans A_{\lambda,1}\bq + \frac{\lambda t - 1}{\lambda^2} \be_{\tau(\varnothing)}^\trans A_{\lambda,2} \bq \bigg).
	\end{equation*}
	However, note that $H_{\dLambda}^\bq(t) = o(e^{\lambda t} t^{3/2})$ as $t\to\infty$ and this term in the scaling is irrelevant. By Theorem \ref{thm:main}, we have
	\begin{equation*}
	\frac{e^{-4 t}}{\sqrt{t^3}} \Big( \cZ_t^{\varphi_\bq} - e^{8 t} W \bu^\trans \bq \Big) \tod \Big(6( 2\bq_1 - \bq_2 - \bq_3 \big)^2 W\Big)^{1/2} \mathcal{N}.
	\end{equation*}
	
	If $\bq$ is orthogonal to $\begin{pmatrix}2&-1&-1\end{pmatrix}^\trans$, then $A_{\lambda,2} \bq = 0$ and thus
	\begin{equation*}
	H_{\dLambda}^\bq(t) = \frac{e^{\lambda t}}\lambda \be_{\tau(\varnothing)}^\trans A_{\lambda,1} \bq = o(e^{\lambda t}\sqrt{t}) \quad \text{as } t\to\infty
	\end{equation*}
	and
	\begin{equation*}
	\Var[Y^\bq(\lambda,0)] = \Var\Big[ \int_\R e^{-\lambda x} A_{\lambda,1}\bq \xi(\dd x)\Big] = \frac12 \begin{pmatrix} 0 \\ 1 \\ 1 \end{pmatrix} (\bq_2 - \bq_3)^2.
	\end{equation*}
	In particular, if $\bq_2 \neq \bq_3$, then
	\begin{equation*}
	(\rho_0^\bq)^2 = \frac34(\bq_2 - \bq_3)^2 > 0
	\end{equation*}
	and by Theorem \ref{thm:main},
	\begin{equation*}
	\frac{e^{-4 t}}{\sqrt{t}} \Big( \cZ_t^{\varphi_\bq} - e^{8 t} W \bu^\trans \bq \Big) \tod \Big(6(\bq_2 - \bq_3)^2 W\Big)^{1/2} \mathcal{N}.
	\end{equation*}
	
	Finally, if $\bq$ is a scalar multiple of $\begin{pmatrix} 1 & 1 & 1 \end{pmatrix}^\trans$, then $\rho_0^\bq = \rho_1^\bq = 0$ and we are again under the assumptions of part (i) of Theorem \ref{thm:main}. Note that in this case $H_\dLambda^\bq \equiv 0$. The constant $\sigma_\bq$ may be calculated exactly as in the previous cases, i.e.\
	\begin{equation*}
	\sigma_\bq^2 = \frac1{\alpha^2} \sum_{i=1}^3 \bu_i (\be_i^\trans A_{\alpha,1}\bq)^2 = \frac38\bq_1^2
	\end{equation*}
	and we have
	\begin{equation*}
	e^{-4 t} \Big( \cZ_t^{\varphi_\bq} - e^{8 t} W \bu^\trans \bq \Big) \tod \Big(3\bq_1^2W\Big)^{1/2} \mathcal{N}.
	\end{equation*}
	
	Observe that for any vector $\bq \in \R^3$ we proved the convergence
	\begin{equation*}
	\frac{e^{-4 t}}{\sqrt{t^3}} \Big( \cZ_t^{\varphi_\bq} - e^{8 t} W \bu^\trans \bq \Big) \tod \Big(6( 2\bq_1 - \bq_2 - \bq_3 \big)^2 W\Big)^{1/2} \mathcal{N},
	\end{equation*}
	with the right-hand side being trivial if $\bq$ is orthogonal to $\begin{pmatrix} 2 & -1 & -1 \end{pmatrix}^\trans$. By the Cram\'er--Wold device,
	\begin{equation*}
	\frac{e^{-4 t}}{\sqrt{t^3}} \Big( \cZ_t^{\varphi} - e^{8 t} W \bu^\trans \Big) \tod \big(6 W\big)^{1/2} \mathcal{N}\big(0,\Sigma\big)
	\end{equation*}
	for the covariance matrix
	\begin{equation*}
	\Sigma = \begin{pmatrix}
	4 & -2 & -2 \\
	-2 & 1 & 1 \\
	-2 & 1 & 1
	\end{pmatrix}.
	\end{equation*}
\end{enumerate}

\subsubsection{Cyclic urns}
Another example of a P\'olya urn with deterministic replacement rule is the cyclic urn. In this model, whenever a ball of type $j\in[p]$ is picked, it is returned together with one ball of type $j+1 \mod p$. That is, the replacement matrix is
\begin{equation*}
B = \begin{pmatrix}
0 & 1   \\
& 0 & 1 \\
& &\ddots & \ddots \\
&  &  & \ddots & 1 \\
1 &  &  &  & 0
\end{pmatrix}.
\end{equation*}
The limit theorems for the cyclic urns were studied in \cite{Mueller+Neininger:2018} in the discrete setting. Using Theorem \ref{thm:main}, we obtain continuous-time version of these results.

Since the characteristic polynomial of $B$ is $\chi_B(z) = (z-1)^p$, elements of $\Lambda$ are $p$th roots of unity. Denote
\begin{equation*}
\omega = e^{\frac{2\pi i}p}.
\end{equation*}
Then
\begin{equation*}
\Lambda = \{ \omega^n \, : \, n \in \{0, \dots, p-1\}, \, \Re(\omega^n) \geq 1/2\} = \{\omega^{\pm n} : n\in\{0,\dots,\lfloor p/6 \rfloor\}\}.
\end{equation*}
Note that all the roots are simple. We have $\alpha = \beta = 1$ and
\begin{equation*}
\bu = \begin{pmatrix} 1 \\ \vdots \\ 1 \end{pmatrix},
\quad
\bv = \frac1p \begin{pmatrix} 1 \\ \vdots \\ 1 \end{pmatrix}.
\end{equation*}

The matrices corresponding to the roots are given by
\begin{equation*}
A_{\omega^n,1} = \frac{\omega^n}p A_n \quad \text{for } A_n = (\omega^{n(i-j)})_{i,j\in[p]}.
\end{equation*}
Observe that each entry of $A_0$ equals $1$ and $\overline{A}_n = A_{-n}$ with the convention $A_{-n} = A_{p-n}$.

For each $n=0,\dots,p-1$ we have the martingale given by
\begin{equation*}
\begin{split}
W_t(\omega^n) &= \frac1p\sum_{u\in\Cc_t} \omega^n e^{-\omega^n S(u)} \be_{\tau(u)}^\trans A_n \\
&= \frac1p \sum_{u\in\Cc_t} e^{-\omega^n S(u)} \begin{pmatrix} \omega^{n\tau(u)} & \omega^{n(\tau(u)-1)} & \cdots & \omega^{n(\tau(u)-p+1)} \end{pmatrix} \\
&= \frac1p \sum_{u\in\Cc_t} e^{-\omega^n S(u)} \omega^{n\tau(u)} \begin{pmatrix} 1 & \omega^{-n} & \cdots & \omega^{-n(p-1)} \end{pmatrix}.
\end{split}
\end{equation*}
Observe that if $(W_t(\omega^n))_{t\geq 0}$ converges almost surely, then there exists a complex-valued random variable $\Xi_n$ such that
\begin{equation*}
\sum_{u\in\Cc_t} e^{-\omega^n S(u)} \omega^{n\tau(u)} \to \Xi_n \quad \text{a.s.}
\end{equation*}
and
\begin{equation*}
W(\omega^n) = \Xi_n \bw(n)^\trans
\end{equation*}
for the vector
\begin{equation*}
\bw(n)^\trans = \frac1p \begin{pmatrix} 1 & \omega^{-n} & \cdots & \omega^{-n(p-1)} \end{pmatrix}.
\end{equation*}
Moreover,
\begin{equation*}
W_t(1) = W_t \begin{pmatrix} 1 &\cdots & 1 \end{pmatrix} = p W_t \bw(0)^\trans,
\end{equation*}
where $(W_t)_{t\geq 0}$ is the Nerman's martingale. In view of the aforementioned properties of the matrices $A_n$,
\begin{equation*}
W_t(\omega^n) + W_t(\omega^{p-n}) = 2 \Re (W_t(\omega^n)) 
\end{equation*}
is a real vector-valued process which converges to $2 \Re(\Xi_n \bw(n)^\trans)$ for $\omega^n \in \intLambda$.

Fix a vector $\mathbf{q} \in \R^p$ and consider the characteristic $\varphi_\mathbf{q}$ as given in \eqref{eq:PolyaChar}. By \eqref{eq:Polya-mean},
\begin{equation*}
\begin{split}
m_t^{\varphi_\mathbf{q}} &= \sum_{n=0}^{p-1} \frac{e^{\omega^n t}}{\omega^{n}} A_{\omega^n,1} \mathbf{q} \1_{[0,\infty)}(t) = \frac1p \sum_{n=0}^{p-1} e^{\omega^n t} A_n \mathbf{q} \1_{[0,\infty)}(t).
\end{split}
\end{equation*}

The scaling limit depends on whether $\dLambda = \varnothing$, that is, whether $6 | p$.

\begin{enumerate}[\normalfont(i), wide]
	\item If $6 \nmid p$, then $\dLambda = \varnothing$. Denote $\Lambda^+ = \{\lambda \in \Lambda : \Im(\lambda) > 0\} = \{\omega^{n} : n\in\{1,\dots,\lfloor p/6 \rfloor\}\} $. We have
	\begin{equation*}
	H_\Lambda^\mathbf{q} = \sum_{\omega^n \in \Lambda} \frac{e^{\omega^n t}}{\omega^n} W(\omega^n) \mathbf{q} = \sum_{\omega^n \in \Lambda^+} 2  \Re\Big( e^{\omega^n t} \omega^{-n} \Xi_n \bw(n)^\trans \Big) \mathbf{q} + p W \bw(0)^\trans \mathbf{q}
	\end{equation*}
	and
	\begin{equation*}
	e^{-\frac{t}2} \big( \cZ_t^{\varphi_\mathbf{q}} - H_\Lambda^{\mathbf{q}} \big) \tod \sigma_{\mathbf{q}} W^{1/2} \mathcal{N},
	\end{equation*}
	where
	\begin{equation*}
	\sigma_\mathbf{q}^2 = \int_\R e^{-t} \bu^\trans \Var[B \xi * f^\mathbf{q}](t) \dd t
	\end{equation*}
	for the function
	\begin{equation*}
	\begin{split}
	f^\mathbf{q}(t) &= m^{\varphi_\mathbf{q}}_t - \sum_{\omega^n \in \Lambda} e^{\omega^n t} \frac{A_{\omega^n,1}}{\omega^n} \mathbf{q} \\
	&= \frac1p  \sum_{n=\lfloor p/6 \rfloor + 1}^{ p - \lfloor p/6 \rfloor - 1}  e^{\omega^n t} A_{n} \mathbf{q}  \1_{[0,\infty)}(t)
	- \frac1p \sum_{n = -\lfloor p/6 \rfloor }^{\lfloor p/6 \rfloor} e^{\omega^n t} A_{n} \mathbf{q}  \1_{(-\infty,0)}(t).
	\end{split}
	\end{equation*}
	Recall that by \eqref{eq:Polya-matrixRec}, for any $n = 0,\dots,p-1$ we have
	$B A_{n} = \omega^n A_{n}$. Moreover, for any $t\in\R$, the random measures $\xi((-\infty,t))$ and $\xi([t,\infty))$ are independent, therefore
	\begin{equation*}
	\begin{split}
	\Var[B\xi * f^\mathbf{q}](t) &= \Var\Big[\int_0^t \frac1p \sum_{n=\lfloor p/6 \rfloor + 1}^{ p - \lfloor p/6 \rfloor - 1}  e^{\omega^n (t-s)} \omega^n A_{n} \mathbf{q} \xi(\dd s) \Big] \\
	&+ \Var \Big[ \int_t^\infty \frac1p \sum_{n = -\lfloor p/6 \rfloor }^{\lfloor p/6 \rfloor} e^{\omega^n (t-s)} \omega^n A_{n}\mathbf{q}  \xi(\dd s) \Big].
	\end{split}
	\end{equation*}
	Observe that for any $j\in[p]$, $\omega^n \be_j^\trans A_n = p \omega^{nj} \bw(n)^\trans$, therefore
	\begin{equation*}
	\begin{split}
	\be_j^\trans \Var[B\xi * f^\mathbf{q}](t) &= \Var\Big[\int_0^t \frac1p \sum_{n=\lfloor p/6 \rfloor + 1}^{ p - \lfloor p/6 \rfloor - 1} e^{\omega^n (t-s)} \omega^n \be_j^\trans A_{n} \mathbf{q} \xi(\dd s) \Big] \\
	&+ \Var \Big[ \int_t^\infty \frac1p \sum_{n = -\lfloor p/6 \rfloor }^{\lfloor p/6 \rfloor} e^{\omega^n (t-s)} \omega^n \be_j^\trans A_{n}\mathbf{q}  \xi(\dd s) \Big] \\
	&= \int_0^t \Big|\sum_{n=\lfloor p/6 \rfloor + 1}^{ p - \lfloor p/6 \rfloor - 1} e^{\omega^n (t-s)} \omega^{nj} \bw(n)^\trans \mathbf{q}\Big|^2 \dd s \\
	&+\int_{t\vee 0}^\infty \Big|\sum_{n = -\lfloor p/6 \rfloor }^{\lfloor p/6 \rfloor} e^{\omega^n (t-s)} \omega^{nj} \bw(n)^\trans \mathbf{q}\Big|^2 \dd s
	\end{split}
	\end{equation*}
	thus
	\begin{equation*}
	\begin{split}
	\sigma_\mathbf{q}^2 &= \int_\R e^{-t} \sum_{j=1}^p \be_j^\trans \Var[B\xi * f^\mathbf{q}](t) \dd t \\
	&=  \int_0^\infty \int_s^\infty e^{-t} \sum_{j=1}^p \Big|\sum_{n=\lfloor p/6 \rfloor + 1}^{ p - \lfloor p/6 \rfloor - 1} e^{\omega^n (t-s)} \omega^{nj} \bw(n)^\trans \mathbf{q}\Big|^2 \dd t \dd s \\
	& + \int_0^\infty \int_{-\infty}^s e^{-t} \sum_{j=1}^p \Big|\sum_{n = -\lfloor p/6 \rfloor }^{\lfloor p/6 \rfloor} e^{\omega^n (t-s)} \omega^{nj} \bw(n)^\trans \mathbf{q}\Big|^2 \dd t \dd s.
	\end{split}
	\end{equation*}
	The roots of unity are orthogonal in the sense that for any $n,k=0,\dots,p-1$,
	\begin{equation*}
	\sum_{j=1}^p \omega^{nj} \overline{\omega}^{kj} = p \1_{\{k=n\}}.
	\end{equation*}
	In particular, when expanding the squares in the above formula, all the cross-terms vanish and thus
	\begin{equation*}
	\begin{split}
	\sigma_\mathbf{q}^2 &= p \int_0^\infty \int_s^\infty e^{-t} \sum_{n=\lfloor p/6 \rfloor + 1}^{ p - \lfloor p/6 \rfloor - 1} e^{2\Re(\omega^n) (t-s)} |\bw(n)^\trans \mathbf{q}|^2 \dd t \dd s \\
	& + p \int_0^\infty \int_{-\infty}^s e^{-t} \sum_{n = -\lfloor p/6 \rfloor }^{\lfloor p/6 \rfloor} e^{2\Re(\omega^n) (t-s)} |\bw(n)^\trans \mathbf{q}|^2 \dd t \dd s \\
	& = p \sum_{n=0}^{ p-1}  \frac{|\bw(n)^\trans \mathbf{q}|^2}{|1 - 2\Re(\omega^n)|},
	\end{split}
	\end{equation*}
	which gives
	\begin{multline*}
	e^{-\frac{t}2} \Big( \cZ_t^{\varphi} - p W \bw(0)^\trans - \sum_{\omega^n \in \Lambda^+} 2  \Re\big( e^{\omega^n t} \omega^{-n} \Xi_n \bw(n)^\trans \big) \Big) \mathbf{q}\\ \tod (pW)^{1/2} \mathcal{N}\Big(0, \sum_{n=0}^{ p-1}  \frac{|\bw(n)^\trans \mathbf{q}|^2}{|1 - 2\Re(\omega^n)|}  \Big).
	\end{multline*}
	Since $\mathbf{q}$ is arbitrary, this implies
	\begin{equation*}
	e^{-\frac{t}2} \Big( \cZ_t^{\varphi} - p W \bw(0)^\trans - \sum_{\omega^n \in \Lambda^+} 2  \Re\big( e^{\omega^n t} \omega^{-n} \Xi_n \bw(n)^\trans \big) \Big) \tod (pW)^{1/2} \mathcal{N}\big(0, \Sigma \big)
	\end{equation*}
	for the covariance matrix
	\begin{equation*}
	\Sigma = \sum_{n=0}^{p-1} \frac1{|1 - 2\Re(\omega^n)|} \bw(n) \bw(n)^*.
	\end{equation*}
	
	\item If $6|p$, then $\dLambda = \{\omega^r, \omega^{-r}\}$, where $r = p/6$. Denote $\Lambda^+ = \{\lambda \in \intLambda : \Im(\lambda) > 0\} = \{\omega^{n} : n\in\{1,\dots, r -1\}\} $. We have
	\begin{align*}
	H_\Lambda^\mathbf{q} &= \sum_{\omega^n \in \intLambda} \frac{e^{\omega^n t}}{\omega^n} W(\omega^n) \mathbf{q} = \sum_{\omega^n \in \Lambda^+} 2  \Re\Big( e^{\omega^n t} \omega^{-n} \Xi_n \bw(n)^\trans \Big) \mathbf{q} + p W \bw(0)^\trans \mathbf{q}, \\
	H_\dLambda^\mathbf{q}(t) &= \frac1p \Big( e^{\omega^r t} \be_{\tau(\varnothing)}A_r  + e^{\omega^{-r} t} \be_{\tau(\varnothing)}A_{p-r} \Big) \mathbf{q}
	= 2\Re\big( e^{\omega^r t} \omega^{r(\tau(\varnothing)-1)}\bw(r)^\trans \big)\mathbf{q}.
	\end{align*}
	The vectors $Y(\omega^r,0), Y(\omega^{-r},0)$ are of the form
	\begin{align*}
	Y(\omega^r,0)
	&=  \Lap\bxi(\omega^r) \frac{\omega^r}p A_r \mathbf{q}, \\
	Y(\omega^{-r},0) &= \Lap\bxi(\omega^{-r}) \frac{\omega^{-r}}p A_{-r} \mathbf{q}.
	\end{align*}
	Observe that $\Var[\Lap\xi(\omega^r)] = (2\Re(\omega^r))^{-1} = 1$. That is, for any $j\in[p]$,
	\begin{equation*}
	\be_j^\trans \Var\Big[Y(\omega^r,0) \Big] = \Big|\frac{\omega^r}p \be_j^\trans B  A_r \mathbf{q} \Big|^2 \Var[\Lap\xi(\omega^r)] = \big|\bw(r)^\trans \mathbf{q}\big|^2
	\end{equation*}
	and similarly
	\begin{equation*}
	\be_j^\trans \Var\big[Y(\omega^{-r},0) \big] = \big|\bw(p-r)^\trans \mathbf{q}\big|^2.
	\end{equation*}
	
	Note that if $\mathbf{q}$ is orthogonal to $\bw(r)$, then the above variances vanish. However, in this case $H_\dLambda^\mathbf{q} \equiv 0$ and we are back in the setting of case (i). That is,
	\begin{equation*}
	\frac{e^{-\frac{t}2}}{\sqrt{t}}\big(\cZ_t^{\varphi_\mathbf{q}} - H_\Lambda^\mathbf{q}\big) \tod 0.
	\end{equation*}
	If $\mathbf{q}$ is not orthogonal to $\bw(r)$, then $n=0$ and
	\begin{equation*}
	\rho_0^2 = p\Big(\big|\bw(r)^\trans \mathbf{q}\big|^2 + \big|\bw(p-r)^\trans \mathbf{q}\big|^2\Big) \neq 0.
	\end{equation*}
	Note that $(t e^{t})^{-1/2}|H_\dLambda^\mathbf{q}(t)| \to 0$ as $t\to\infty$, therefore
	\begin{equation*}
	\frac{e^{-\frac{t}2}}{\sqrt{t}}\big(\cZ_t^{\varphi_\mathbf{q}} - H_\Lambda^\mathbf{q}\big) \tod (pW)^{1/2} \mathcal{N}\Big(0, \big|\bw(r)^\trans \mathbf{q}\big|^2 + \big|\bw(p-r)^\trans \mathbf{q}\big|^2 \Big).
	\end{equation*}
	
	By the Cram\'er--Wold device,
	\begin{equation*}
	\frac{e^{-\frac{t}2}}{\sqrt{t}} \Big( \cZ_t^{\varphi} - p W \bw(0)^\trans - \sum_{\omega^n \in \Lambda^+} 2  \Re\big( e^{\omega^n t} \omega^{-n} \Xi_n \bw(n)^\trans \big) \Big) \tod (pW)^{1/2} \mathcal{N}\big(0, \Sigma \big)
	\end{equation*}
	for the covariance matrix
	\begin{equation*}
	\Sigma = \bw(r) \bw(r)^* + \bw(p-r) \bw(p-r)^*.
	\end{equation*}
	
\end{enumerate}

\subsection{The elephant random walk}\label{sec:ERW}

The one-dimensional, nearest-neighbour elephant random walk was introduced in \cite{Schutz+Trimper:2004} in order to study the influence of long-range memory effects on asymptotic properties of random walks. The model was later generalised, both by considering different step distribution \cite{Businger:2018, Kiss+Veto:2022}, non-uniform memory \cite{Laulin:2022}, and a walk on higher dimensional lattice \cite{Bercu+Laulin:2019}. In what follows, we consider a multi-dimensional elephant random walk (MERW) as introduced in \cite{Bercu+Laulin:2019}.

Fix $d \in \N$ and consider an elephant performing a nearest-neighbour random walk on $\Z^d$. It starts at the origin and makes the first step according to some fixed distribution $\nu$. From there on, the elephant's moves depend on its past. At $n$th step, it samples uniformly one of its previous $n-1$ steps and moves in the same direction with probability $p$ or in one of the remaining directions, uniformly with probability $(1-p)/(2d-1)$ for some parameter $p\in(0,1)$. Observe that this process may be modelled by a P\'olya urn: the urn is the elephant's memory, balls of type $i \in [2d]$ are its steps in direction $i$, i.e.\ along $\be_k$ if $i=2k$ and along $-\be_k$ if $i=2k-1$. The urn contains initially one ball, whose type is chosen according to $\nu$. When a ball is picked from the urn, it is put back together with a ball of the same type with probability $p$ or one of the other types, each with probability $(1-p)/(2d-1)$. The continuous-time P\'olya urn corresponds to the elephant taking steps at exponential rate.

Let $Y_t$ be the position of the elephant at time $t \geq 0$. With the help of Theorem \ref{thm:main} one obtains a central limit theorem for $(Y_t)_{t\geq 0}$. In the discrete-time setting, such a theorem was proved in the case $d=1$ in \cite{Coletti+al:2017, Kubota+Takei:2019}. For $d > 1$, a functional central limit theorem was obtained in \cite{Bertenghi:2022} under the diffusive and critical regime. The Gaussian fluctuations in the superdiffusive regime were studied in \cite{Bercu:2025} for a lazy walk.

We proceed with defining an appropriate CMJ model. Let $\xi = \sum_{n\in\N} \delta_{T_n}$ be a homogeneous Poisson point process on $[0,\infty)$ with intensity $1$ and let $(X_n)_{n\in\N}$ be a family of i.i.d.\ random variables independent of $\xi$ and such that
$$\P[X_1 = 0] = p, \qquad \P[X_1 = i] = \frac{1-p}{2d-1}, \quad i=1,\dots,2d-1.$$
Further, let $\epsilon_{ij}(n) = \1_{\{X_n = i - j \mod 2d\}}$ for $i,j\in [2d]$ and define
\begin{equation*}
\bxi = \bigg( \sum_{n\in\N} \epsilon_{ij}(n)\delta_{T_n}\bigg)_{i,j \in [2d]}.
\end{equation*}
Observe that for every $i \in [2d]$,
$$(\xi^{ij})_{j\in[2d]} = \bigg( \sum_{n\in\N} \epsilon_{ij}(n)\delta_{T_n}\bigg)_{j\in [2d]}$$
is a coloured Poisson point process (c.f.\ \cite[Chapter 5]{Kingman:1993}). In particular, $\xi^{ij}$ is a homogeneous Poisson point process on $[0,\infty)$ with intensity $p$ if $i=j$ and $(1-p)/(2d-1)$ otherwise, and the processes $\xi^{ij}, j\in[2d]$ are independent.

The position of the elephant at time $t\geq 0$ is given by
$$ Y_t = ( \cZ_t^{\varphi_1}, \cdots, \cZ_t^{\varphi_d})$$
for characteristics
$$ \varphi_k(t) = (\be_{2k} - \be_{2k-1})\1_{[0,\infty)}(t), \quad k\in[d]. $$
Therefore, in order to obtain a central limit theorem for the vector $Y_t$ as $t\to\infty$, we may use Theorem \ref{thm:main} applied to characteristic
\begin{equation}\label{eq:elephant-char}
\varphi_\bq(t) = \sum_{k=1}^d \bq_k (\be_{2k} - \be_{2k-1}) \1_{[0,\infty)}(t)
\end{equation}
for an arbitrary vector $\bq = (\bq_k)_{k\in[d]} \in \R^{d}$, together with the Cram\'er--Wold device.

The intensity measure of $\bxi$ is given by $\bmu[0,t] = Bt$, where
\begin{equation*}
B = \begin{pmatrix}
p & \frac{1-p}{2d-1} &  \cdots & \frac{1-p}{2d-1} \\
\frac{1-p}{2d-1} & p &  \ddots & \vdots \\
\vdots & \ddots & \ddots & \frac{1-p}{2d-1}  \\
\frac{1-p}{2d-1}  & \cdots & \frac{1-p}{2d-1} & p
\end{pmatrix}
\end{equation*}
is a $2d \times 2d$ matrix with value $p$ on the diagonal and $(1-p)/(2d-1)$ everywhere else. Since
\begin{equation*}
\Lap\bmu(z) = \frac1z B,
\end{equation*}
a direct calculation shows that
\begin{equation*}
(I - \Lap\bmu(z))^{-1} = \frac{z}{\big(z-1\big)\big(z - \frac{2pd-1}{2d-1} \big)}
\begin{pmatrix}
z - \frac{p+2d -2}{2d-1} & \frac{1-p}{2d-1} & \cdots & \frac{1-p}{2d-1} \\
\frac{1-p}{2d-1} & z - \frac{p+2d -2}{2d-1} & \ddots &  \vdots \\
\vdots & \ddots & \ddots & \frac{1-p}{2d-1} \\
\frac{1-p}{2d-1}  & \cdots & \frac{1-p}{2d-1} & z - \frac{p+2d -2}{2d-1}
\end{pmatrix}.
\end{equation*}
In particular, regardless of the dimension $d$, there are only two roots
\begin{equation*}
\alpha \coloneqq 1, \quad \lambda \coloneqq \frac{2pd-1}{2d-1},
\end{equation*}
each of multiplicity one. The corresponding matrices are
\begin{equation*}
A_{\alpha,1} = \frac1{2d} \begin{pmatrix} 1 & \cdots & 1 \\ \vdots & & \vdots \\ 1 & \cdots & 1 \end{pmatrix}, \quad
A_{\lambda,1} = \frac{2pd - 1}{2d} \begin{pmatrix}
1 & -\frac1{2d-1} & \cdots & -\frac1{2d-1} \\
-\frac1{2d-1} & 1 & \ddots & \vdots \\
\vdots & \ddots & \ddots & -\frac1{2d-1} \\
-\frac1{2d-1} & \cdots & -\frac1{2d-1} & 1
\end{pmatrix},
\end{equation*}
thus the martingales take the form
\begin{align*}
W_t(\alpha) &= \frac1{2d}\sum_{u\in\Cc_t} e^{-S(u)} \begin{pmatrix} 1 & \cdots & 1 \end{pmatrix}, \\
W_t(\lambda) &= \frac{2pd-1}{2d} \sum_{u\in\Cc_t} e^{-\lambda S(u)} \bw(\tau(u))^\trans,
\end{align*}
where for $i\in[2d]$, $\bw(i) \in \R^{2d}$ is a vector with value $1$ on $i$th coordinate and $-1/(2d-1)$ everywhere else. Observe that if $(W_t(\lambda))_{t\geq 0}$ converges almost surely to $W(\lambda)$, then for every $k \in [d]$,
\begin{equation*}
\sum_{u\in\Cc_t} e^{-\lambda S(u)} (\1_{\{\tau(u) = 2k\}} - \1_{\{\tau(u)=2k-1\}}) = \frac{2d-1}{2d} \sum_{u\in\Cc_t} e^{-\lambda S(u)} \bw(\tau(u))^\trans (\be_{2k} - \be_{2k-1})
\end{equation*}
converges almost surely as $t\to\infty$ to a random variable
\begin{equation}\label{eq:elephant-lims}
\Xi_k \coloneqq \frac{2d-1}{2pd-1} W(\lambda) (\be_{2k} - \be_{2k-1}).
\end{equation}

The properly scaled left and right eigenvectors $\bu^\trans,\, \bv$ of $\Lap\bmu(\alpha)$ are
\begin{equation*}
\bu = \begin{pmatrix} 1 \\ \vdots \\ 1 \end{pmatrix}, \quad \bv = \frac1{2d} \begin{pmatrix} 1 \\ \vdots \\ 1 \end{pmatrix}.
\end{equation*}
In particular, $\beta = 1$ and the Nerman's martingale is given by
\begin{equation*}
W_t = \frac1{2d} \sum_{u\in\Cc_t} e^{-S(u)},
\end{equation*}
so that $W_t(\alpha) = \begin{pmatrix} 1 & \cdots & 1 \end{pmatrix} W_t $. Finally, the mean of the process counted with characteristic $\varphi_\bq$ as in \eqref{eq:elephant-char} may be calculated explicitly with the help of \eqref{eq:Polya-mean}. We have
\begin{equation*}
m_t^{\varphi_\bq} = e^{\lambda t} \sum_{k=1}^d \bq_k(\be_{2k} - \be_{2k-1}).
\end{equation*}

The shape of the limit theorem depends on whether $\lambda > \frac\alpha2$, i.e., whether $p > \frac{2d+1}{4d}$.
\begin{enumerate}[\normalfont(i), wide]
	\item The diffusive regime: $p < \frac{2d+1}{4d}$. In this case $\Lambda = \{\alpha\}$ and $\dLambda = \varnothing$. Observe that for any vector $\bq \in\R^{d}$,
	\begin{equation*}
	H_\Lambda^\bq(t) = e^t W(\alpha) \int_\R e^{-s} \varphi_\bq(s) \dd s = e^t W \sum_{k=1}^d \bq_k \begin{pmatrix} 1 & \cdots & 1 \end{pmatrix} (\be_{2k} - \be_{2k-1}) \equiv 0.
	\end{equation*}
	By Theorem \ref{thm:main},
	\begin{equation*}
	e^{-\frac{t}2} \cZ_t^{\varphi_\bq} \tod \sigma_\bq W^{1/2} \mathcal{N},
	\end{equation*}
	where
	\begin{equation*}
	\sigma_\bq^2 = \int_\R e^{-t} \bu^\trans \Var[\bxi * f^\bq](t) \dd t
	\end{equation*}
	for the function $ f^\bq(t) = e^{\lambda t} \sum_{k=1}^d \bq_k(\be_{2k} - \be_{2k-1})\1_{[0,\infty)}(t)$.
	
	Observe that for any $i\in[2d]$ and  $t \geq 0$, $i$th entry of the vector $\Var[\bxi*f^\bq](t)$ equals
	\begin{equation*}
	\be_i^\trans \Var[\bxi*f^\bq](t) = \Var\Big[ \sum_{k=1}^d \bq_k (\xi^{i,2k} * g(t) - \xi^{i,2k-1} * g(t)) \Big],
	\end{equation*}
	where $g(t) = e^{\lambda t}$. Using the fact that $(\xi^{ij})_{j\in[2d]}$ is a coloured Poisson point process, we obtain
	\begin{equation*}
	\begin{split}
	\be_i^\trans \Var[\bxi*f^\bq](t) &= e^{2\lambda t} \sum_{k=1}^d \bq_k^2 \bigg(\Var \Big[ \int_0^t e^{-\lambda s} \xi^{i,2k}(\dd s)\Big] + \Var \Big[\int_0^t e^{-\lambda s} \xi^{i,2k-1}(\dd s)\Big] \bigg) \\
	&= \frac{e^{2\lambda t} - 1}{2\lambda} \sum_{k=1}^d \bq_k^2 \bigg( \frac{2(1-p)}{2d-1} \1_{\{i \notin \{2k-1,2k\}\}} + \Big(p + \frac{1-p}{2d-1}\Big) \1_{\{i \in \{2k-1,2k\}\}} \bigg).
	\end{split}
	\end{equation*}
	Therefore,
	\begin{equation*}
	\begin{split}
	\sigma^2_\bq &= \int_0^\infty \frac{e^{(2\lambda-1)t} - e^{-t}}{2\lambda} \dd t \sum_{k=1}^d \bq_k^2 \Big((2d-2) \frac{2(1-p)}{2d-1} + 2\Big(p + \frac{1-p}{2d-1} \Big) \Big) \\
	&= \frac{2(2d-1)}{2d-4pd+1} \sum_{k=1}^d \bq_k^2.
	\end{split}
	\end{equation*}
	Since $\bq$ is arbitrary, this implies
	\begin{equation*}
	e^{-\frac{t}2} Y_t \tod W^{1/2} \mathcal{N}\Big(0, \frac{2(2d-1)}{2d-4pd+1} \Id \Big).
	\end{equation*}
	\item The superdiffusive regime: $p> \frac{2d + 1}{4d}$. In this case $\Lambda = \{\alpha,\lambda\}$ and
	\begin{equation*}
	H_\Lambda^\bq(t) = e^{\lambda t} \frac1{\lambda} W(\lambda) \sum_{k=1}^d \bq_k(\be_{2k} - \be_{2k-1}) = e^{\lambda t} \sum_{k=1}^d \bq_k \Xi_k
	\end{equation*}
	for $\Xi_k$ given by \eqref{eq:elephant-lims}. Theorem \ref{thm:main} implies
	\begin{equation*}
	e^{-\frac{t}2} \Big( \cZ_t^{\varphi_\bq} - e^{\lambda t} \sum_{k=1}^d \bq_k \Xi_k \Big) \tod \sigma_\bq W^{1/2} \mathcal{N}.
	\end{equation*}
	In this case
	\begin{equation*}
	f^\bq(t) = -e^{\lambda t} \sum_{k=1}^d \bq_k(\be_{2k} - \be_{2k-1})\1_{(-\infty,0)}(t)
	\end{equation*}
	and a calculation analogous to that in (i) gives
	\begin{equation*}
	\sigma_\bq^2 = \frac{2(2d-1)}{4pd - 2d-1} \sum_{k=1}^d \bq_k^2.
	\end{equation*}
	By the Cram\'er--Wold device,
	\begin{equation*}
	e^{-\frac{t}2} \big(Y_t - e^{\lambda t}\Xi\big) \tod W^{1/2} \mathcal{N}\Big(0, \frac{2(2d-1)}{4pd - 2d-1} \Id \Big)
	\end{equation*}
	for the vector $\Xi = (\Xi_j)_{j\in[d]}$.
	\item The critical regime: $p = \frac{2d + 1}{4d}$. In this case $\lambda = \frac12$ lies on the critical line. We have $H_\Lambda^\bq(t) \equiv 0$ and
	\begin{equation*}
	H_\dLambda^\bq(t) = e^{\frac{t}2} \frac{2pd - 1}{2d} \sum_{k=1}^d \bq_k \bw(\tau(\varnothing))^\trans (\be_{2k} - \be_{2k-1}) = o\big(e^{\frac{t}2}\sqrt{t}\big).
	\end{equation*}
	The vector $Y(\lambda,0)$ is of the form
	\begin{equation*}
	Y(\lambda,0) = \int e^{-\lambda x} \bxi(\dd x) \sum_{k=1}^d \bq_k(\be_{2k} - \be_{2k-1}).
	\end{equation*}
	That is, its $i$th coordinate, $i\in[2d]$, is given by
	\begin{equation*}
	\be_i^\trans Y(\lambda,0) = \sum_{k=1}^d \bq_k \Big( \int e^{-\lambda x} \xi^{i,2k}(\dd x) - \int e^{-\lambda x} \xi^{i,2k-1}(\dd x) \Big),
	\end{equation*}
	therefore, since $(\xi^{ij})_{j\in[2d]}$ is a coloured Poisson point process,
	\begin{equation*}
	\be_i^\trans \Var[Y(\lambda,0)] = \sum_{k=1}^d \frac{\bq_k^2}{2\lambda} \bigg( \frac{2(1-p)}{2d-1} \1_{\{i \notin \{2k-1,2k\}\}} + \Big(p + \frac{1-p}{2d-1}\Big) \1_{\{i \in \{2k-1,2k\}\}} \bigg),
	\end{equation*}
	which gives
	\begin{equation*}
	\rho_0^2 = 2 \sum_{k=1}^d \bq_k^2.
	\end{equation*}
	In particular, $\rho_0^2 > 0$ as long as $\bq$ is a non-zero vector. Theorem \ref{thm:main} then implies
	\begin{equation*}
	\frac{e^{-\frac{t}2}}{\sqrt{t}} \cZ_t^{\varphi_\bq} \tod \Big( 2 \sum_{k=1}^d \bq_k^2 W\Big)^{1/2} \mathcal{N},
	\end{equation*}
	which leads to
	\begin{equation*}
	\frac{e^{-\frac{t}2}}{\sqrt{t}} Y_t \tod W^{1/2} \mathcal{N}\big(0, 2 \Id\big).
	\end{equation*}
\end{enumerate}

\subsection{m-ary search trees}\label{sec:m-aryST} The following branching process is used to model space requirement in searching and sorting algorithms in computer science. For $m \geq 3$, consider a $m-1$-type branching process with the following reproduction rule: a particle of type $j$ lives for an exponential time with rate $j$, and upon death produces one particle of type $j+1$ if $j<m-1$ and $m$ particles of type $1$ if $j=m-1$. In the Markov theory this process is usually described by a P\'olya urn model in which the rate at which the ball is drawn depends on its colour and the ball is not returned to the urn after it was drawn. In particular, the configuration of the balls forms a Markov process with the infinitesimal generator
\begin{equation*}
M = \begin{pmatrix}
-1 & 1 & \\
& -2 & 2 \\
&&\ddots & \ddots \\
&&& -(m-2) & m-2 \\
m(m-1) &&&& -(m-1) 
\end{pmatrix}.
\end{equation*}

Our approach, however, is to give a description in terms of a CMJ process with a characteristic that takes into account the death of particles. That is, we consider the reproduction process
\begin{equation*}
\bxi = \begin{pmatrix}
0 & \delta_{E_1}  \\
& 0 & \delta_{E_2} \\
&& \ddots & \ddots \\
&&& 0 & \delta_{E_{m-2}} \\
m\delta_{E_{m-1}} &&&& 0
\end{pmatrix},
\end{equation*}
where $E_j \sim \mathcal{E}xp(j)$ is a random variable with density $x\mapsto je^{-jx}\1_{(0,\infty)}(x)$. The configuration of particles at time $t$ is given by $\cZ_t^\varphi$ with the random characteristic
\begin{equation*}
\varphi(t) = \mathrm{diag}(\1_{[0,E_1)}(t), \1_{[0,E_2)}(t), \dots, \1_{[0,E_{m-1})}(t)).
\end{equation*}
As in the previous examples, the limit theorem for $\cZ^\varphi$ follows from Theorem \ref{thm:main} applied to the projection characteristics
\begin{equation}\label{eq:maryChar}
\varphi_\bq(t) = (\bq_j \1_{[0,E_j)}(t))_{j\in[m-1]},
\end{equation}
where $\bq = (\bq_j)_{j\in[m-1]} \in \R^{m-1}$.

Observe that
\begin{equation*}
\Lap\bmu(z) = \begin{pmatrix}
0 & \frac1{z+1} \\
& 0 & \frac2{z+2} \\
& & \ddots \\
\frac{m(m-1)}{z+m-1} &&& 0
\end{pmatrix}.
\end{equation*}
In particular,
\begin{equation*}
\det(\I_{m-1} - \Lap\bmu(z)) = \frac{\prod_{j=1}^{m-1} (z+j) - m!}{\prod_{j=1}^{m-1} (z+j)}.
\end{equation*}
That is, $\Lambda$ consists of roots of the polynomial
\begin{equation*}
w_m(z) = \prod_{j=1}^{m-1} (z+j) - m!,
\end{equation*}
which is the characteristic polynomial of $M$. It is known \cite{Chern+Hwang:2001}, that all the roots of $w_m$ are simple and $1$ is the root with largest real value, i.e.\ $\alpha = 1$. Moreover, the only other real root is $-m-1$ if $m$ is odd; if $m$ is even, all the other roots are non-real. Most importantly, there exist roots different from $1$ with real part larger than $1/2$ if and only if $m \geq 27$. Let $\Lambda_w$ denote the set of all roots of $w_m$, so that $\Lambda = \{\lambda \in \Lambda_w : \Re\lambda \geq 1/2\}$.

\begin{figure}[h!]\label{fig:m-ary}
	\centering
	\begin{subfigure}{.46\textwidth}
		\centering
		\includegraphics[width=\linewidth]{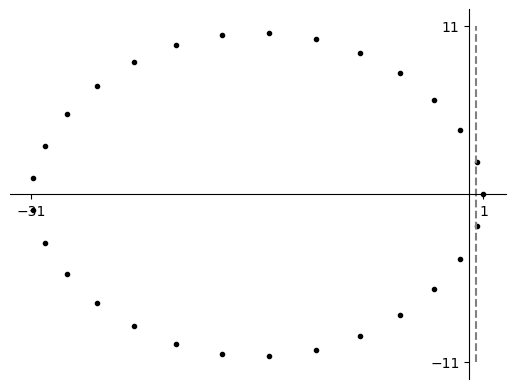}
		\phantomcaption{\footnotesize $m=30$: $3$ roots behind the critical line.}
	\end{subfigure}
	\begin{subfigure}{.46\textwidth}
		\centering
		\includegraphics[width=\linewidth]{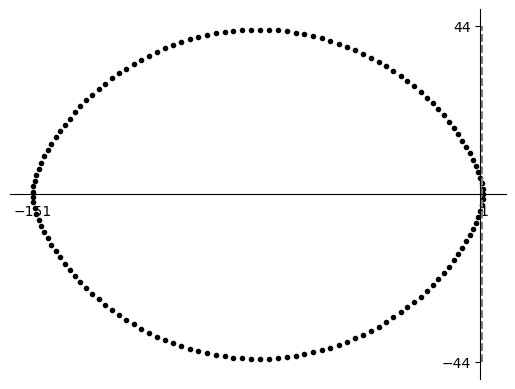}
		\phantomcaption{\footnotesize $m=150$: $5$ roots behind the critical line.}
	\end{subfigure}
	\caption{Numeric approximation of the roots of $w_m$. The dashed line is the critical line $\Re(\lambda) = 1/2$.}
\end{figure}

Observe that the right and left eigenvectors $\bv,\, \bu^\trans$ of $\Lap\bmu(1)$ are 
\begin{equation*}
\bv = \frac2{m(m-1)} \begin{pmatrix} 1 \\ 2 \\ \vdots \\ m-1 \end{pmatrix}, \quad  \bu = \frac{m}2 \begin{pmatrix} 1 \\ \frac12 \\ \vdots \\ \frac1{m-1} \end{pmatrix}.
\end{equation*}
In particular,
\begin{equation*}
\beta = \frac1{m-1} \sum_{k=2}^m \frac1k.
\end{equation*}

For any $\lambda \in \Lambda$, we have $k_\lambda = 1$.
Instead of analysing the exact form of $A_{\lambda,1}$ by expanding the function $z \mapsto (\I_{m-1} - \Lap\bmu(z))^{-1}$ around its poles, let us note that in this case we may extend Remark \ref{rem:NermanviaA} to other roots as well. Namely, recall that by \eqref{eq:matrixRec}, columns of $A_{\lambda,1}$ are right eigenvectors of $\Lap\bmu(\lambda)$ corresponding to the eigenvalue $1$ and the rows are left eigenvectors. However, the characteristic polynomial of $\Lap\bmu(\lambda)$ may be easily computed to be
\begin{equation*}
x \mapsto (-x)^{m-1} + (-1)^m \frac{m!}{(\lambda+1)\dots(\lambda + m-1)} = (-1)^{m-1} (x^{m-1} - 1),
\end{equation*}
where we have used the fact that $\lambda$ is a root of $w_m$. In particular, $\Lap\bmu(\lambda)$ has $m-1$ distinct eigenvalues, those being the roots of unity, and each of its eigenspaces is one-dimensional. Denote by $\bu^\trans(\lambda), \bv(\lambda)$ the left and right eigenvectors of $\Lap\bmu(\lambda)$ corresponding to the eigenvalue $1$. Note that we may choose the vectors such that $\bv(1) = \bv, \ \bu(1) = \bu$, $\bu(\lambda)^\trans \bv(\lambda) = 1$ for all roots $\lambda$ and $\bv(\overline{\lambda}) = \overline{\bv(\lambda)}$, $\bu(\overline{\lambda}) = \overline{\bu(\lambda)}$ for non-real $\lambda$. Then each column of $A_{\lambda,1}$ must be a scalar multiple of $\bv(\lambda)$ and each row a scalar multiple of $\bu(\lambda)$, i.e.\
\begin{equation*}
A_{\lambda,1} = \frac1{\beta_\lambda}\bv(\lambda) \bu(\lambda)^\trans
\end{equation*}
for some $\beta_\lambda\neq 0$. The eigenvectors may be written explicitly: for some non-zero constants $c_\lambda$,
\begin{equation*}
\bv(\lambda) = c_\lambda \bigg(\frac{\prod_{k=1}^{j-1} (\lambda + k)}{(j-1)!}\bigg)_{j\in[m-1]}, \quad \bu(\lambda) = \frac1{c_\lambda (m-1)} \bigg(\frac{(j-1)!}{\prod_{k=1}^{j-1} (\lambda + k)}\bigg)_{j\in[m-1]}. 
\end{equation*}
Proceeding as in Remark \ref{rem:NermanviaA}, we conclude that
\begin{equation*}
\beta_\lambda = \bu(\lambda)^\trans (-(\Lap\bmu)'(\lambda))\bv(\lambda) 
= \frac1{m-1}\sum_{k=1}^{m-1}\frac1{\lambda+k}.
\end{equation*}
Finally, it may be checked by direct calculation that $\bv(\lambda)$ is a right eigenvector of $M$ corresponding to the eigenvalue $\lambda$. Note that this construction of $A_{\lambda,1}$ may be performed for all roots of the polynomial $w_m$, not necessarily lying in $\Lambda$.

The martingale associated to $\lambda \in \Lambda$ is of the form
\begin{equation*}
W_t(\lambda) = \sum_{u\in\Cc_t} e^{-\lambda S(u)} \be_{\tau(u)}^\trans A_{\lambda,1} = \frac1{\beta_\lambda}\sum_{u\in\Cc_t} e^{-\lambda S(u)} \be_{\tau(u)}^\trans \bv(\lambda) \bu(\lambda)^\trans.
\end{equation*}
If it converges a.s., then
\begin{equation*}
\sum_{u\in\Cc_t} e^{-\lambda S(u)} \be_{\tau(u)}^\trans \bv(\lambda) \xrightarrow{t\to\infty} \Xi(\lambda)
\end{equation*}
almost surely for a complex-valued random variable $\Xi(\lambda)$ such that $W(\lambda) = \beta_\lambda^{-1}\Xi(\lambda) \bu(\lambda)^\trans$. In particular, $\Xi(\overline{\lambda}) = \overline{\Xi(\lambda)}$. Note that $\Xi(1)$ is the limit of Nerman's martingale.

Since the process is Markovian, the mean of the configuration vector may be calculated explicitly by exponentiating the generator $M$. This gives
\begin{equation*}
\Ev[\cZ_t^\varphi] = \sum_{\lambda\in\Lambda_w} \frac{e^{\lambda t}}{\beta_\lambda}  \bv(\lambda) \bw(\lambda)^\trans,
\end{equation*}
where, for each root $\lambda \in \Lambda_w$,
\begin{equation*}
\bw(\lambda) = \frac1{c_\lambda(m-1)}\bigg(\frac{(j-1)!}{\prod_{k=1}^j(\lambda+k)}\bigg)_{j\in[m-1]}
\end{equation*}
is the left eigenvector of $M$ corresponding to $\lambda$ such that $\beta_\lambda^{-1}\bw(\lambda)^\trans \bv(\lambda) = 1$.

Recall that we consider the characteristic given by \eqref{eq:maryChar}. The above formula gives
\begin{equation*}
\Ev[\cZ_t^{\varphi_\bq}] = \sum_{\lambda\in\Lambda_w} \frac{e^{\lambda t}}{\beta_\lambda} \bv(\lambda) \bw(\lambda)^\trans \bq,
\end{equation*}
which is consistent with \eqref{eq:meanAsymptotics} since $\E[\varphi_\bq](t) = (\bq_j e^{-jt})_{j=1}^{m-1}\1_{[0,\infty)}(t)$ and
\begin{equation*}
\begin{split}
\sum_{\lambda\in\Lambda} e^{\lambda t} A_{\lambda,1} \int_\R e^{-\lambda s} \E[\varphi_\bq](s) \dd s &= \sum_{\lambda\in\Lambda} \frac{e^{\lambda t}}{\beta_\lambda} \bv(\lambda) \bu(\lambda)^\trans \bigg( \frac{\bq_j}{\lambda+j} \bigg)_{j\in[m-1]} \\
&= \sum_{\lambda\in\Lambda} \frac{e^{\lambda t}}{\beta_\lambda} \bv(\lambda) \bw(\lambda)^\trans \bq.
\end{split}
\end{equation*}

As in the previous examples, the scaling depends on whether there are any roots of $w_m$ on the critical line $\Re(z) = 1/2$. In any case we have
\begin{equation*}
\begin{split}
H_\Lambda^\bq(t) &=
\sum_{\lambda\in\intLambda} \frac{e^{\lambda t}}{\beta_\lambda} \Xi(\lambda) \bw(\lambda)^\trans \mathbf{q} \\
&= \bigg(\frac{e^t}\beta W \bw(1)^\trans + \sum_{\lambda \in \Lambda^+} 2\Re\bigg(\frac{e^{\lambda t}}{\beta_\lambda} \Xi(\lambda) \bw(\lambda)^\trans\bigg) \bigg) \mathbf{q},
\end{split}
\end{equation*}
where $\Lambda^+ = \{\lambda \in \Lambda: \Re(\lambda) \in (1/2,1), \Im(\lambda) > 0\}$.

If $\dLambda = \varnothing$, then Theorem \ref{thm:main}
and the Cram\'er--Wold device imply
\begin{equation*}
e^{-\frac{t}2} \Big(\cZ_t^\varphi - e^t W \bw(1)^\trans - \sum_{\lambda \in \Lambda^+} 2\Re\big(e^{\lambda t} \Xi(\lambda) \bw(\lambda)^\trans\big)\Big) \tod \bigg( \frac{W}\beta \bigg)^{1/2} \mathcal N,
\end{equation*}
where $\mathcal N$ is a multidimensional Gaussian variable independent of $W$.

If $\dLambda$ is non-empty, that is, consists of two roots $\lambda_b, \overline{\lambda}_b$, then we have the additional term
\begin{equation*}
H_\dLambda^\mathbf{q}(t)
= 2\Re \Big( e^{\lambda_b t} \be_{\tau(\varnothing)}^\trans \bv(\lambda_b) \bw(\lambda_b)^\trans \Big) \mathbf{q}.
\end{equation*}
However, $H_\dLambda^\mathbf{q}(t) = o((te^{t})^{1/2})$, therefore proceeding as in the previous examples we conclude that
\begin{equation*}
\frac{e^{-\frac{t}2}}{\sqrt{t}} \Big(\cZ_t^\varphi - e^t W \bw(1)^\trans - \sum_{\lambda \in \Lambda^+} 2\Re\big(e^{\lambda t} \Xi(\lambda) \bw(\lambda)^\trans\big)\Big)\tod \bigg(\frac{W}\beta\bigg)^{1/2} \mathcal N,
\end{equation*}
where $\mathcal N$ is a multidimensional Gaussian variable independent of $W$.

Since we know the exact form of $m^\varphi$, the covariance matrix of $\mathcal{N}$ in both cases may be calculated as in the previous examples and expressed in terms of the vectors $\bv(\lambda), \bw(\lambda)$. We refrain from giving further details.

\subsection{Multi-type Galton--Watson process}

The $p$-type Galton--Watson process , $p\in\N$, is a special case of a lattice CMJ process with reproduction given by
\begin{equation*}
\bxi = (B^{ij}\delta_1)_{i,j \in [p]}
\end{equation*}
for some family of random variables $(B^{ij})_{i,j\in[p]}$ taking values in $\N_0$. Limit theorems for finite-type Galton--Watson processes counted with random characteristics supported on the positive half-line were considered in \cite{Kolesko+Sava-Huss:2023}. In this section, we briefly discuss how our results may be applied in this setting.

The Laplace transform of the intensity measure $\bmu$ is given by
\begin{equation*}
\Lap\bmu(z) = B e^{-z}
\end{equation*}
for the matrix $B = (\E[B^{ij}])_{i,j\in[p]}$. Since $\det(\Ip - \Lap\bmu(z)) = e^{-pz}\det(e^z \Ip - B)$, for each $\lambda \in \Lambda$, $e^\lambda$ is an eigenvalue of $B$. In particular, the Malthusian parameter is a positive $\alpha$ such that $e^\alpha$ is the Perron--Frobenius eigenvalue of $B$. Denote by $\rho_1,\dots,\rho_n$, $n\leq p$, the eigenvalues of $B$ and by $m_{\rho_1},\dots,m_{\rho_n}$ their geometric multiplicities. We have
\begin{equation*}
\Lambda = \{\lambda \in [\alpha/2,\alpha] + \imag(-\pi,\pi]: e^\lambda\in\{\rho_1,\dots,\rho_n\} \}.
\end{equation*}

Similarly to the continuous-time Pólya urn model discussed in Section \ref{sec:Polya}, for any $\lambda\in\Lambda$ the matrices $A_{\lambda,k}$ are closely related to generalized eigenvectors of $B$ corresponding to the eigenvalue $e^\lambda$. In what follows, we use the notation introduced in Section \ref{sec:Polya}. That is, we consider the Jordan decomposition
$B = PDP^{-1}$ for $D$ being the Jordan matrix with Jordan blocks $J(\rho_1,1),\dots,J(\rho_n,m_{\rho_n})$. The matrix of $(\Ip - \Lap\bmu(z))^{-1} = (\Ip - e^{-z}B)^{-1}$ written in the eigenbasis of columns of $P$ is a diagonal block matrix, with diagonal consisting of blocks
\begin{equation*}
(\I_{d_{j,l}} - e^{-z} J(\rho_j,l))^{-1}
= e^{z}
\begin{pmatrix}
(e^z-{\rho_j})^{-1} & (e^z-{\rho_j})^{-2} &\cdots & (e^z-{\rho_j})^{-d_{j,l}} \\
& \ddots & \ddots \\
& &  (e^z-{\rho_j})^{-1} & (e^z-{\rho_j})^{-2} \\
&&& (e^z-{\rho_j})^{-1}
\end{pmatrix}.
\end{equation*}
As a consequence, for $\lambda_i = \log(\rho_i) \in \Lambda$, the multiplicity $k_{\lambda_i}$ is the size of the largest Jordan block corresponding to $\rho_j$. Furthermore, for $k\in[k_{\lambda_i}]$ we have $A_{\lambda_i,k} = P A_{\lambda_i,k}' P^{-1}$, where $A_{\lambda_i,k}'$ is a block diagonal matrix; the $l$th block corresponding to the eigenvalue $\rho_i$ is of the form
\begin{equation*}
(k_\lambda - k)! \begin{pmatrix}
g_{i,1}^{(k_\lambda - k)}(\lambda_i) & g_{i,2}^{(k_\lambda - k)}(\lambda_i) & \cdots & g_{i,d_{i,l}}^{(k_\lambda - k)}(\lambda_i) \\
& \ddots & \ddots \\
& & & g_{i,1}^{(k_\lambda - k)}(\lambda_i),
\end{pmatrix}
\end{equation*}
where $g_{i,j}(z) = e^z (z-\lambda_i)^{k_{\lambda_i}}(e^z - \rho_i)^{-j}$. The entries of $A_{\lambda_i,k}'$ in the blocks corresponding to other eigenvalues are zero. 

Note that the expected configuration of particles in the Galton--Watson process present at time $n\in\N_0$ is $B^n$. In particular, if a characteristic $\varphi$ is supported on the positive half-line, the mean of the process counted with $\varphi$ at time $n$ is exactly
\begin{equation*}
\Ev[\cZ_n^\varphi] = \sum_{k=0}^n B^{k} \E[\varphi](n-k).
\end{equation*}
That is, when applying Theorem \ref{thm:main} (i), the variance of the limiting Gaussian variable may be calculated explicitly using formula \eqref{eq:hvarphi}.

\subsection{Supercritical binary CMJ process}
In this section we consider a multi-type version of the binary supercritical CMJ process. Each particle reproduces at times given by arrivals of a homogeneous Poisson process with rate $b > 0$. Upon reproduction, it produces one child whose type is chosen randomly. Moreover, each particle has a finite lifetime given by a copy of a variable $\zeta$. That is,
\begin{equation*}
\bxi = \bigg(\sum_{n\in\N} B_{ij}(n)\1_{\{T_n < \zeta\}} \delta_{T_n}\bigg)_{i,j\in[p]}
\end{equation*}
where $\xi \coloneqq \sum_{n\in\N} \delta_{T_n}$ is a homogeneous, rate $b$ Poisson process and $(B_{ij}(n))_{i,j\in[p]}$ are i.i.d.\ copies of a random matrix $(B_{ij})_{i,j\in[p]}$ whose each row, almost surely, has exactly one entry equal to $1$ and all other entries equal to $0$. We assume that $\xi$, the matrix $(B_{ij})_{i,j\in[p]}$ and $\zeta$ are independent. In particular, for any $i\in[p]$, $(\xi^{ij})_{j\in[p]}$ is a coloured Poisson point process, i.\,e.\ a vector of independent Poisson processes, restricted to a random interval $[0,\zeta)$. Note that the branching process with such reproduction is not, in general, Markovian.

Our goal is to examine the asymptotic behaviour of the configuration of particles, that is, the process $(\cZ_t^\varphi)_{t\geq 0}$ with the characteristic
\begin{equation*}
\varphi(t) = \Ip \1_{[0,\zeta)}(t).
\end{equation*}
To this end, we may consider a characteristic $\varphi_\bq(t) = \varphi(t) \bq$ for an arbitrary vector $\bq \in \R^p$ and use Theorem \ref{thm:main} together with the Cram\'er--Wold device.

We have
\begin{equation*}
\bmu(\dd x) = \E[\1_{[0 ,\zeta)}(x) b \dd x] B,
\end{equation*}
where $B = (\E[B_{ij}])_{i,j\in[p]}$, and the Laplace transform is given by
\begin{equation*}
\Lap\bmu(z) = \E\Big[\int_0^\zeta e^{-zx} b \dd x \Big] B = \frac{1 - \Lap_\zeta(z)}{z} b B,
\end{equation*}
where $\Lap_\zeta$ is the Laplace transform of the random variable $\zeta$. We assume that $B$ is irreducible. Observe that, since each row of $B$ sums to $1$, its Perron--Frobenius eigenvalue is $1$ and the corresponding right eigenspace is spanned by a vector with all entries equal to $1$. The Malthusian parameter is the unique $\alpha > 0$ that satisfies
\begin{equation*}
\frac{1 - \Lap_\zeta(\alpha)}{\alpha} = \frac1b.
\end{equation*}
Note that the vectors $\bu^\trans, \bv$ are left and right eigenvectors of $B$ corresponding to the eigenvalue $1$. In particular, $\bv = (p^{-1})_{i\in[p]}$ and the Nerman's martingale is
\begin{equation*}
W_t = \frac1p\sum_{u\in\cC_t} e^{-\alpha S(u)}.
\end{equation*}
The parameter $\beta$ is given by
\begin{equation*}
\beta = \frac{\alpha\Lap'_\zeta(\alpha) + 1 - \Lap_\zeta(\alpha)}{\alpha^2} b \bu^\trans B \bv = \frac{b}{\alpha} \bigg(\Lap'_\zeta(\alpha) + \frac1{b}\bigg) = \frac{b\Lap'_\zeta(\alpha) + 1}{\alpha }.
\end{equation*}

In general, if $\gamma_1, \dots, \gamma_p$ are eigenvalues of $B$ (possibly with repetitions), then $\Lambda$ consists of roots of equations
\begin{equation}\label{eq:killedPoisson}
\frac{1 - \Lap_\zeta(z)}z = \frac1{b \gamma_k} 
\end{equation}
for $k\in[p]$ that satisfy $\Re(z) \geq \alpha/2$. It is worth noting that even when $\lambda \in \Lambda$ is a simple root to \eqref{eq:killedPoisson}, then its multiplicity $k_\lambda$ may be larger than $1$ if the corresponding eigenvalue is defective. That is, the asymptotic behaviour of the process depends strongly on the spectral properties of the matrix $B$. In contrast, in a single-type setting the Malthusian parameter is always the only element of $\Lambda$, as was shown in \cite[Section 3.4]{Iksanov+al:2024}. As a consequence, regardless of the specific choice of $\zeta$ the single-type process has Gaussian fluctuations around its a.\,s.\ limit and no terms of smaller order appear in the scaling. In the multi-type setting the interplay between the law of $\zeta$ and eigenvalues of $B$ becomes relevant. Observe that the function
\begin{equation*}
f(z) = \frac{1-\Lap_\zeta(z)}z = \int_0^\infty e^{-zx} \P[\zeta \geq x] \dd x
\end{equation*}
satisfies $f(0) = \E[\zeta]$, $|f(z)| \leq f(\Re(z))$ for all $z \in [0,\infty) + \imag \R$ and is strictly decreasing on $[0,\infty)$. It was also observed in \cite{Iksanov+al:2024} that when $\Re(z) > 0$, then $f(z) \in \R$ if and only if $z\in\R$. As a consequence, if $\E[\zeta]$ is finite and satisfies $b\E[\zeta] \in (1,r^{-1})$, where $r = \max\{|\gamma_k| : k\in[p], \gamma_k\neq 1\}$, then equations \eqref{eq:killedPoisson} have no roots apart from $\alpha$. If, on the other hand, $\E[\zeta] = \infty$ and $B$ has several real positive eigenvalues, then for every $\gamma_k \in (0,\infty)$ the equation \eqref{eq:killedPoisson} has exactly one solution, which is real and simple.

Let us briefly describe how Theorem \ref{thm:main} applies when $B$ is diagonalizable and all the roots $\lambda\in\Lambda$ satisfy $k_\lambda = 1$. This holds, for example, when $B$ has real eigenvalues only. Write
\begin{equation*}
B = P\mathrm{diag}(\gamma_1,\dots,\gamma_p) P^{-1}
\end{equation*}
with $\gamma_1 = 1$ and $P$ whose columns are the corresponding eigenvectors, which we denote by $\bv(1), \dots, \bv(p)$. Without loss of generality, we assume that $\bv = \bv(1)$ and that $\bv(k) = \overline{\bv(j)}$ if $\gamma_k = \overline{\gamma_j}$ (note that, since $B$ has real entries only, its non-real eigenvalues come in conjugate pairs). Further, let $\bu(k)^\trans$ be the $k$th row of $P^{-1}$. Then $\bu(k)^\trans$ is a left eigenvector of $B$ corresponding to $\gamma_k$ and, since $P^{-1}P = \Ip$, we have
\begin{equation}\label{eq:uv orthogonality}
\bu(j)^\trans \bv(i) = \1_{\{i=j\}} \quad \text{for } i,j\in[p].
\end{equation}
In particular, $\bu = \bu(1)$.

For $\lambda\in\Lambda$ denote by $\gamma_\lambda$ the corresponding eigenvalue. We have
\begin{equation*}
(\Ip - \Lap\bmu(z))^{-1} = P \mathrm{diag} \bigg(\frac{z}{z - \gamma_1b(1 - \Lap_\zeta(z))}, \dots, \frac{z}{z - \gamma_pb(1 - \Lap_\zeta(z))}\bigg) P^{-1},
\end{equation*}
therefore
\begin{equation*}
A_{\lambda, 1} = \frac{\lambda}{1 + \gamma_\lambda b\Lap_\zeta'(\lambda)} P\mathrm{diag}(\1_{\{\gamma_1 = \gamma_\lambda\}}, \dots, \1_{\{\gamma_p = \gamma_\lambda\}}) P^{-1} = \frac1{\beta_\lambda} \sum_{j: \gamma_j = \gamma_\lambda} \bv(j) \bu(j)^\trans,
\end{equation*}
where $\beta_\lambda = (1 + \gamma_\lambda b\Lap'_\zeta(\lambda))/\lambda$. The corresponding martingale is given by
\begin{equation*}
W_t(\lambda) = \frac1{\beta_\lambda} \sum_{u \in \Cc_t} \sum_{j : \gamma_j=\gamma_\lambda} e^{-\lambda S(u)}\be_{\tau(u)}^\trans \bv(j) \bu(j)^\trans.
\end{equation*}
Denote
\begin{equation*}
\Xi_j(t) = \sum_{u\in\Cc_t} e^{-\lambda S(u)} \be_{\tau(u)}^\trans \bv(j)
\end{equation*}
and observe that, due to \eqref{eq:uv orthogonality}, $\Xi_j(t) = \beta_j W_t(\lambda) \bv(j)$. In particular, if $(W_t(\lambda))_{t\geq 0}$ converges, then so does $(\Xi_j(t))_{t\geq 0}$ for any $j\in[p]$ such that $\gamma_j = \gamma_\lambda$. We denote its limit by $\Xi_j$ and note that
\begin{equation*}
W(\lambda) = \frac1{\beta_\lambda}\sum_{j:\gamma_j = \gamma_\lambda} \Xi_j \bu(j)^\trans.
\end{equation*}

Recall that we consider a random characteristic $\varphi_\bq(t) = \bq \1_{[0,\zeta)}(t)$. For any $\lambda\in\Lambda$,
\begin{equation*}
\int e^{-\lambda s}\E[\varphi_\bq](s) \dd s = \E\Big[\int_0^\zeta e^{-\lambda s} \dd s\Big] \bq = \frac{1 - \Lap_\zeta(\lambda)}{\lambda} \bq = \frac1{b\gamma_\lambda} \bq,
\end{equation*}
therefore
\begin{align*}
H_\Lambda^\bq(t) &= \1_{[0,\infty)}(t) \sum_{\lambda\in\intLambda} \frac{e^{\lambda t}}{b\beta_\lambda\gamma_\lambda} \sum_{j:\gamma_j=\gamma_\lambda} \Xi_j \bu(j)^\trans \bq, \\
H_\dLambda^\bq(t) &= \1_{[0,\infty)}(t) \sum_{\lambda\in\dLambda} \frac{e^{\lambda t}}{b\beta_\lambda\gamma_\lambda} \sum_{j:\gamma_j=\gamma_\lambda} (\be_{\tau(\varnothing)}^\trans \bv(j)) \bu(j)^\trans \bq. 
\end{align*}

Assume first that $\dLambda \neq \varnothing$ and denote $I = \{j\in[p] : \gamma_j = \gamma_\lambda \text{ for some } \lambda\in\dLambda\}$. For each $\lambda\in\dLambda$, we have the vector
\begin{equation*}
\begin{split}
Y(\lambda,0) &= \int e^{-\lambda x} \bxi(\dd x) \int e^{-\lambda s} A_{\lambda,1} \E\varphi_\bq(s) \dd s 
\\
&= \int e^{-\lambda x} \bxi(\dd x) \frac1{\gamma_\lambda\beta_\lambda} \sum_{j:\gamma_j = \gamma_\lambda} \bv(j) \bu(j)^\trans\bq.
\end{split}
\end{equation*}
In particular, $\rho_0^\bq = 0$ if $\bq$ is orthogonal to the space spanned by $\{\bu(j) : j\in I\}$. Note that, in view of \eqref{eq:uv orthogonality} and the fact that $\{\bv(j): j\in[p]\}$ is a basis of $\C^p$, this is equivalent to $\bq \in \mathrm{lin}\{\bv(j) : j \notin I\}$. In any case, Theorem \ref{thm:main} gives
\begin{equation*}
\frac{e^{-\frac\alpha2t}}{\sqrt{t}} \Big( \cZ_t^{\varphi_\bq} - H_\Lambda^\bq(t) - H_\dLambda^\bq(t)\Big) \tod \rho_0^\bq \Big(\frac{W}{\beta}\Big)^{1/2}  \mathcal{N},
\end{equation*}
with the limit being degenerate when $\bq \in \mathrm{lin}\{\bv(j) : j\notin I\}$. Note that $H_\dLambda(t) = o(\sqrt{t}e^{\frac\alpha2t})$ as $t\to\infty$, therefore this term may be skipped. With the Cram\'er--Wold device, we conclude that
\begin{equation*}
\frac{e^{-\frac\alpha2t}}{\sqrt{t}} \Big( \cZ_t^{\varphi} - \sum_{\lambda\in\intLambda} \frac{e^{\lambda t}}{b\beta_\lambda\gamma_\lambda} \sum_{j:\gamma_j=\gamma_\lambda} \Xi_j \bu(j)^\trans\Big) \tod \Big(\frac{W}{\beta}\Big)^{1/2} \mathcal N
\end{equation*}
for a centered Gaussian vector $\mathcal N$ whose covariance matrix may be determined by calculating $\rho_0^\bq$ as in the previous examples.

If $\dLambda = \varnothing$, then Theorem \ref{thm:main} together with the Cram\'er--Wold device give
\begin{equation*}
e^{-\frac\alpha2t} \Big( \cZ_t^{\varphi} - \sum_{\lambda\in\intLambda} \frac{e^{\lambda t}}{b\beta_\lambda\gamma_\lambda} \sum_{j:\gamma_j=\gamma_\lambda} \Xi_j \bu(j)^\trans\Big) \tod \Big(\frac{W}{\beta}\Big)^{1/2} \mathcal N
\end{equation*}
for a centered Gaussian vector $\mathcal N$ whose covariance matrix may be determined with the help of Remark \ref{rem:sigma}.

\subsection{Cell cycle model} In this section, we consider a model used in
\cite[Chapter 6.4]{Kimmel+Axelrod:2015} 
to describe the effect of chemotherapy on cancer cells. The authors distinguished $18$ phases which each cell undergoes during its lifetime. The last phase is mitosis, after which the cell splits into two new cells. The newborn cells then enter phase $1$ and undergo the same cycle. Based on experimental data, the authors claim that most of the randomness in the duration of cell cycle comes from phase $1$, which is the ``gap time'' between the cell birth and the beginning of DNA replication, and thus the durations of phases $2-18$ may be without loss assumed to be deterministic. The duration of phase $1$ is assumed to have a continuous distribution. Furthermore, each cell is constantly exposed to a drug whose effectiveness depends on the cycle phase. That is, the cell in phase $i$ has probability $p_i \in (0,1)$ to proceed to the next phase and probability $1-p_i$ to get blocked by the drug and stay in phase $i$ forever.

In general, the described model is a multi-type Bellman--Harris process with reproduction
\begin{equation*}
\bxi = \begin{pmatrix}
0 & X_1 \delta_{T_1} \\
& 0 & X_2 \delta_{t_2} \\
& & \ddots & \ddots \\
&&& 0 & X_{m-1} \delta_{t_{m-1}} \\
2 X_m \delta_{t_m} &&&& 0
\end{pmatrix}
\end{equation*}
for some $m\in \N$, deterministic times $t_2,\dots t_m \in (0,\infty)$, a random, continuously distributed $T_1$ and random variables $X_1,\dots,X_m$ such that $\P[X_i = 1] = p_i = 1- \P[X_i = 0]$. We assume for simplicity that $X_1$ and $T_1$ are independent. For a concrete example, let $T_1$ be gamma-distributed with density $f(x) = (\Gamma(\kappa) \theta^\kappa)^{-1} x^{\kappa - 1} e^{-x/\theta}$ for some parameters $\kappa,\theta > 0$. The Laplace transform of the intensity measure is then
\begin{equation*}
\Lap\bmu(z) = \begin{pmatrix}
0 & p_1 (1+\theta z)^{-\kappa} \\
& 0 & p_2 e^{-z t_2} \\
& & \ddots & \ddots \\
&&& 0 & p_{m-1} e^{-z t_{m-1}} \\
2 p_m e^{-z t_m} &&&& 0
\end{pmatrix}
\end{equation*}
for $\Re(z) > -1/\theta$. In particular,
$$\det(\I_m - \Lap\bmu(z)) = 1 - c_p (1+\theta z)^{-\kappa} e^{-z c_t},$$
where
$$ c_p = 2 \prod_{i=1}^m p_i, \quad c_t = \sum_{i=2}^m t_i.$$
\begin{figure}[h!]
	\centering
	\begin{subfigure}{.45\textwidth}
		\centering
		\includegraphics[width=\linewidth]{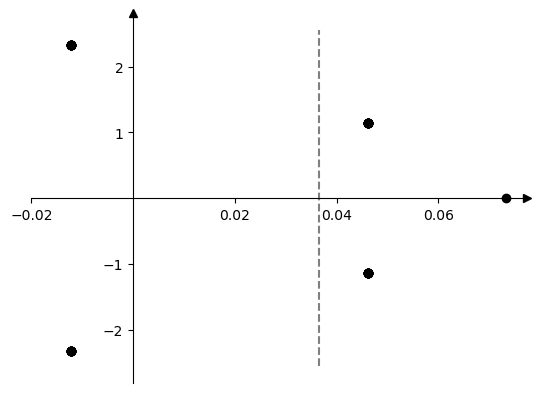}
		\phantomcaption{\footnotesize $\kappa = 1.1, \theta = 0.5, c_p = 1.5, c_t = 5$: $\alpha \approx 0.0732$.}
	\end{subfigure}
	\begin{subfigure}{.45\textwidth}
		\centering
		\includegraphics[width=\linewidth]{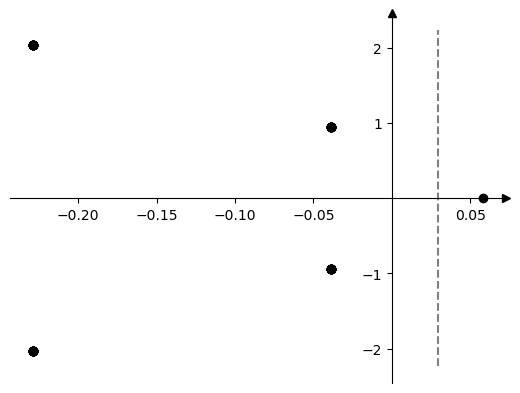}
		\phantomcaption{\footnotesize $\kappa = 2, \theta = 1, c_p = 1.5, c_t = 5$: $\alpha \approx 0.0584$}
	\end{subfigure}
	\caption{Numeric approximation of solutions to $\det(\I_m - \Lap\bmu(z)) = 0$ in the cell cycle model. The dashed line is the critical line $\Re(\lambda) = \alpha/2$.}
	\label{fig:cellcycle}
\end{figure}

Observe that if $\prod_{i\in[m]} p_i$, the probability that a cell completes the cycle, is smaller than $1/2$, then the process goes extinct almost surely. Therefore, we assume that $c_p > 1$. The Malthusian parameter $\alpha$ is then the unique real positive number satisfying
$$ \alpha c_t + \kappa\log(1 + \theta \alpha) = \log c_p. $$
Whether $|\Lambda| > 1$ depends on the parameters of the model (see Figure \ref{fig:cellcycle} for exemplary realisations of $\Lambda$). Note, however, that all roots in $\Lambda$ are simple since
\begin{equation*}
\frac{\dd}{\dd z} \det(\Ip - \Lap\bmu(z)) = c_p(1+\theta z)^{-\kappa} e^{-zc_t} (\kappa(1+\theta z)^{-1} + c_t) \neq 0 \quad \text{when } \Re(z) > 0.
\end{equation*}
That is, $k_\lambda = 1$ for any $\lambda\in\Lambda$. Similarly to Section \ref{sec:m-aryST}, we may express the matrices $A_{\lambda,1}$ in terms of left and right eigenvectors of $\Lap\bmu(\lambda)$. That is, the characteristic polynomial of $\Lap\bmu(\lambda)$ is
\begin{equation*}
x \mapsto \det(x\I_m - \Lap\bmu(\lambda) ) = x^p - c_pe^{-\lambda c_t}(1+\theta \lambda)^{-\kappa} = x^p - 1,
\end{equation*}
therefore the left and right eigenspaces of $\Lap\bmu(\lambda)$ corresponding to the eigenvalue $1$ are spanned by some vectors $\bu(\lambda), \bv(\lambda)$ respectively. We may choose them such that $\bu(\lambda)^\trans \bv(\lambda) = 1$, $\bu(\alpha) = \bu, \ \bv(\alpha) = \bv$ and $\bv(\overline{\lambda}) = \overline{\bv(\lambda)}$ for any $\lambda \in \Lambda$. In view of \eqref{eq:matrixRec}, we have, for some $\beta_\lambda \neq 0$,
\begin{equation*}
A_{\lambda,1} = \frac1{\beta_\lambda} \bv(\lambda)\bu(\lambda)^\trans
\end{equation*}
and reasoning as in Remark \ref{rem:NermanviaA} we conclude that $\beta_\lambda = \bu(\lambda)^\trans (-\Lap\bmu'(\lambda))\bv(\lambda)$. Note that a significant difference between this setting and other models presented in this section is that the matrices $\Lap\bmu(\lambda)$, $\lambda\in\Lambda$ need not share eigenbases.

With this representation of the matrix $A_{\lambda,1}$, the martingale $W_t(\lambda)$ may be written as
\begin{equation*}
W_t(\lambda) = \frac1{\beta_\lambda} \Xi_t(\lambda) \bu(\lambda)^\trans, \quad \Xi_t(\lambda) = \sum_{u\in\cC_t} e^{-\lambda S(u)} \be_{\tau(u)}^\trans \bv(\lambda) 
\end{equation*}
and if $W_t(\lambda) \to W(\lambda)$ as $t\to\infty$, then also $\Xi_t(\lambda) \to \Xi(\lambda)$ for a complex-valued random variable $\Xi(\lambda)$ that satisfies $\beta_\lambda^{-1} \Xi(\lambda) \bu(\lambda)^\trans = W(\lambda)$. In particular, $\Xi(\alpha) = W$ is the limit of Nerman's martingale.

For $n \in [m]$, let
$$\varphi_n(t) = \be_n (1 - X_n) \1_{[0,\infty)}(t),$$
so that $\cZ_t^{\varphi_n}$ is the number of cells that were born up to time $t$ and stopped by the drug in phase $n$. Proposition \ref{thm:asconv} implies that
$$ e^{-\alpha t} Y_t \coloneqq e^{-\alpha t} (\cZ_t^{\varphi_1}, \cdots, \cZ_t^{\varphi_m}) \to \frac{W}{\alpha \beta} (\bu_1(1-p_1), \cdots, \bu_m(1-p_m))$$
almost surely as $t\to\infty$. The fluctuations of the process may be obtained with the help of Theorem \ref{thm:main} and the Cram\'er--Wold device. That is, for a fixed vector $\bq = (\bq_n)_{n\in[m]} \in \R^m$ consider a characteristic $\varphi_\bq(t) = (\bq_n (1 - X_n))_{n\in[m]}\1_{[0,\infty)}(t)$. Then $\E[\varphi_\bq](t) = (\bq_n(1-p_n))_{n\in[m]}\1_{[0,\infty)}(t)$ and thus
\begin{equation*}
\bu(\lambda)^\trans \E[\varphi_\bq](t) = \bw(\lambda)^\trans \bq \1_{[0,\infty)}(t)
\end{equation*}
for a vector $\bw(\lambda) \in \C^m$ whose $n$th entry is $\be_n^\trans \bu(\lambda)(1-p_n)$. As a consequence,
\begin{align*}
H_\Lambda^\bq(t) &= \sum_{\lambda\in\intLambda} \frac{e^{\lambda t}}{\lambda \beta_\lambda} \Xi(\lambda) \bw(\lambda)^\trans \bq, \\
H_\dLambda^\bq(t) &= \sum_{\lambda\in\dLambda} \frac{e^{\lambda t}}{\lambda \beta_\lambda} \be_{\tau(\varnothing)}^\trans \bv(\lambda) \bw(\lambda)^\trans \bq.
\end{align*}
Observe that $\alpha$ is the only real root, therefore the contribution of other elements $\lambda \in \Lambda$ (if they exist) to the scaling $H_\Lambda^\bq, H_\dLambda^\bq$ gives periodic terms.

If $\dLambda = \varnothing$, then we conclude
\begin{equation*}
e^{-\frac\alpha2t} \Big( Y_t - \sum_{\lambda\in\Lambda} \frac{e^{\lambda t}}{\lambda \beta_\lambda} \Xi(\lambda) \bw(\lambda)^\trans \Big) \to W^{1/2} \mathcal N
\end{equation*}
for a multidimensional centered Gaussian variable $\mathcal N$ independent of $W$. Since the roots are simple, the covariance matrix of $\mathcal N$ may be determined with the help of Remark \ref{rem:sigma}.

If $\dLambda \neq \varnothing$, then, since $H_\dLambda^\bq(t) = o(\sqrt{t}e^{\frac\alpha2t})$ for any $\bq$, proceeding as in previous examples we conclude
\begin{equation*}
\frac{e^{-\frac\alpha2t}}{\sqrt{t}} \Big( Y_t - \sum_{\lambda\in\intLambda} \frac{e^{\lambda t}}{\lambda \beta_\lambda} \Xi(\lambda) \bw(\lambda)^\trans \Big) \to W^{1/2} \mathcal N
\end{equation*}
for a multidimensional centered Gaussian variable $\mathcal N$ independent of $W$.

\section{The associated Markov random walk}\label{sec:MRW}

Suppose that, in addition to our standing assumptions, the Malthusian assumption~\ref{assumpt:Malthusian alpha} holds.
Then we can exploit a connection between the CMJ process and an associated Markov random walk
as follows. 
Define the stochastic process $((M_n, S_n))_{n\in\N_0}$ on $[p] \times \R$ via
\begin{multline}\label{eq:manytoone}
\E^i[f((M_0,S_0), \dots (M_n,S_n))] \\
= \E^i\bigg[\sum_{|u|=n} e^{-\alpha S(u)} \frac{\bv_{\tau(u)}}{\bv_i} f((i,0), (\tau(u|_1), S(u|_1)), \dots, (\tau(u),(S(u)))) \bigg]
\end{multline}
for any Borel-measurable $f : ([p]\times \R)^{n+1} \to [0,\infty)$.
One can check that $((M_n,S_n))_{n\in\N_0}$ is a Markov random walk,
i.e., $((M_n,S_n-S_{n-1}))_{n\in\N}$ is a Markov chain with transition kernel depending only on the first coordinate,
see e.g.~\cite[Section 3.1]{Iksanov+Meiners:2015}.
The following facts can also be found in the same reference.
The initial distribution of this process is $\P^i[(M_0,S_0)= (i,0)] = 1$ and its transition kernel is given by
\begin{equation*}
\P^i\big[(M_{n+1}, S_{n+1}-S_n) \in \{k\}\times B \, \big|\, M_n=j \big]
= \frac{\bv_k}{\bv_j} \E^j \bigg[ \sum_{|u|=1} e^{-\alpha S(u)} \1_{\{\tau(u)=k, \, S(u)\in B\}} \bigg]
\end{equation*}
for any Borel $B \subseteq \R$ and $j,k \in [p]$. The process $(M_n)_{n\in\N_0}$
is an irreducible Markov chain with stationary distribution given by $\pi_i = \bv_i\bu_i$ for $i\in[p]$.
We write $\P^\pi \coloneqq \sum_{i=1}^p \pi_i \P^i$ and $\E^\pi$ for corresponding expectation.
Notice for later use that
\begin{align}
\E^\pi[S_1]
&= \sum_{i=1}^p \bu_i\bv_i \E^i[S_1]
= \sum_{i=1}^p \bu_i\bv_i \E^i\bigg[\sum_{|u|=1} e^{-\alpha S(u)} \frac{\bv_{\tau(u)}}{\bv_i} S(u) \bigg]	\notag	\\
&= \sum_{i=1}^p \sum_{j=1}^p \bu_i \E^i\bigg[\sum_{|u|=1} e^{-\alpha S(u)} \1_{\{\tau(u)=j\}} S(u) \bigg] \bv_j
= \bu^\trans (-(\Lap\bmu)'(\alpha)) \bv = \beta.	\label{eq:E^pi[X_1]=beta}
\end{align}
The Markov random walk $((M_n,S_n))_{n\in\N_0}$ has the same lattice type as $\bmu$.

\subsection{Asymptotic behavior of the mean of a CMJ}

The underlying Markov random walk may be used to analyze the mean of a multi-type CMJ process.
The link is known in the literature at least since the work of Crump \cite{Crump:1970a}
and has since been exploited in various works. For the reader's convenience,
we state the result on the first-order asymptotic behavior of
the mean in the present context.

\begin{lem}	\label{lem:meanviarenewal}
Assume that~\ref{assumpt:Malthusian alpha} holds and let $\varphi$ be a non-negative,
c\`adl\`ag characteristic satisfying~\ref{assumpt:dRimean}.
Then, for every $i \in [p]$, the function $t \mapsto e^{-\alpha t} \E^i[\cZ_t^\varphi]$ is bounded and
	\begin{equation*}
	e^{-\alpha t} \E^i[\cZ_t^\varphi] \xrightarrow{t\to\infty} \frac{\bv_i}{\beta} \int_\G e^{-\alpha s} \bu^\trans \E[\varphi(s)] \, \ell(\dd s).
	\end{equation*}
	In particular,
	\begin{equation}\label{eq:varseries}
	\E^i \bigg[ \sum_{u\in\Tree} \be_{\tau(u)}^\trans\varphi_u (t - S(u)) \bigg] \leq e^{\alpha t} C
	\end{equation}
	for some constant $C>0$ and any $t\in\G$.
\end{lem}
\begin{proof}
	Since $\varphi$ is non-negative, we may apply Fubini's theorem and~\eqref{eq:manytoone} to obtain
	\begin{equation*}
	\begin{split}
	e^{-\alpha t} \E^i[\cZ_t^\varphi]
	&= \E^i \bigg[ \sum_{n=0}^\infty \sum_{|u|=n} e^{-\alpha t} \be_{\tau(u)}^\trans \varphi_u(t-S(u)) \bigg] \\
	&= \bv_i \sum_{n=0}^\infty \E^i \bigg[ \sum_{|u|=n} e^{-\alpha S(u)} \frac{\bv_{\tau(u)}}{\bv_i}
	\frac{e^{-\alpha (t-S(u))}}{\bv_{\tau(u)}} \be_{\tau(u)}^\trans \E[\varphi](t-S(u)) \bigg] \\
	&= \bv_i \sum_{n=0}^\infty \E^i \bigg[  \frac{e^{-\alpha(t-S_n)}}{\bv_{M_n}} \be_{M_n}^\trans \E[\varphi](t-S_n) \bigg].
	\end{split}
	\end{equation*}
	Therefore and in view of~\eqref{eq:E^pi[X_1]=beta},
	the statement of the lemma follows from~\eqref{eq:E^pi[X_1]=beta} and the renewal theorem for Markov random walks,
	see e.g.~\cite[Theorem 3.1]{Crump:1970a} in the case that $\varphi$
	is concentrated on the positive halfline or \cite[Theorem 2.1]{Alsmeyer:1994} for a very general formulation.
\end{proof}

\subsection{Existence of the CMJ process}\label{sec:existence}

Lemma~\ref{lem:meanviarenewal} implies that for every $t\in\G$,
the random variable $\cZ_t^\varphi$ is well-defined and integrable
whenever the characteristic $\varphi$ is non-negative and satisfies~\ref{assumpt:dRimean}.
In what follows, we provide two further sets of assumptions under which a CMJ process counted with general characteristic
is well-defined.

To state these results, we first introduce the notion of an \emph{admissible ordering} of $\UHTree$.
We call a sequence $u_1,u_2,\ldots \in \UHTree$ an admissible ordering of $\UHTree$ if
\begin{itemize}
	\item	$\UHTree_n \coloneqq \{u_1,\ldots,u_n\}$ is a subtree of the Ulam-Harris tree $\UHTree$ of cardinality $n$,
	\item	$\UHTree = \bigcup_{n \in \N} \UHTree_n$.
\end{itemize}
Admissible orderings exist, see \cite[p.\;1545]{Iksanov+al:2024}.

\begin{lem}\label{lem:mtg}
Assume that~\ref{assumpt:Malthusian alpha} holds and let $\chi$ be a centered characteristic satisfying~\ref{assumpt:dRivar}.
Fix $t\in \G$ and let $(u_1, u_2, \dots)$ be an admissible ordering of $\UHTree$.
Then the process $(M_n(t))_{n\in\N_0}$ given by
	\begin{equation*}
	M_n(t) = \sum_{k=1}^n \be_{\tau(u_k)}^\trans \chi_{u_k}(t - S(u_k)),	\quad	n \in \N_0
	\end{equation*}
	is a martingale that converges a.\,s.\ and in $L^2$. The series
	\begin{equation*}
	\cZ_t^\chi = \sum_{u\in\Tree} \be_{\tau(u)}^\trans \chi_u(t-S(u)) 
	\end{equation*}
	converges unconditionally in $L^2$ and a.\,s.
	In particular, $\cZ_t^\chi$ is the limit of the martingale $(M_n(t))_{n\in\N_0}$. Moreover,
	\begin{equation}
	\Varv[\cZ_t^\chi] = \Ev\big[\cZ_t^{\E[\chi^2]}\big] = \Ev\big[\cZ_t^{\Var[\chi]}\big].
	\end{equation}
\end{lem}

The lemma is the multi-type extension of \cite[Lemma 4.1]{Iksanov+al:2024}.
The proof of the latter lemma can be adapted to the present situation with only minor modifications
and is therefore omitted.
Similarly, the following result is the multi-type extension of \cite[Lemma 4.2]{Iksanov+al:2024}.
Once again, we state the result without proof, as the argument is a direct generalization of the proof of the corresponding single-type result.

\begin{lem}\label{lem:well-def}
Suppose that~\ref{assumpt:Malthusian alpha} holds.
Further, let $\varphi$ be a characteristic satisfying~\ref{assumpt:dRimean} and~\ref{assumpt:dRivar}.
Then, for any $t\in\G$, the series
	\begin{equation*}
		\cZ_t^\varphi = \sum_{u\in\Tree} \be_{\tau(u)}^\trans \varphi(t - S(u))
	\end{equation*}
	converges unconditionally in $L^1$ and a.\,s.\ under any admissible ordering of $\Tree$.
\end{lem}

\section{Mean expansion and complex Nerman-type martingales}\label{sec:mean-and-mtgs}

In this section we briefly summarize results of \cite{Kolesko+al:2025,Kolesko+al:2025b},
in which a connection between the Laplace transform of the intensity measure $\bmu$,
complex-valued martingales and the asymptotic behavior of the mean of CMJ processes is studied.
We also prove some facts concerning said martingales which will be of use to us in proving our main result,
Theorem~\ref{thm:main}. 

\subsection{The Laplace transform and complex Nerman-type martingales}

Observe that $(\Ip - \Lap\bmu(z))^{-1}$ may be calculated by
\begin{equation}\label{eq:roots}
	(\Ip - \Lap\bmu(z))^{-1} = \frac{{\rm adj}(\Ip - \Lap\bmu(z))}{\det(\Ip - \Lap\bmu(z))} 
\end{equation}
whenever the denominator is non-zero, where ${\rm adj}(A)$ denotes the adjoint matrix of $A$.
Since both the adjoint and the determinant are defined as holomorphic functions of the matrix entries,
the function $z \mapsto (\Ip - \Lap\bmu(z))^{-1}$ is meromorphic.
For $\lambda \in \Lambda$, let $k_\lambda$ denote the maximal order of the poles of the entries of $(\Ip - \Lap\bmu(z))^{-1}$ at $z=\lambda$.
Then~\eqref{eq:roots} implies that $k_\lambda$ does not exceed the order of the zero of $\det(\Ip - \Lap\bmu(z))$ at $z = \lambda$.
Moreover, in a neighborhood of $\lambda$, the function $z \mapsto (\Ip - \Lap\bmu(z))^{-1}$ admits a Laurent series expansion
\begin{equation*}
	(\Ip - \Lap\bmu(z))^{-1}
	= (z-\lambda)^{-k_\lambda} A_{\lambda, k_\lambda} + (z-\lambda)^{-k_\lambda + 1} A_{\lambda, k_\lambda - 1}
	+ \ldots + (z-\lambda)^{-1} A_{\lambda,1} + H_\lambda(z)
\end{equation*}
for some $p\times p$ matrices $(A_{\lambda,j})_{j\in[k_\lambda]}$ and a matrix of holomorphic functions $H_\lambda$.
On the other hand, $\Ip - \Lap\bmu(z)$ may be expanded componentwise into a Taylor series;
from the fact that the product of both series is the identity matrix, one may deduce that the matrices $A_{\lambda,j}$ satisfy~\eqref{eq:matrixRec}.

The matrices $A_{\lambda,j}$ play an essential role in understanding the asymptotic properties of a CMJ process with reproduction intensity $\bmu$. As was already mentioned in Section~\ref{sec:setup}, it was shown in \cite{Kolesko+al:2025b} that there is a matrix-valued martingale $(W_t(\lambda))_{t\geq 0}$ connected to each $\lambda \in \Lambda$.
Using the Kronecker product and the exponential matrices defined in Section~\ref{sec:matrices}, it may be written as
\begin{equation}\label{eq:mtg1}
	W_t(\lambda) = \sum_{u \in \cC_t} (\exp(\lambda, -S(u), k_\lambda) \otimes \be_{\tau(u)}^\trans) \Av_\lambda,
\end{equation}
where $\Av_\lambda = \sum_{j=1}^{k_\lambda} \be_{j,k_\lambda} \otimes A_{\lambda,j}$.
Moreover, $(W_t(\lambda))_{t\geq 0}$ may be expressed in terms of a matrix-valued CMJ process as defined in Section~\ref{sec:matrices}.
Namely, let $\phi_\lambda$ be the $\M_{k_\lambda \times 1}(\C) \otimes \M_{p \times p}(\C)$-valued characteristic given by
\begin{equation}\label{def:philambda}
	\phi_\lambda(t) = \1_{[0,\infty)}(t) \bigg(\int_{(t,\infty)} \exp(\lambda,t-s,k_\lambda) \otimes \bxi(\dd s)\bigg) \Av_\lambda.
\end{equation}
Note that the corresponding CMJ process is $\M_{k_\lambda \times p}(\C)$-valued.
By \cite[Theorem 2.1]{Kolesko+al:2025b},
\begin{equation}\label{eq:mtg2}
	W_t(\lambda) = \exp(\lambda, -t, k_\lambda) \cZ_t^{\phi_\lambda},	\quad	t \geq 0.
\end{equation}

\begin{rem}\label{rem:NermanviaA}
	Recall that under assumption~\ref{assumpt:Malthusian alpha},
	$1$ is the dominant (Perron--Frobenius) eigenvalue of $\Lap\bmu(\alpha)$ and $\{\bv\}$,
	$\{\bu^\trans\}$ are the bases of its right and left eigenspaces, respectively. Since~\eqref{eq:matrixRec} implies
	\begin{equation*}
		A_{\alpha, k_{\alpha}} = \Lap\bmu(\alpha) A_{\alpha, k_{\alpha}} = A_{\alpha,k_\alpha} \Lap\bmu(\alpha),
	\end{equation*}
	all columns of the matrix $A_{\alpha,k_\alpha}$ must be scalar multiples of the vector $\bv$,
	while all the rows must be scalar multiples of $\bu^\trans$. That is, $A_{\alpha,k_{\alpha}} = c  \bv \bu^\trans$ for some $c \neq 0$. Moreover, by \cite[Proposition 2.2(d)]{Kolesko+al:2025},~\ref{assumpt:LaplaceDomain} implies that $k_\alpha = 1$. In particular, the martingale $(W_t(\alpha))_{t\geq0}$ is $\M_{1\times p}(\C)$-valued. Furthermore, the same argument that gives~\eqref{eq:matrixRec} also implies
	\begin{equation*}
		A_{\alpha,0} (\Ip - \Lap\bmu(\alpha)) + A_{\alpha,1} (-(\Lap\bmu)'(\alpha)) = \Ip
	\end{equation*}
	where $A_{\alpha,0}$ denotes the constant term in the Laurent expansion of $(\Ip-\Lap\bmu(z))^{-1}$ around $z=\alpha$.
	Multiplying both sides by $\bv$ from the right, we get $c\beta \bv = \bv$,
	which implies that $c = 1/\beta$ since $\bv$ is nonzero.
	Therefore,
	\begin{equation*}
	A_{\alpha,1} = \frac1\beta  \bv \bu^\trans.
	\end{equation*}
	In this case,~\eqref{eq:mtg1} becomes
	\begin{equation*}
		W_t(\alpha) = \sum_{u \in \cC_t} e^{-\alpha S(u)} \frac1\beta \bv_{\tau(u)} \bu^\trans = \frac{1}\beta W_t \bu^\trans,
		\quad	t \geq 0
	\end{equation*}
	where $(W_t)_{t\geq0}$ is the Nerman's martingale.
\end{rem}

The following proposition is the main result of \cite{Kolesko+al:2025b}, specialized to the $L^2$ setting
and stated in terms of the assumptions of the present paper.
Note that Lemma~\ref{lem:mtgsasCMJ} below implies the proposition
and thus provides a second proof.

\begin{prop}	\label{Prop:main rslt Tomic's thesis}
Suppose that~\ref{assumpt:Malthusian alpha} and~\ref{assumpt:hreprMoments} hold and let $\lambda \in \Lambda$
with $\Re(\lambda) > \frac\alpha2$.
Then $(W_t(\lambda))$ converges a.\,s.\ and in $L^2$ to some random matrix $W(\lambda)$.
\end{prop}
\begin{proof}[Source]
The proposition follows from \cite[Theorem 2.3]{Kolesko+al:2025b}
once we show that the assumptions here imply those there.
Assumption~\ref{assumpt:Malthusian alpha} here implies (A1) and (A2) there and
\ref{assumpt:hreprMoments}
together with $\Re(\lambda)>\frac\alpha2$ implies that $\lambda$ is in the interior of the domain
where $\Lap\bmu$ is finite. Assumption~\ref{assumpt:hreprMoments} further implies (A3) with $q=2$ in the cited paper.
\end{proof}

For our analysis it is convenient to express not only $(W_t(\lambda))_{t\geq 0}$, but also its limit in terms of a CMJ process.
Therefore, we provide yet another characterization of this martingale.
Consider the random element of $\M_{k_\lambda\times k_\lambda}(\C) \otimes \M_{p\times p}(\C)$ given by
\begin{equation*}
	Y(\lambda) = \int \exp(\lambda, -s, k_\lambda) \otimes\bxi(\dd s) - \Ikl \otimes \Ip
\end{equation*}
and, for an interval $J \subseteq \R$, let $\chi_{\lambda,J} : \Omega_\varnothing \times \R \to \M_{k_\lambda \times 1}(\C) \otimes \M_{p \times p}(\C)$ be given by
\begin{equation*}
	\chi_{\lambda,J}(t) = \1_J(t) (\exp(\lambda, t, k_\lambda)\otimes \Ip) Y(\lambda) \Av_\lambda,	\quad	t \in \R.
\end{equation*}
Note that if $J$ is an interval of the form $[a,b)$ or $(-\infty,b)$ for some $a,b\in\R$, then $\chi_{\lambda,J}$ is c\`adl\`ag. Moreover, we have
\begin{equation*}
	\begin{split}
		\E &\bigg[ \int \exp(\lambda, -s, k_\lambda) \otimes\bxi(\dd s) \bigg] = \int \exp(\lambda, -s, k_\lambda) \otimes\bmu(\dd s) \\
		&= \begin{pmatrix}
			\int e^{-\lambda s} \, \bmu(\dd s) & \int -se^{-\lambda s} \, \bmu(\dd s) & \dots & \int \frac{(-s)^{k_\lambda-1}}{(k_\lambda-1)!} e^{-\lambda s} \, \bmu(\dd s) \\
			0 & \int e^{-\lambda s} \, \bmu(\dd s) & \dots & \int \frac{(-s)^{k_\lambda-2}}{(k_\lambda-2)!} e^{-\lambda s} \, \bmu(\dd s) \\
			\vdots & & \ddots & \vdots \\
			0 &  \dots & & \int e^{-\lambda s} \, \bmu(\dd s)
		\end{pmatrix} \\
		&= \begin{pmatrix}
			\Lap\bmu(\lambda) & (\Lap\bmu)^{(1)}(\lambda) & \dots & \frac{(\Lap\bmu)^{(k_\lambda-1)}(\lambda)}{(k_\lambda-1)!} \\
			0 & \Lap\bmu(\lambda) & \dots & \frac{(\Lap\bmu)^{(k_\lambda - 2)}(\lambda)}{(k_\lambda-2)!} \\
			\vdots & & \ddots & \vdots \\
			0 &  \dots & & \Lap\bmu(\lambda)
		\end{pmatrix}.
	\end{split}
\end{equation*}
In particular,~\eqref{eq:matrixRec} gives
\begin{equation*}
	\E \bigg[ \int \exp(\lambda, -s, k_\lambda) \otimes\bxi(\dd s) \bigg] \Av_\lambda = \Av_\lambda.
\end{equation*}
Consequently, $\E[Y(\lambda) \Av_\lambda] = 0$ and $\chi_{\lambda,J}$ is centered.

\begin{lem}\label{lem:chilambda}
Assume~\ref{assumpt:Malthusian alpha} and~\ref{assumpt:hreprMoments}.
Fix $\lambda \in \Lambda$ and let $J$ be an interval of the form $J = [-a,b) \cap \R$ for some $-\infty \leq a < b < \infty$.
	\begin{enumerate}[\normalfont(i), wide]
		\item The characteristic $\phi_\lambda$ satisfies~\ref{assumpt:dRimean} through~\ref{assumpt:local ui}.
		\item If $\Re(\lambda) > \alpha/2$,
		then the characteristic $\chi_{\lambda,J}$ satisfies~\ref{assumpt:dRimean} through~\ref{assumpt:local ui}.
	\end{enumerate}
\end{lem}
\begin{proof}\begin{enumerate}[\normalfont(i), wide]
\item By Remark~\ref{rem:charAssumpts}, it is enough to show that
\begin{equation*}
	\int e^{-\alpha x} \E\Big[\sup_{|x-t|\leq1} \|\phi_\lambda(t)\|_{\mathrm{HS}}^2\Big] \dd x < \infty.
\end{equation*}
Since
\begin{equation*}
	\|\phi_\lambda(t)\|_{\mathrm{HS}}^2
	\leq \1_{[0,\infty)}(t) \bigg\|\int_{(t,\infty)} \exp(\lambda,t-s,k_\lambda)\otimes \bxi(\dd s) \bigg\|^2_{\mathrm{HS}} \|\Av_\lambda\|_{\mathrm{HS}}^2,
\end{equation*}
it is in turn sufficient to show that
\begin{equation}\label{eq:chilGoal}
\int_{-1}^\infty e^{-\alpha x} \E\bigg[\sup_{|x-t|\leq1}
\bigg\|\int_{(t,\infty)} \exp(\lambda,t-s,k_\lambda)\otimes \bxi(\dd s) \bigg\|^2_{\mathrm{HS}}\bigg] \, \dd x < \infty.
\end{equation}
Here,
\begin{multline*}
\int_{(t,\infty)} \!\! \exp(\lambda, t-s, k_\lambda) \otimes \bxi(\dd s) \\
= \begin{pmatrix}
\int_{(t,\infty)} e^{\lambda(t-s)} \, \bxi(\dd s) & \int_{(t,\infty)} (t-s)e^{\lambda(t-s)} \, \bxi(\dd s)
& \ldots & \int_{(t,\infty)} \frac{(t-s)^{k_\lambda-1}}{(k_\lambda-1)!} e^{\lambda(t-s)} \, \bxi(\dd s) \\
0 & \int_{(t,\infty)} e^{\lambda(t-s)} \, \bxi(\dd s) & \ldots & \int_{(t,\infty)} \frac{(t-s)^{k_\lambda-2}}{(k_\lambda-2)!}e^{\lambda(t-s)} \, \bxi(\dd s) \\
\vdots & & \ddots & \vdots \\
0 & \ldots && \int_{(t,\infty)} e^{\lambda(t-s)} \, \bxi(\dd s)
\end{pmatrix}.
\end{multline*}
Therefore, for some constant $C_1$ depending on $k_\lambda$,
\begin{equation}\label{eq:chilTemp'}
\begin{split}
\bigg\| \int \limits_{(t,\infty)} \!\! \exp(\lambda, t-s, k_\lambda) \otimes \bxi(\dd s) \bigg\|_{\mathrm{HS}}^2 
&\leq C_1 \bigg\| \int \limits_{(t,\infty)} \!\! (1+|t-s|)^{k_\lambda-1} e^{\Re(\lambda)(t-s)} \, \bxi(\dd s) \bigg\|_{\mathrm{HS}}^2 \\
&=C_1 \sum_{i,j=1}^p \! \bigg( \! \int \limits_{(t,\infty)} \!\! (1+|t-s|)^{k_\lambda-1} e^{\Re(\lambda)(t-s)} \, \xi^{ij}(\dd s)\bigg)^{\!\!2}.
\end{split}
\end{equation}

Assume first that $\Re(\lambda) > \frac\alpha2$ and take $\gamma \in (\frac\alpha2,\Re(\lambda))$.
There is a constant $C_2$ (depending on $\gamma,k_\lambda$) such that, for any $x \leq 0$,
\begin{equation*}
(1+|x|)^{k_\lambda-1}e^{\Re(\lambda)x} \leq C_2 e^{\gamma x},
\end{equation*}
hence, by~\eqref{eq:chilTemp'},
\begin{equation*}
\bigg\| \int_{(t,\infty)} \exp(\lambda, t-s, k_\lambda) \otimes \bxi(\dd s) \bigg\|_{\mathrm{HS}}^2
\leq C_1 C_2^2 \sum_{i,j=1}^p \bigg( \int_{(t,\infty)} e^{\gamma (t-s)} \, \xi^{ij}(\dd s) \bigg)^{\!\!2} 
\end{equation*}
and thus
\begin{multline}\label{eq:chilTemp}
	\E\bigg[\sup_{|x-t|\leq1} \bigg\|\int_{(t,\infty)} \exp(\lambda,t-s,k_\lambda)\otimes \bxi(\dd s) \bigg\|^2_{\mathrm{HS}}\bigg] \\
	\begin{aligned}	
	&\leq \E\bigg[\sup_{|x-t|\leq1} C_1 C_2^2 \sum_{i,j=1}^p \bigg( \int_{(t,\infty)} e^{\gamma (t-s)} \, \xi^{ij}(\dd s) \bigg)^{\!\!2}\bigg] \\
	&\leq C_1 C_2^2  \sum_{i,j=1}^p e^{2\gamma (x+1)} \E\bigg[ \bigg( \int_{(x-1,\infty)} e^{-\gamma s} \, \xi^{ij}(\dd s) \bigg)^{\!\!2}\bigg].
	\end{aligned}
\end{multline}
For any $i,j\in[p]$,
\begin{multline*}
	\int_{-1}^\infty e^{(2\gamma-\alpha) x} \E\bigg[ \bigg( \int_{(x-1,\infty)} e^{-\gamma s} \, \xi^{ij}(\dd s) \bigg)^{\!\!2}\bigg] \, \dd x \\
	\begin{aligned}
	&= \E\bigg[ \iint e^{-\gamma s} e^{-\gamma r}  \int_{-1}^{s\wedge r+1} e^{(2\gamma - \alpha) x} \, \dd x \, \xi^{ij}(\dd s) \, \xi^{ij}(\dd r) \bigg] \\
	&\leq \frac{e^{2\gamma - \alpha}}{2\gamma - \alpha}\E\bigg[ \iint e^{-\gamma s} e^{-\gamma r} e^{(2\gamma-\alpha) s \wedge r} \, \xi^{ij}(\dd s) \, \xi^{ij}(\dd r)\bigg] \\
	&\leq \frac{e^{2\gamma - \alpha}}{2\gamma - \alpha}\E\bigg[ \iint e^{-\gamma s} e^{-\gamma r} e^{\big(\gamma-\frac\alpha2\big) (s + r)} \, \xi^{ij}(\dd s) \, \xi^{ij}(\dd r)\bigg] \\
	&= \frac{e^{2\gamma - \alpha}}{2\gamma - \alpha}\E\bigg[ \bigg(\int e^{-\frac\alpha2 s} \, \xi^{ij}(\dd s) \! \bigg)^{\!\!2}\bigg] < \infty,
	\end{aligned}
\end{multline*}
where the last step follows from~\ref{assumpt:hreprMoments}. This together with~\eqref{eq:chilTemp} implies~\eqref{eq:chilGoal}.

Next, consider the case $\Re(\lambda) = \frac\alpha2$. For any $i,j\in[p]$,
\begin{multline*}
\int_{-1}^\infty e^{-\alpha x} \E\bigg[\sup_{|x-t|\leq1} e^{\alpha t}
\bigg( \! \int \limits_{(t,\infty)} \!\! (1+|t - s|)^{k_\lambda-1} e^{-\frac\alpha2s} \, \xi^{ij}(\dd s)\bigg)^{\!\!2}\bigg] \, \dd x \\
	\begin{aligned}	
		&\leq e^\alpha  \int_{-1}^\infty \E\bigg[ \bigg( \! \int \limits_{(x-1,\infty)} \!\! (2 + s -x)^{k_\lambda-1} e^{-\frac\alpha2s}
		\, \xi^{ij}(\dd s)\bigg)^{\!\!2}\bigg] \, \dd x \\
		&= e^\alpha \E\bigg[ \iiint_{-1}^{r\wedge s + 1}(2 + s -x)^{k_\lambda-1}(2 + r -x)^{k_\lambda-1}
		\, \dd x \, e^{-\frac\alpha2s} e^{-\frac\alpha2 r} \, \xi^{ij}(\dd s) \, \xi^{ij}(\dd r)\bigg] \\
		&\leq e^\alpha \E\bigg[ \iiint_{-1}^{r\wedge s + 1}
		\, \dd x\, (3 + s)^{k_\lambda-1}(3 + r)^{k_\lambda-1} \, e^{-\frac\alpha2s} e^{-\frac\alpha2 r} \, \xi^{ij}(\dd s) \, \xi^{ij}(\dd r)\bigg] \\
		&\leq e^\alpha 2^{2k_\lambda-2}
		\E\bigg[ \iint (2 + r\wedge s) (2 + s)^{k_\lambda-1}(2 + r)^{k_\lambda-1} \, e^{-\frac\alpha2s} e^{-\frac\alpha2 r}
		\, \xi^{ij}(\dd s) \xi^{ij}(\dd r)\bigg] \\
		&\leq e^\alpha 2^{4k_\lambda - 3} \E\bigg[\bigg( \! \int \Big(1+\frac{s}2\Big)^{k_\lambda-\frac12} e^{-\frac\alpha2s} \, \xi^{ij}(\dd s)\bigg)^{\!\!2}\bigg] < \infty,
	\end{aligned}
\end{multline*}
where the last step follows from~\ref{assumpt:hreprMoments} and the penultimate one from the inequality $(1 + a \wedge b) \leq \sqrt{(1+a)(1+b)}$. This together with~\eqref{eq:chilTemp} implies~\eqref{eq:chilGoal}.

\item Recall that $\chi_{\lambda,J}$ is centered and thus trivially satisfies~\ref{assumpt:dRimean}.
By Remark~\ref{rem:charAssumpts} it is enough to show that
\begin{equation*}
	\int e^{-\alpha x} \E\Big[\sup_{|x-t|\leq1} \|\chi_{\lambda,J}(t)\|_{\mathrm{HS}}^2\Big] \dd x < \infty.
\end{equation*}
Pick $\gamma \in (\frac\alpha2,\Re(\lambda))$. By~\eqref{eq:expNorm} there is a constant $C_1$,
depending on $\gamma,k_\lambda,p,b$, such that, for all $t \in (-\infty,b)$,
\begin{equation*}
	\|\exp(\lambda,t,k_\lambda) \otimes \Ip \|_{\mathrm{HS}} \leq C_1 e^{\gamma t}.
\end{equation*}
Moreover, since
\begin{equation*}
Y(\lambda) = \begin{pmatrix}
	\int e^{-\lambda s} \, \bxi(\dd s) & \int -se^{-\lambda s} \, \bxi(\dd s) & \dots &
	\int \frac{(-s)^{k_\lambda-1}}{(k_\lambda-1)!} e^{-\lambda s} \, \bxi(\dd s) \\
	0 & \int e^{-\lambda s} \, \bxi(\dd s) & \dots & \int \frac{(-s)^{k_\lambda-2}}{(k_\lambda-2)!} e^{-\lambda s} \, \bxi(\dd s) \\
	\vdots & & \ddots & \vdots \\
	0 &  \dots & & \int e^{-\lambda s} \, \bxi(\dd s)
\end{pmatrix} - \Ikl \otimes \Ip,
\end{equation*}
we have that, for some constants $C_2, C_3$ depending on $\lambda, k_\lambda$,
\begin{equation*}
\|Y(\lambda)\|^2_{\mathrm{HS}} \leq C_2 \bigg\|\int (1+|s|^{k_\lambda -1}) e^{-\lambda s} \, \bxi(\dd s)\bigg\|^2_{\mathrm{HS}}
\leq C_3 \sum_{i,j=1}^p \bigg( \int e^{-\frac\alpha 2 s} \,\xi^{ij}(\dd s) \! \bigg)^{\!2}. 
\end{equation*}
Therefore, assumption~\ref{assumpt:hreprMoments} implies that $\E[\|Y(\lambda)\|^2_{\mathrm{HS}}] < \infty$ and
\begin{equation*}
\begin{split}
\E[\|\chi_{\lambda,J}(t)\|_{\mathrm{HS}}^2]
&\leq \1_{(-\infty,b)}(t) \|\exp(\lambda,t,k_\lambda) \otimes \Ip\|_{\mathrm{HS}}^2 \,
\E[\|Y_\lambda\|^2_{\mathrm{HS}}] \, \|\Av_\lambda\|_{\mathrm{HS}}^2 \\
& \leq \1_{(-\infty,b)}(t) C_4 e^{2\gamma t}
\end{split}
\end{equation*}
for all $t \in \R$ and some finite constant $C_4$. Since $2\gamma-\alpha > 0$,
\begin{equation*}
\int e^{-\alpha x} \E\Big[\sup_{|x-t|\leq1} \|\chi_{\lambda,J}(t)\|_{\mathrm{HS}}^2\Big] \, \dd x
\leq C_4 e^{2\gamma} \int_{-\infty}^{b+1} e^{(2\gamma-\alpha) x} \, \dd x < \infty
\end{equation*}
as claimed.
\end{enumerate}
\end{proof}

\begin{lem}\label{lem:mtgsasCMJ}
Assume that~\ref{assumpt:Malthusian alpha} and~\ref{assumpt:hreprMoments} hold.
For $\lambda \in \Lambda$ with $\Re(\lambda) > \frac\alpha2$,
the martingale $(W_t(\lambda))_{t\geq 0}$ satisfies
	\begin{equation}\label{eq:mtg3}
		W_t(\lambda) = (\Ikl \otimes \be_{\tau(\varnothing)}^\trans) \Av_\lambda + \cZ_0^{\chi_{\lambda,[-t,1)}},	\quad	t \geq 0
	\end{equation}
	and converges a.\,s.\ and in $L^2$ to
	\begin{equation}\label{eq:mtgLimit}
		W(\lambda) = (\Ikl \otimes \be_{\tau(\varnothing)}^\trans) \Av_\lambda + \cZ_0^{\chi_{\lambda,(-\infty,1)}}.
	\end{equation}
	In particular, $\Ev[W(\lambda)] = \Av_\lambda$,
	with the convention introduced in~\eqref{eq:mean of matrix-valued CMJ}.
\end{lem}
\begin{proof}
Lemma~\ref{lem:chilambda} together with Lemma~\ref{lem:mtg} imply that
the variables $\cZ_0^{\chi_{\lambda,[-t,1)}}, \cZ_0^{\chi_{\lambda,(-\infty,1)}}$
are well-defined as unconditional limits in the a.\,s.\ and $L^2$ sense. Observe that, for any $t \geq 0$,
\begin{multline*}
\cZ_0^{\chi_{\lambda,[-t,1)}}
= \sum_{u \in \Tree} \1_{[-t,1)}(- S(u)) (\Ikl \otimes \be_{\tau(u)}^\trans) (\exp(\lambda, -S(u), k_\lambda) \otimes \Ip) \, Y_u(\lambda) \, \Av_\lambda \\
	\begin{aligned}
	= &\sum_{u\in\Tree} \1_{\{S(u) \leq t\}} \big(\exp(\lambda, -S(u), k_\lambda) \otimes \be_{\tau(u)}^\trans\big) \bigg(\int \exp(\lambda, -s, k_\lambda) \otimes \bxi_u(\dd s) - \Ikl \otimes \Ip\bigg) \, \Av_\lambda \\
	= &\sum_{u\in\Tree} \1_{\{S(u) \leq t\}} \int \exp(\lambda, -s-S(u), k_\lambda) \otimes (\be_{\tau(u)}^\trans \bxi_u)(\dd s) \, \Av_\lambda \\
	&- \sum_{u\in\Tree} \1_{\{S(u) \leq t\}} (\exp(\lambda, -S(u), k_\lambda) \otimes \be_{\tau(u)}^\trans) \, \Av_\lambda \\
	= &\sum_{u \in \Tree} \1_{\{S(u) \leq t\}} \sum_{k=1}^{N_u} (\exp(\lambda, -S(uk), k_\lambda) \otimes \be_{\tau(uk)}^\trans) \, \Av_\lambda \\
	&- \sum_{u\in\Tree} \1_{\{S(u) \leq t\}} (\exp(\lambda, -S(u), k_\lambda) \otimes \be_{\tau(u)}^\trans) \, \Av_\lambda \\
	= &\sum_{u\in\Tree\setminus\{\varnothing\}} \!\!\! \big(\1_{\{S(u|_{|u|-1})\leq t\}} - \1_{\{S(u) \leq t\}} \big) (\exp(\lambda, -S(u), k_\lambda) \otimes \be_{\tau(u)}^\trans) \, \Av_\lambda \\
	&- (\exp(\lambda, 0, k_\lambda) \otimes \be_{\tau(u)}^\trans) \, \Av_\lambda,
	\end{aligned}
\end{multline*}
	where all the manipulations are justified by the fact that all the series have a.\,s.\ only finitely many non-zero terms.
	Finally, note that $\1_{\{S(u|_{|u|-1})\leq t\}} - \1_{\{S(u)\leq t\}} = \1_{\{u\in\cC_t\}}$. Therefore,
	\begin{equation*}
		\cZ_0^{\chi_{\lambda,[-t,1)}} = \sum_{u\in\cC_t} (\exp(\lambda, -S(u),k_\lambda) \otimes \be_{\tau(u)}^\trans ) \Av_\lambda - (\Ikl \otimes \be_{\tau(\varnothing)}^\trans) \Av_\lambda, 
	\end{equation*} 
	which, in view of~\eqref{eq:mtg1}, proves~\eqref{eq:mtg3}.	\smallskip
	
	To show the $L^2$-convergence,
	observe that
	\begin{equation*}
		\chi_{\lambda,(-\infty,1)}(s) - \chi_{\lambda,[-t,1)}(s) = \chi_{\lambda,(-\infty,-t)}(s).
	\end{equation*}
	Fix any $x\in \R^{k_\lambda}$, $y \in \R^p$. Then the characteristic
	\begin{equation*}
		\chi_{\lambda,x,y,-t} = (x^\trans \otimes \I_p) \chi_{\lambda,(-\infty,-t)} (1 \otimes y)
	\end{equation*}
	is $\C^p$-valued, centered and satisfies~\ref{assumpt:dRivar}. Therefore, by Lemma~\ref{lem:mtg},
	\begin{equation*}
		\Ev\big[\big|(x^\trans \otimes \I_p) \cZ_0^{\chi_{\lambda, (-\infty,-t)}} (1 \otimes y)\big|^2 \big] = \Ev\big[\big|\cZ_0^{\chi_{\lambda,x,y,-t}}\big|^2\big] = \Ev\Big[\cZ_0^{\E[|\chi_{\lambda,x,y,-t}|^2]}\Big].
	\end{equation*}
	However, $\E[|\chi_{\lambda,x,y,-t}|^2] \leq \E[|\chi_{\lambda,x,y,1}|^2]$ entry-wise.
	Therefore, by the dominated convergence theorem,
	\begin{equation*}
		\lim_{t\to\infty} \Ev \Big[\cZ_0^{\E[|\chi_{\lambda,x,y,-t}|^2]}\Big]
		= \Ev \bigg[ \sum_{u\in\Tree} \be_{\tau(u)}^\trans \lim_{t\to\infty}\E\big[ |\chi_{\lambda,x,y,-t}|^2\big] (-S(u)) \bigg] = 0.
	\end{equation*}
	Since $x,y$ were arbitrary, we obtain
	\begin{equation*}
		\Ev\big[\| \cZ_0^{\chi_{\lambda,(-\infty,1)}} - \cZ_0^{\chi_{\lambda,[-t,1)}}\|^2\big]
		= \Ev\big[\| \cZ_0^{\chi_{\lambda,(-\infty,-t)}} \|^2\big] \to 0	\quad	\text{as } t \to \infty,
	\end{equation*}
	which completes the proof.	
\end{proof}

\begin{cor}\label{lem:mtgLasCMJ}
Let~\ref{assumpt:Malthusian alpha} and~\ref{assumpt:hreprMoments} hold.
If $\Re(\lambda) > \frac\alpha2$, then
	\begin{equation}\label{eq:mtgLim}
		W(\lambda) = \exp(\lambda, -t, k_\lambda) \cZ_t^{\phi_\lambda + \chi_\lambda},
	\end{equation}
	where $\chi_\lambda \coloneqq \chi_{\lambda, (-\infty,0)}$.
\end{cor}
\begin{proof}
	By Lemma~\ref{lem:mtgsasCMJ}, $W(\lambda) - W_t(\lambda) = \cZ_0^{\chi_{\lambda, (-\infty,-t)}}$. Observe that for any $s\in\R$, $t\geq 0$,
	\begin{equation*}
		\begin{split}
			(\exp(\lambda, -t, k_\lambda) \otimes \Ip) \chi_\lambda(s)
			&= \1_{(-\infty,0)}(s)(\exp(\lambda, s-t, k_\lambda) \otimes \Ip) Y(\lambda) \Av_\lambda \\
			&= \chi_{\lambda,(-\infty,-t)} (s-t).
		\end{split}
	\end{equation*}
	Thus,
	\begin{equation*}
		\begin{split}
			\cZ_0^{\chi_{\lambda, (-\infty,-t)}}
			&= \sum_{u\in\Tree} (\Ikl \otimes \be_{\tau(u)}^\trans ) \chi_{\lambda, (-\infty, -t),u} (-S(u)) \\
			& = \sum_{u\in\Tree} (\Ikl \otimes \be_{\tau(u)}^\trans ) (\exp(\lambda, -t, k_\lambda) \otimes \Ip) \chi_{\lambda,u} (t-S(u)) \\
			& = \sum_{u\in\Tree} (\exp(\lambda, -t, k_\lambda) \otimes 1)(\Ikl \otimes \be_{\tau(u)}^\trans) \chi_{\lambda,u}(t-S(u)) \\
			& = \exp(\lambda, -t, k_\lambda) \cZ_t^{\chi_\lambda},
		\end{split}
	\end{equation*}
	which together with~\eqref{eq:mtg2} finishes the proof.
\end{proof}

\subsection{The mean expansion}

The series expansions of $(\Ip - \Lap\bmu(z))^{-1}$ in the neighborhoods of its roots play a crucial role
in understanding the mean of a CMJ process counted with a characteristic $\varphi$.
In \cite{Kolesko+al:2025}, the asymptotic behavior of $m_t^\varphi = \Ev[\cZ_t^\varphi]$ was established by studying the solutions to Markov renewal equations. In the non-lattice case, it was additionally assumed that
\begin{align}	\label{eq:sup||(Ip-Lmu(theta+ieta)^-1||<infty}
\sup_{\theta \geq \vartheta,\, \eta \geq \eta_0} \|(\Ip-\Lap\bmu(\theta+\imag \eta))^{-1}\| < \infty
\end{align}
for some $\eta_0>0$ and $\vartheta < \alpha$.
For the purpose of establishing a central limit theorem, we shall assume here that $\vartheta \in (0,\alpha/2)$ 
such that $\Lap\bmu$ is finite in some neighborhood of $\vartheta$ and
$\det(\Ip-\Lap\bmu(z))$ has no roots in the strip $[\vartheta, \alpha/2) + \imag \R$. 
For a vector-valued function $f$,
denote by $\mathrm{V}\!f(t)$ its total variation on $(-\infty,t]$, calculated entry-wise.
If the characteristic $\varphi$ vanishes on the negative half-line and satisfies
\begin{equation*}
	\int_0^\infty e^{-\theta x} \, \mathrm{V} \E[\varphi](x) \, \dd x < \infty
\end{equation*}
for some $\theta \in (\vartheta, \alpha/2)$, then for arbitrarily small $\eps>0$,
\begin{equation}
	m_t^\varphi
		= \sum_{\lambda \in \Lambda} e^{\lambda t} \sum_{k=0}^{k_\lambda-1} A_{\lambda,k+1}
		\int \E[\varphi](s)\frac{(t-s)^{k}}{k!} e^{-\lambda s} \, \dd s + O\big(e^{(\theta+\eps)t}\big)
\end{equation}
as $t\to\infty$, see \cite[Theorem 2.5]{Kolesko+al:2025}.
This formula may be rewritten using the exponential matrices and Kronecker product as
\begin{equation}\label{eq:meanAsymptotics}
	m_t^\varphi
	= \sum_{\lambda \in \Lambda} \int_\G (\be_1^\trans \exp(\lambda, t-s, k_\lambda) \otimes \Ip) \Av_\lambda \E[\varphi](s) \, \ell(\dd s) + r(t),
\end{equation}
where $r$ satisfies~\eqref{eq:rasympt}.

In the lattice case the mean expansion is expressed in \cite{Kolesko+al:2025}
not in terms of matrices connected to the Laplace transform, but to the generating function of $\bmu$. That is, let
\begin{equation*}
	\Gen \bmu(z) = \sum_{n=0}^\infty \bmu(\{n\}) z^n,	\quad	|z| \leq e^{-\vartheta}
\end{equation*}
where $\vartheta > 0$ is as in~\ref{assumpt:LaplaceDomain}.
Since $\Gen\bmu(e^{-z}) = \Lap\bmu(z)$, for each $\lambda \in \Lambda$, $e^{-\lambda}$ is a pole of $(\Ip - \Gen\bmu(z))^{-1}$ of multiplicity $k_\lambda$, cf.~\cite[Lemma 2.6]{Kolesko+al:2025}, which gives rise to the expansion
\begin{equation*}
	(\Ip - \Gen\bmu(z))^{-1} = (z-e^{-\lambda})^{-k_\lambda} B_{\lambda, k_\lambda} + \dots + (z-e^{-\lambda})^{-1} B_{\lambda,1} + R_\lambda(z)
\end{equation*}
around $e^{-\lambda}$ for some $p\times p$ matrices $(B_{\lambda,k})_{k\in[k_\lambda]}$
and a holomorphic remainder $R_\lambda$. In the proof of \cite[Theorem 2.7]{Kolesko+al:2025}
it was shown that if $\varphi$ is a characteristic vanishing on the negative half-line and satisfying
\begin{equation*}
	\sum_{n=0}^\infty \E[\varphi](n) e^{-\theta n} < \infty
\end{equation*}
for some $\vartheta < \theta < \alpha/2$ such that $\det(\Ip-\Lap\bmu(z))$ has no zeros in the strip $[\theta,\frac\alpha2) + \imag \R$,
then, for arbitrarily small $\eps>0$,
\begin{equation}	\label{eq:mean expansion lattice}
	\begin{split}
		m^\varphi_n &= \sum_{\lambda\in\Lambda} e^{\lambda n} \sum_{k=1}^{k_\lambda} B_{\lambda,k} \sum_{l=0}^{k-1}  (-e^\lambda)^{k-l} \binom{n+k-l-1}{k-l-1} \frac1{l!} (\Gen \E[\varphi])^{(l)}(e^{-\lambda}) + O\big(e^{(\theta+\eps)n}\big) \\
		&= \sum_{\lambda\in\Lambda} \sum_{k=1}^{k_\lambda} B_{\lambda,k} e^{\lambda k} \sum_{j=0}^\infty e^{\lambda(n-j)} \E[\varphi](j) \sum_{l=0}^{k-1} (-1)^{k-l}  \binom{j}{l} \binom{n+k-l-1}{k-l-1} + O\big(e^{(\theta + \eps)n}\big).
	\end{split}
\end{equation}
This expansion can likewise be written in the form~\eqref{eq:meanAsymptotics}.
Since we were only able to do so with considerable effort,
we have relegated the corresponding calculation to Appendix~\ref{sec:mean expansion lattice}.

Although the results of \cite{Kolesko+al:2025b} were stated only for characteristics vanishing on the negative half-line, one would expect them to hold also for general characteristics, possibly under additional mild assumptions, as is the case for the single-type process
\cite[Lemmas 7.1 and 7.6]{Iksanov+al:2024}. Therefore, we state Theorem~\ref{thm:main} under the assumption that~\eqref{eq:meanAs1} and~\eqref{eq:meanAsymptotics-vecs} hold for a characteristic $\varphi$ not necessarily vanishing on the negative half-line.

\begin{rem}\label{rem:meanAsympMTG}
Observe that using Lemma~\ref{lem:mtgsasCMJ} and~\eqref{eq:mtg2}, Equation~\eqref{eq:meanAsymptotics} may be rewritten as
\begin{equation*}
	\begin{split}
	m_t^\varphi
	&= \sum_{\lambda\in\Lambda} \int (\be_1^\trans \exp(\lambda, t-s, k_\lambda) \otimes \Ip) \Ev[W_t(\lambda)] \E[\varphi](s) \, \ell(\dd s) + r(t) \\
	&= \sum_{\lambda \in\Lambda} \int (\be_1^\trans \exp(\lambda, -s, k_\lambda) \otimes \Ip) \Ev[\cZ_t^{\phi_\lambda}] \E[\varphi](s) \, \ell(\dd s) + r(t)
	\end{split}
\end{equation*}
as $t \to \infty$, $t \in \G$.
\end{rem}

\section{Proofs of the main results}	\label{sec:core}

In this section we prove a series of auxiliary results that lead to the proof of Theorem~\ref{thm:main}.
We adapt the method pioneered in \cite{Iksanov+al:2024}, that is, we first prove convergence in distribution for
CMJ processes counted with centered characteristic (Theorem~\ref{thm:centeredChar}).
The main tool here is the martingale central limit theorem.
Then, in Lemma~\ref{lem:detRecentered}, we show how a process counted with deterministic characteristic may be decomposed into two parts,
one being its mean and the second being a CMJ process counted with centered characteristic.
These two results are used in Theorem~\ref{thm:smallMean},
which covers the convergence of CMJ processes with slowly growing mean.
Finally, the proof of Theorem~\ref{thm:main} is presented in Section~\ref{sec:main-proof}.

\subsection{Centered characteristics}

The following lemma is the multi-type version of Lemma 6.1 in \cite{Iksanov+al:2024}.

\begin{lem}\label{lem:polyCorr}
	Assume that~\ref{assumpt:Malthusian alpha} through~\ref{assumpt:hreprMoments} hold. 
	Let $\theta \geq 0$ and $f : [0,\infty) \to [0,\infty)^p$ be a continuous function
	such that $f(x) = O(x^\theta)$ as $x \to \infty$ and $x^{-\theta}f(x)$ is uniformly continuous on $[1,\infty)$. Suppose that the limit
	$$ \lim_{t\to\infty, t\in\G} \frac1{t^{\theta+1}} \int_{[0,t]} \bu^\trans f(s) \, \ell(\dd s) \eqqcolon c \in (0,\infty)	$$
	exists. Let $\varphi(t) = e^{\alpha t} f(t) \1_{[0,\infty)}(t)$, $t \in \R$. Then
	$$ \sup_{t \geq 1} \big( e^{-\alpha t} t^{-\theta -1} \E[\cZ_t^\varphi]\big) < \infty$$
	and
	$$ \frac{e^{-\alpha t}}{t^{\theta+1}} \cZ^\varphi_t \to \frac{c W}\beta \quad \P\text{-a.\,s.\ as } t\to\infty,\,t\in\G.$$
\end{lem}

\begin{proof}
The proof of \cite[Lemma 6.1]{Iksanov+al:2024} extends to the multi-type setting considered here with only minor modifications.
The most relevant changes concern the application of limit theorems for CMJ processes.
For instance, in \cite[Lemma 6.1]{Iksanov+al:2024}, in the lattice case, \cite[Corollary 3.1(a)]{Meiners:2010} is used
to conclude that, for given $\varepsilon > 0$, with probability one, there is a (possibly random) $r \in \N$
such that, for all $k \in \N$, $k \geq r$,
\begin{equation*}
(1-\varepsilon) \frac{W}{\beta} \leq e^{-\alpha k} N(\{k\}) (1-\varepsilon) \frac{W}{\beta}
\end{equation*}
where $N(\{k\})$ is the number of individuals born at time $k$ in the single-type process.
In the present multi-type setting, writing $N_j(\{k\})$ for the number of type-$j$ individuals born at time $k$,
we may use Proposition~\ref{thm:asconv}
to conclude that, with $\P$-probability one, there is a (possibly random) $r \in \N$
such that, for all $k \in \N$, $k \geq r$, we have
\begin{equation*}
(1-\varepsilon) \frac{W}{\beta} \bu_j \leq e^{-\alpha k} N_j(\{k\}) \leq (1+\varepsilon) \frac{W}{\beta} \bu_j.
\end{equation*}
The following computation in the proof of \cite[Lemma 6.1]{Iksanov+al:2024} then naturally extends as follows.
For $t \in \N$, $t \geq r$, one infers
\begin{multline*}
\frac{e^{-\alpha t}}{t^{\theta + 1}} \sum_{\substack{u \in \Tree, \\ r \leq S(u) \leq t}} e^{\alpha(t-S(u))} \be_{\tau(u)}^\trans f(t-S(u))
= \frac{1}{t^{\theta + 1}} \sum_{j=1}^p \sum_{k=r}^t e^{-\alpha k} \be_{j}^\trans f(t-S(u)) N_j(\{k\})	\\
\leq \frac{1+\varepsilon}{t^{\theta + 1}} \frac{W}{\beta} \sum_{j=1}^p \sum_{k=r}^t \bu_j \be_{j}^\trans f(t-k)
= \frac{1+\varepsilon}{t^{\theta + 1}} \frac{W}{\beta} \sum_{k=0}^{t-r}  \bu^\trans f(t-k) \to (1+\varepsilon) \frac{c W}{\beta}
\end{multline*}
$\P$-a.\,s.\ as $t \to \infty$. The non-lattice case can be dealt with in an analogous manner.
We refrain from providing further details.
\end{proof}

\begin{thm}\label{thm:centeredChar}
Assume~\ref{assumpt:Malthusian alpha},~\ref{assumpt:hreprMoments} and let $\cN$ be a standard normal variable independent of
$\F = \sigma(\tau(\varnothing),\, \pi_u : u\in\UHTree)$.
	\begin{enumerate}[\normalfont(i), wide]
		\item Assume that the characteristic $\chi$ is centered and satisfies~\ref{assumpt:dRivar}. Then
		\begin{equation*}
			e^{-\frac{\alpha}{2} t} \cZ_t^\chi
			\tod \bigg( \frac{W}{\beta} \int_\G e^{-\alpha s} \bu^\trans \E[\chi^2](s) \, \ell(\dd s)\bigg)^{\!1/2} \cN \quad \text{as } t\to\infty,\,t\in\G.
		\end{equation*}
		\item Assume that the characteristic $\chi$ is centered
		and that there exist $\theta \geq 0$ and a function $f$ (not vanishing identically on $\G$)
		satisfying the conditions of Lemma~\ref{lem:polyCorr} such that $\E[\chi^2(t)] = e^{\alpha t} f(t) \1_{[0,\infty)(t)}$
		and
		\begin{equation}\label{eq:centCharii}
			\E\big[ \be_i^\trans \chi^2(t) \1_{\{\be_i^\trans \chi^2(t) > \eps e^{\alpha t} t^{\theta+1}\}} \big]
			= o\big(t^\theta e^{\alpha t}\big) \quad \text{as } t \to \infty, \text{ for every } \eps>0, \, i\in[p].
		\end{equation}
		Then
		\begin{equation*}
		\bigg(\int_{[0,t]} e^{-\alpha s} \bu^\trans \E[\chi^2](s) \, \ell(\dd s)\bigg)^{-1/2} e^{-\frac\alpha2t} \cZ_t^\chi
		\tod \Big(\frac{W}\beta\Big)^{\!1/2} \cN \quad \text{as } t\to\infty,\;t\in\G.
		\end{equation*}
	\end{enumerate}
\end{thm}

\begin{proof}
	Observe that in case (i), if $\E[\chi^2](t) = 0$ for a.\,e.\ $t\in\G$, then the statement is trivial.
	We may thus assume that in both cases, $\E[\chi^2]$ does not vanish identically on $\G$.
	Let
	\begin{equation*}
		c_t = \begin{cases}
			e^{-\frac\alpha2t} \Big( \frac{1}{\beta} \int_\G e^{-\alpha s} \bu^\trans \E[\chi^2](s) \, \ell(\dd s)\Big)^{\!-1/2} \quad& \text{in case (i),} \\
			e^{-\frac\alpha2t} \Big( \frac{1}{\beta} \int_{[0,t]} e^{-\alpha s} \bu^\trans \E[\chi^2](s) \, \ell(\dd s)\Big)^{\!-1/2} &\text{in case (ii).}
			\end{cases}
	\end{equation*}
	Then $c_t$ is well-defined (in case (ii) for sufficiently large $t$)
	since all entries of the vector $\bu$ are positive. Fix an admissible ordering $(u_1, u_2, \dots)$ of $\UHTree$ and consider
	$$M_n(t) = c_t \sum_{k=1}^n \be_{\tau(u_k)}^\trans \chi_{u_k}(t - S(u_k)).$$
	By Lemma~\ref{lem:mtg}, for every $t\in\G$, $(M_n(t))_{n\in\N_0}$ is a martingale convergent a.\,s.\ and in $L^2$
	to $M(t) \coloneqq c_t \cZ_t^\chi$.
	Fix an increasing sequence $t_n \to \infty$ and let $(k_n)_{n\in\N}$ be increasing such that
	$$ \E^i[(M(t_n) - M_{k_n}(t_n))^2] < 2^{-n} \quad \text{for all } i\in[p]. $$
	In particular, $M_{k_n}(t_n) - M(t_n) \to 0$ in probability.
	By Slutsky's theorem, in order to prove the claimed convergence, it is enough to show that
	\begin{equation}\label{eq:Mknconv}
		M_{k_n}(t_n) \tod W^{1/2} \cN
	\end{equation}
	under $\P$. By the martingale convergence theorem, see e.g.\ \cite[Corollary 3.1 on page 58]{Hall+Heyde:1980},
	\eqref{eq:Mknconv} follows from
	\begin{align}
		c_{t_n}^2 &\sum_{k=1}^{k_n} \E\big[ \be_{\tau(u_k)}^\trans \chi_{u_k}^2(t_n\!-\!S(u_k)) \,\big|\, \cG_{k-1} \big] \topr W, \label{eq:convmtg1} \\
		c_{t_n}^2 &\sum_{k=1}^{k_n} \E \big[ \be_{\tau(u_k)}^\trans \chi_{u_k}^2(t_n\!-\!S(u_k)) \1_{\{c_{t_n}\be_{\tau(u_k)}^\trans|\chi_{u_k}(t_n - S(u_k))| > \eps\}}  \,\big|\, \cG_{k-1} \big] \topr 0 \quad \textnormal{for all } \eps>0 \label{eq:convmtg2}
	\end{align}
	where $\cG_k = \sigma(\pi_{u_j}: j=1,\ldots,k)$, $k \in \N_0$.
	Observe that, since the increments of $(M_k(t_n))_{k\in\N_0}$ are uncorrelated,
	\begin{equation*}
		c_{t_n}^2 \E\bigg[ \sum_{k=k_n+1}^\infty \E\big[ \be_{\tau(u_k)}^\trans \chi_{u_k}^2(t_n - S(u_k)) \,\big|\, \cG_{k-1} \big] \bigg]
		= \E[(M(t_n) - M_{k_n}(t_n))^2] \leq 2^{-n}.
	\end{equation*}
	Therefore,~\eqref{eq:convmtg1} is equivalent to
	\begin{equation}\label{eq:convmtg1a}
		c_{t_n}^2 \sum_{k=1}^\infty \E\big[ \be_{\tau(u_k)}^\trans \chi_{u_k}^2(t_n - S(u_k)) \,\big|\, \cG_{k-1} \big] \topr W.
	\end{equation}
	The left-hand side may be rewritten as
	\begin{equation*}
	c_{t_n}^2 \sum_{k=1}^\infty \be_{\tau(u_k)}^\trans \E[\chi^2] (t - S(u_k))
	=  c_{t_n}^2\cZ_{t_n}^{\E[\chi^2]}	\quad	\P\text{-a.\,s.}
	\end{equation*}
	Thus,~\eqref{eq:convmtg1a} follows from Proposition~\ref{thm:asconv} in case (i) and from Lemma~\ref{lem:polyCorr} in case (ii).
	To show~\eqref{eq:convmtg2}, let, for $r > 0$ and $t \in \G$, $v_2(r,t) \in \R^p$ be given by
	$\be_i^\trans v_2(r,t) = \E[\be_i^\trans \chi^2(t) \1_{\{\be_i^\trans |\chi(t)| > r\}}]$ for $i \in[p]$. Then, in case (i),
	\begin{equation*}
		\begin{split}
			\limsup_{n\to\infty} c_{t_n}^2
			&\sum_{k=1}^{k_n} \E \big[ \be_{\tau(u_k)}^\trans \chi_{u_k}^2(t_n - S(u_k))
			\1_{\{c_{t_n}\be_{\tau(u_k)}^\trans|\chi_{u_k}(t_n - S(u_j))| > \eps\}}  \,\big|\, \cG_{k-1} \big] \\
			&= \limsup_{n\to\infty} c_{t_n}^2 \sum_{k=1}^{k_n} \be_{\tau(u_k)}^\trans v_2(c_{t_n}^{-1} \eps,t_n - S(u_k)) \\
			&\leq \bigg(\frac{1}{\beta} \int_\G e^{-\alpha s} \bu^\trans \E[\chi^2](s) \, \ell(\dd s)\bigg)^{-1}
			\liminf_{r\to\infty} \limsup_{n\to\infty} e^{-\alpha t_n} \cZ_{t_n}^{v_2(r,\cdot)}.
		\end{split}	
	\end{equation*}
	Using first Proposition~\ref{thm:asconv} and then the dominated convergence theorem, we get
	\begin{equation*}
		\liminf_{r \to \infty} \limsup_{n\to\infty} e^{-\alpha t_n} \cZ_{t_n}^{v_2(r,\cdot)}
		= \liminf_{r \to \infty} \frac{W}{\beta} \int_\G e^{-\alpha s} \sum_{i=1}^p \bu_i
		\E\big[\be_i^\trans \chi^2(s)\1_{\{\be_i^\trans |\chi(s)| > r\}}\big] \, \ell(\dd s) = 0
	\end{equation*}
	in probability, which proves~\eqref{eq:convmtg2} in case (i). To prove the statement in case (ii), fix $\eps,\delta>0$. By~\eqref{eq:centCharii}, there exists $T>0$ such that
	\begin{equation*}
		v_2(\eps e^{\alpha t/2} t^{(\theta+1)/2},t) = \E\big[ \be_i^\trans \chi^2(t) \1_{\{\be_i^\trans \chi(t)^2 > \eps^2 e^{\alpha t} t^{\theta+1}\}}\big] \leq \delta t^\theta e^{\alpha t}  \quad \text{for all } i\in[p], \, t \geq T.
	\end{equation*}
	Therefore,
	\begin{equation*}
		\begin{split}
			\limsup_{n\to\infty} & \frac{e^{-\alpha t_n}}{t_n^{\theta+1}}
			\sum_{k=1}^{k_n} \E \big[ \be_{\tau(u_k)}^\trans \chi_{u_k}^2(t_n - S(u_k)) \1_{\{\be_{\tau(u_k)}^\trans|\chi_{u_k}(t_n - S(u_k))| > \eps e^{\alpha t/2} t^{(\theta+1)/2}\}}  \,\big|\, \cG_{k-1} \big] \\
			&\leq \limsup_{n\to\infty} \frac{e^{-\alpha t_n}}{t_n^{\theta+1}} \sum_{u\in\Tree} \be_{\tau(u)}^\trans v_2(\eps e^{\alpha t/2} t^{(\theta+1)/2},t_n-S(u)) \\
			&\leq \delta \limsup_{n\to\infty} \frac{e^{-\alpha t_n}}{t_n^{\theta+1}}
			\sum_{u\in\Tree} (t_n - S(u))^\theta e^{\alpha(t_n - S(u))} \1_{\{t_n - S(u) \geq T\}} \\
			&+C_T \limsup_{n\to\infty} \frac{e^{-\alpha t_n}}{t_n^{\theta+1}} \sum_{u\in\Tree}  \1_{\{0 \leq t_n - S(u) < T\}},
		\end{split}
	\end{equation*}
	where $C_T = \max_{i\in[p]} \sup_{t \in [0,T]} \be_i^\trans \E[\chi^2](t)$. Observe that
	\begin{multline*}
		\frac{e^{-\alpha t_n}}{t_n^{\theta+1}} \sum_{u\in\Tree} (t_n - S(u))^\theta e^{\alpha(t_n - S(u))} \1_{\{t_n - S(u) \geq T\}}
		 \\ \leq \frac{e^{-\alpha t_n}}{t_n^{\theta+1}} \sum_{u\in\Tree} (t_n - S(u))^\theta e^{\alpha(t_n - S(u))} \1_{\{t_n - S(u) \geq 0\}}
		 \to \frac{\|\bu\|_1 W}{(\theta+1)\beta}
	\end{multline*}
	almost surely by Lemma~\ref{lem:polyCorr} with $f(t) = t^\theta \sum_{i=1}^p \be_i$, while
	\begin{equation*}
		\frac{e^{-\alpha t_n}}{t_n^{\theta+1}} \sum_{u\in\Tree} \1_{\{0 \leq t_n - S(u) < T\}} \to 0
	\end{equation*}
	in probability by Proposition~\ref{thm:asconv} applied to the characteristic $t \mapsto \1_{[0,T)}(t) \sum_{i=1}^p \be_i$. Since $\delta > 0$ is arbitrary, this implies
	\begin{equation*}
		\lim_{n\to\infty} \frac{e^{-\alpha t_n}}{t_n^{\theta+1}} \sum_{k=1}^{k_n} \E \big[ \be_{\tau(u_k)}^\trans \chi_{u_k}^2(t_n - S(u_k))
		\1_{\{\be_{\tau(u_k)}^\trans|\chi_{u_k}(t_n - S(u_k))| > \eps e^{\alpha t/2} t^{(\theta+1)/2}\}}  \,\big|\, \cG_{k-1} \big] = 0	\quad \text{in }	\P.
	\end{equation*}
	This completes the proof of~\eqref{eq:convmtg2} in case (ii) since $c_t^2 \sim (\beta/c) e^{-\alpha t} t^{-\theta-1}$ as $t\to\infty$.
\end{proof}

\subsection{Deterministic characteristic}

Let
\begin{equation}
	\bV = \sum_{n=0}^\infty \bmu^{*n}.
\end{equation}
Then $\bV$ satisfies the renewal equation
\begin{equation}\label{eq:Vrenewal}
	\bV = \Ip \delta_0 + \bmu * \bV.
\end{equation}
For a deterministic function $f : \R \to \R^p$ consider the characteristic $\chi_f$ given by
\begin{equation}
	\chi_f = \bxi * \bV * f - \bmu* \bV * f.
\end{equation}

\begin{lem}\label{lem:detRecentered}
	Assume~\ref{assumpt:Malthusian alpha}. Let $f : \R \to \R^p$ be a deterministic c\`adl\`ag characteristic satisfying~\ref{assumpt:dRimean}. Then
	\begin{enumerate}[\normalfont(i)]
		\item $\chi_f$ is a well-defined, c\`adl\`ag, centered characteristic.
		\item \label{item:chif1} If~\ref{assumpt:hreprMoments} holds and the function $t \mapsto m_t^f e^{-\frac\alpha2 t} (1 + t^2)$ is bounded,
		then $\chi_f$ satisfies~\ref{assumpt:dRimean} through~\ref{assumpt:local ui}.
		\item If the assumptions of~\ref{item:chif1} hold or $f$ is supported on $[0,\infty)$, then	
		\begin{equation}\label{eq:chif}
			\cZ_t^{\chi_f} = \cZ_t^f - \be_{\tau(\varnothing)}^\trans m_t^f.
		\end{equation}
	\end{enumerate}
\end{lem}
\begin{proof}
\begin{enumerate}[\normalfont(i), wide]
	\item Observe that~\ref{assumpt:dRimean} implies that $e^{-\alpha t}|f|(t)$ is directly Riemann integrable.
	In particular, Lemma~\ref{lem:meanviarenewal} implies that, for any $t\in\G$, $\cZ_t^{|f|}$ is integrable.
	We may thus use Fubini's theorem to obtain 
	\begin{equation}\label{eq:V*f}
		m_t^f = \Ev \big[\cZ_t^f\big] = \Ev \bigg[ \sum_{n=0}^\infty \sum_{|u| = n} \be_{\tau(u)}^\trans f(t - S(u)) \bigg] = \sum_{n=0}^\infty \bmu^{*n} * f(t) = \bV*f(t).
	\end{equation}
	A similar calculation gives $\bV * |f|(t) = m_t^{|f|} < \infty$. Since Equation~\eqref{eq:Vrenewal} implies $\bmu * \bV \leq \bV$, we have $\bmu*\bV*|f| \leq \bV*|f| < \infty$, which means that the characteristic $\chi_f$ is well-defined and centered. To show that it is c\`adl\`ag, we may repeat the argument from \cite[Lemma 6.4]{Iksanov+al:2024}, which in turn uses Proposition~\ref{prop:char}.
	
	\item Note that $\chi_f$ satisfies~\ref{assumpt:dRimean} trivially since it is centered.
	Further observe that, for any $t\in\G$ and $i \in [p]$, by Jensen's inequality,
	\begin{equation*}
	\begin{split}
	\E[|\be_i^\trans \chi_f|^2](t)
	&\leq \E[|\be_i^\trans \bxi*m^f(t)|^2] + 2|\be_i^\trans \bmu*m^f(t)|\E[|\be_i^\trans \bxi*m^f(t)|] + |\be_i^\trans \bmu*m^f(t)|^2	\\
	&\leq 4\E[|\be_i^\trans \bxi*m^f(t)|^2].
	\end{split}
	\end{equation*}
	Thus, by Proposition~\ref{prop:char}\ref{item:dRivar conds}, it is enough to show that
	\begin{equation*}
		\int e^{-\alpha x} \E\Big[\sup_{|t-x|\leq 1}|\be_i^\trans \bxi*m^f(t)|^2 \Big] \, \dd x < \infty 
	\end{equation*}
	for all $i \in[p]$.
	Note that $\be_i^\trans \bxi*m^f(t)$ is of the form
	\begin{equation*}
	\be_i^\trans \bxi*m^f(t) = \sum_{j=1}^p \int \be_j^\trans m^f_{t-s} \, \xi^{ij}(\dd s).
	\end{equation*}
	By assumption, there exists $C$ such that $\be_j^\trans |m_t^f| \leq Ce^{\frac\alpha2 t}(1 + t^2)^{-1}$ for every $j\in[p]$.
	Hence,
	\begin{equation*}
	|\be_i^\trans \bxi * m^f(t)| \leq C \sum_{j=1}^p \int \frac{ e^{\frac\alpha2 (t-s)}}{1+(t-s)^2} \, \xi^{ij}(\dd s).
	\end{equation*}
	Further, note that for any $s,x,t \in \R$ such that $|t-x| \leq 1$,
	\begin{equation*}
		1 + (x-s)^2 \leq 1 + 2(t-s)^2 + 2(x-t)^2 \leq 3(1 + (t-s)^2).
	\end{equation*}
	Consequently,
	\begin{multline*}
		\int e^{-\alpha x} \E\Big[\sup_{|t-x|\leq 1}|\be_i^\trans \bxi*m^f(t)|^2 \Big] \, \dd x \\
		\begin{aligned}
		&\leq 9C^2 e^\alpha \E\bigg[ \sum_{j,k=1}^p \iiint \frac{1}{(1 + (x-s)^2)(1 + (x-r)^2)}
		\, \dd x \, e^{-\frac\alpha2s} e^{-\frac\alpha2r} \, \xi^{ij}(\dd s) \, \xi^{ik}(\dd r)\bigg] \\
		&\leq 9 C^2 e^\alpha \pi \E\bigg[ \bigg(\sum_{j=1}^p \int e^{-\frac\alpha2s} \, \xi^{ij}(\dd s) \! \bigg)^{\!\!2}\bigg] < \infty
		\end{aligned}
	\end{multline*}
	by~\ref{assumpt:hreprMoments}.
	\item
	Under the assumptions of~\ref{item:chif1} $\chi_f$ satisfies~\ref{assumpt:dRimean} and~\ref{assumpt:dRivar}.
	Moreover, $f$ satisfies~\ref{assumpt:dRimean} by assumption and~\ref{assumpt:dRivar} trivially.
	Therefore, Lemma~\ref{lem:well-def} implies that $\cZ_t^f$ and $\cZ_t^{\chi_f}$ are well-defined
	as limits in the a.\,s.\ sense (under admissible orderings) and unconditionally in $L^1$.
	
	We first provide a symbolic calculation suggesting that~\eqref{eq:chif} is correct.
	Note that for any $u \in \Tree$ and function $g : \R \to \R^p$ (satisfying suitable assumptions), $\bxi_u * g(s)$
	is a column vector with entries
	\begin{equation*}
	\be_i^\trans \bxi_u * g(s) = \sum_{j=1}^p \sum_{k=1}^{N_u^{ij}} \be_j^\trans g\big(s - X_{uk}^{ij}\big),	\quad i \in [p].
	\end{equation*}
	In particular,
	\begin{equation*}
	\begin{split}
	\be_{\tau(u)}^\trans \bxi_u * g(t - S(u))
	&= \sum_{j=1}^p \sum_{k=1}^{N_u^{\tau(u),j}} \be_j^\trans g\big(t - S(u) - X_{uk}^{\tau(u),j}\big)
	= \sum_{k=1}^{N_u} \be_{\tau(uk)}^\trans g(t - S(uk)).
	\end{split}
	\end{equation*}
	Therefore,
	\begin{equation*}
	\begin{split}
		\cZ_t^{\chi_f} &= \sum_{u\in\Tree} \be_{\tau(u)}^\trans \bxi_u*\bV*f(t-S(u)) - \sum_{u\in\Tree} \be_{\tau(u)}^\trans \bmu* \bV * f(t-S(u)) \\
		&= \sum_{u\in\Tree\setminus\{\varnothing\}} \be_{\tau(u)}^\trans \bV*f(t-S(u)) - \sum_{u\in\Tree} \be_{\tau(u)}^\trans \bmu* \bV * f(t-S(u)) \\
		& = \sum_{u\in\Tree} \be_{\tau(u)}^\trans (\bV - \bmu * \bV) * f(t - S(u)) - \be_{\tau(\varnothing)}^\trans \bV * f(t) \\
		& = \cZ_t^f - \be_{\tau(\varnothing)}^\trans m_t^f,
	\end{split}
	\end{equation*}
	where the last equality follows from~\eqref{eq:Vrenewal} and~\eqref{eq:V*f}.
	For a rigorous proof, note that, for any $n \in \N$,
	\begin{equation}\label{eq:V*fest}
		\begin{split}
			\Ev \bigg[\sum_{|u| = n} |\be_{\tau(u)}^\trans \bV * f(t - S(u)) |\bigg]
			&\leq \Ev \bigg[ \sum_{|u| = n} \be_{\tau(u)}^\trans \bV * |f|(t - S(u)) \bigg] \\
			&= \Ev \bigg[ \sum_{|u| \geq n} \be_{\tau(u)}^\trans |f|(t - S(u)) \bigg] \leq m_t^{|f|}.
		\end{split}
	\end{equation}
	Thus, the sum is finite a.\,s.\ under the assumptions of~\ref{item:chif1}
	since the assumptions then also apply to $|f|$ by the definition of direct Riemann integrability.
	The same is true if $f$ vanishes on the negative half-line since then also $\chi_f$ vanishes there.
	In this case, consequently, all the sums have only finitely many non-zero terms a.\,s.
	We may repeat the symbolic calculation above for the truncated sum to obtain 
	\begin{equation*}
	\begin{split}
	\sum_{0\leq|u|\leq n} \!\!  \be_{\tau(u)}^\trans \chi_{f,u} (t \!-\! S(u))
	= \sum_{0 \leq |u| \leq n} \!\! \be_{\tau(u)}^\trans f(t \!-\! S(u)) - \be_{\tau(\varnothing)}^\trans m^f_t 
	+ \sum_{|u| = n+1} \!\! \be_{\tau(u)}^\trans \bV * f(t \!-\! S(u))
	\end{split}
	\end{equation*}
	a.\,s.\ for any $n\in\N$. Note that by~\eqref{eq:V*fest} and the dominated convergence theorem, the last summand converges to $0$ in $L^1$\ as $n\to\infty$, which implies~\eqref{eq:chif}.
\end{enumerate}
\end{proof}

\subsection{Slowly growing mean}

\begin{thm}\label{thm:smallMean}
	Assume that~\ref{assumpt:Malthusian alpha} and~\ref{assumpt:hreprMoments} hold
	and let $\varphi$ be a characteristic satisfying~\ref{assumpt:dRimean} through~\ref{assumpt:local ui}.
	If the function $t \mapsto m_t^\varphi e^{-\frac\alpha2 t} (1 + t^2)$ is bounded, then
	\begin{equation}\label{eq:smallmeanconv}
	e^{-\frac\alpha2t} \cZ_t^\varphi \tod \sigma_\varphi \Big(\frac{W}{\beta}\Big)^{\! 1/2} \cN \quad \text{as } t\to\infty,\;t\in\G
	\end{equation}
	under $\P$ where
	\begin{equation*}
	\sigma_\varphi^2 = \int_\G e^{-\alpha s} \bu^\trans \Var[\bxi*m^\varphi+\varphi](s) \, \ell(\dd s)
	\end{equation*}
	and $\cN$ is a standard normal variable independent of $W$.
	
	Moreover, $\sigma_\varphi = 0$ if and only if, for every $t\in\G$, $t\geq 0$, $\cZ_t^\varphi = \be_{\tau(\varnothing)}^\trans m^\varphi_t$ a.\,s.
\end{thm}
\begin{proof}
	Decompose $\varphi = (\varphi - \E[\varphi]) + \E[\varphi]$. Since $m_t^\varphi = m_t^{\E[\varphi]}$, we have
	\begin{equation*}
	\sup_{t \in \G} m_t^{\E[\varphi]} e^{-\frac\alpha2 t} (1 + t^2) = \sup_{t \in \G} m_t^\varphi e^{-\frac\alpha2 t} (1 + t^2).
	\end{equation*}
	Further, $\E[\varphi]$ is a deterministic characteristic. Consequently, Lemma~\ref{lem:detRecentered} applies to
	\begin{equation*}
	\chi \coloneqq \chi_{\E[\varphi]} = \bxi * \bV * \E[\varphi] - \bmu * \bV * \E[\varphi].
	\end{equation*}
	The lemma implies that $\chi$ is well-defined, centered and satisfies~\ref{assumpt:dRivar} and~\ref{assumpt:local ui}.
	Moreover, by~\eqref{eq:chif}, we have
	\begin{equation}\label{eq:chismall}
	\cZ_t^\varphi = \cZ_t^{\varphi - \E[\varphi]} + \cZ_t^{\chi} + \be_{\tau(\varnothing)}^\trans m_t^\varphi
	= \cZ_t^{\varphi - \E[\varphi] + \chi} + \be_{\tau(\varnothing)}^\trans m_t^\varphi.
	\end{equation}
	By assumption, $e^{-\frac\alpha2t} m_t^\varphi \to 0$ as $t\to\infty$.
	Observe that the characteristic $\varphi - \E[\varphi] + \chi$ is centered and satisfies~\ref{assumpt:dRivar}
	since $\chi$ and $\varphi - \E[\varphi]$ satisfy~\ref{assumpt:dRivar} and~\ref{assumpt:local ui}.
	Hence,~\eqref{eq:smallmeanconv} follows from Theorem~\ref{thm:centeredChar}.
	
	Finally, note that, since all entries of the vector $\bu$ are positive, $\sigma_\varphi = 0$ if and only if for a.\,e.~$t\in\G$,
	$\varphi(t) - \E[\varphi](t) + \chi(t) = 0$. Since the characteristics $\chi$ and $\varphi - \E[\varphi]$ have a.\,s.\ c\`adl\`ag paths,
	this is equivalent to $\varphi - \E[\varphi] + \chi = 0$ on $\G$ a.\,s.
	In view of~\eqref{eq:chismall}, this means $\cZ_t^\varphi = \be_{\tau(\varnothing)}^\trans m^\varphi_t$ a.\,s.
\end{proof}

\subsection{General characteristic: proof of Theorem~\ref{thm:main}}\label{sec:main-proof}

\begin{proof}[Proof of Theorem~\ref{thm:main}]
Observe that $H_\Lambda, H_\dLambda$ may be written in the form
\begin{align*}
H_\Lambda(t)
&= \sum_{\lambda\in \intLambda} \int (\be_1^\trans \exp(\lambda, t-s, k_\lambda) \otimes \Ip) W(\lambda) \E[\varphi](s) \, \ell(\dd s) \1_{[0,\infty)}(t), \\
H_\dLambda(t)
&= \sum_{\lambda\in \dLambda} \int (\be_1^\trans \exp(\lambda, t-s, k_\lambda) \otimes \be_{\tau(\varnothing)}^\trans)  \Av_\lambda \E[\varphi](s) \,\ell(\dd s) \1_{[0,\infty)}(t).
\end{align*}
In particular, Lemmas~\ref{lem:mtgsasCMJ} and~\ref{lem:mtgLasCMJ} imply that, for $t \geq 0$,
\begin{equation*}
\Ev[H_\Lambda](t)
= \sum_{\lambda \in \intLambda} \int (\be_1^\trans \exp(\lambda, t-s, k_\lambda) \otimes \Ip) \Av_\lambda \E[\varphi](s) \, \ell(\dd s),
\end{equation*}
and
\begin{equation*}
\begin{split}
H_\Lambda(t)
= \sum_{\lambda \in \intLambda}\int (\be_1^\trans \exp(\lambda, -s, k_\lambda) \otimes \Ip) \cZ_t^{\phi_\lambda + \chi_\lambda} \E[\varphi](s) \, \ell(\dd s)
= \cZ_t^{\psi_\Lambda} 
\end{split}
\end{equation*}
for characteristics
\begin{align}
\psi_\lambda(t) &= \int (\be_1^\trans \exp(\lambda, -s, k_\lambda) \otimes \Ip) (\phi_\lambda(t) + \chi_\lambda(t)) \E[\varphi](s) \, \ell(\dd s), \\
\psi_\Lambda(t) &= \sum_{\lambda\in\intLambda} \psi_\lambda(t).
\end{align}
Similarly, $H_\dLambda(t) = \be_{\tau(\varnothing)}^\trans m_t^{\E[\psi_\dLambda]}$ for the characteristics
\begin{align*}
\varphi_\lambda(t) &= \int (\be_1^\trans \exp(\lambda, -s, k_\lambda) \otimes \Ip) \phi_\lambda(t) \E[\varphi](s) \, \ell(\dd s), \\
\psi_\dLambda(t) &= \sum_{\lambda\in\dLambda} \varphi_\lambda(t).
\end{align*}
Observe that each entry of the matrix $\psi_\Lambda(t)$ is a linear combination of entries of $\phi_\lambda(t), \chi_\lambda(t)$
for $\lambda\in\intLambda$ with coefficients given by $\int e^{-\alpha s} \frac{(-s)^j}{j!} \E[\be_i^\trans\varphi](s) \, \ell(\dd s)$.
Similarly, the entries of $\psi_\dLambda(t)$ are linear combinations of entries of $\phi_\lambda(t)$ for $\lambda \in \dLambda$.
In particular, Lemma~\ref{lem:chilambda} implies that the characteristics $\psi_\Lambda$, $\psi_\dLambda$ satisfy~\ref{assumpt:dRimean} through~\ref{assumpt:local ui} and so does $\varphi - \psi_\Lambda - \psi_\dLambda$.
Moreover, since $\E[\psi_\dLambda]$ is deterministic and supported on $[0,\infty)$, Lemma~\ref{lem:detRecentered} implies that
\begin{equation*}
	H_\dLambda(t) = \be_{\tau(\varnothing)}^\trans m_t^{\E[\psi_\dLambda]} = \cZ_t^{\E[\psi_\dLambda]} - \cZ_t^{\chi_\dLambda},
\end{equation*}
where
\begin{equation*}
	\chi_\dLambda = \bxi*m^{\E[\psi_\dLambda]} - \bmu*m^{\E[\psi_\dLambda]}.
\end{equation*}
Observe that for any $\lambda \in \Lambda$, by Lemma~\ref{lem:mtgsasCMJ},
\begin{equation*}
m_t^{\phi_\lambda} = (\exp(\lambda, t, k_\lambda)\otimes \Ip) \Ev [W_t(\lambda)] \1_{[0,\infty)}(t)
= (\exp(\lambda, t, k_\lambda) \otimes \Ip) \Av_\lambda \1_{[0,\infty)}(t).
\end{equation*}
In particular,
\begin{equation}\label{eq:xi*mlambda}
	\begin{split}
		(\Ikl \otimes \bxi) * m^{\phi_\lambda}(t) &= \int  \1_{[0,\infty)}(t-x) \exp(\lambda, t-x, k_\lambda) \otimes \bxi(\dd x) \Av_\lambda \\
		&= \int \exp(\lambda,t-x,k_\lambda) \otimes \bxi(\dd x)\Av_\lambda \1_{[0,\infty)}(t) - \phi_\lambda(t)
	\end{split}
\end{equation}
and thus
\begin{multline}\label{eq:xi*psidLambda}
\bxi*m^{\E[\psi_\dLambda]}(t)
= \bxi*m^{\psi_\dLambda}(t) \\
= \sum_{\lambda \in\dLambda}\iint (\be_1^\trans \exp(\lambda,t-x-s,k_\lambda) \otimes \bxi(\dd x)) \Av_\lambda \E[\varphi](s)
\, \ell(\dd s) \1_{[0,\infty)}(t) - \psi_\dLambda(t).
\end{multline}
That is, $\chi_\dLambda$ may be decomposed into $\chi_\dLambda = \chi - (\psi_\dLambda - \E[\psi_\dLambda])$ where
\begin{equation*}
	\begin{split}
	\chi(t)
	&= \sum_{\lambda \in\dLambda}\iint (\be_1^\trans \exp(\lambda,t-x-s,k_\lambda) \otimes \bxi(\dd x)) \Av_\lambda \E[\varphi](s) \, \ell(\dd s) \1_{[0,\infty)}(t) \\
	&\hphantom{=}~ -\sum_{\lambda \in\dLambda}\iint (\be_1^\trans \exp(\lambda,t-x-s,k_\lambda) \otimes \bmu(\dd x)) \Av_\lambda \E[\varphi](s) \, \ell(\dd s) \1_{[0,\infty)}(t).
\end{split}
\end{equation*}
This leads to the decomposition
\begin{equation}\label{eq:generalDecomp}
	\cZ_t^\varphi - H_\Lambda(t) - H_\dLambda(t) = \cZ_t^{\varphi - \psi_\Lambda - \psi_\dLambda} + \cZ_t^\chi.
\end{equation}
Let us treat both terms separately. By~\eqref{eq:meanAsymptotics},
\begin{equation*}
	\E\big[ \cZ^{\varphi - \psi_\Lambda - \psi_\dLambda} \big] = r(t) = O\big(e^{\frac\alpha2t}/(1+t^2) \big),
\end{equation*}
and the characteristic $\varphi - \psi_\Lambda - \psi_\dLambda$ satisfies the assumptions of Theorem~\ref{thm:smallMean}.
Therefore,
\begin{equation}\label{eq:conv1}
	e^{-\frac\alpha2t} \cZ_t^{\varphi - \psi_\Lambda - \psi_\dLambda} \tod \sigma \Big( \frac{W}{\beta}\Big)^{\!1/2} \cN
\end{equation}
under $\P$, where
\begin{equation}\label{eq:sigma}
\sigma^2
= \int_\R e^{-\alpha s} \bu^\trans \Var[\bxi * m^{\varphi - \psi_\Lambda - \psi_\dLambda} + \varphi - \psi_\Lambda - \psi_\dLambda](s) \, \ell(\dd s).
\end{equation}
Moreover, as stated in Theorem~\ref{thm:smallMean}, $\sigma = 0$ if and only if $\cZ_t^{\varphi - \psi_\Lambda - \psi_\dLambda} = \be_{\tau(\varnothing)} r(t)$ a.\,s.\ for any $t \in \G$, $t \geq 0$.
Formula~\eqref{eq:sigma} may be further simplified using relation~\eqref{eq:xi*mlambda}. Namely, since each $\chi_\lambda$ is centered,
we obtain, for $\lambda \in \Lambda$,
\begin{align*}
\bxi * m^{\psi_\lambda}(t)
&= \int (\be_1^\trans \exp(\lambda,-s,k_\lambda) \otimes \Ip) \bxi*m^{\phi_\lambda}(t) \E[\varphi](s) \, \ell(\dd s) \\
&= \iint (\be_1^\trans \exp(\lambda,t-x-s,k_\lambda) \otimes \bxi(\dd x)) \Av_\lambda \E[\varphi](s) \, \ell(\dd s) \1_{[0,\infty)}(t) \\
&\hphantom{=}~- \psi_\lambda(t) + \int (\be_1^\trans \exp(\lambda, -s, k_\lambda) \otimes \Ip) \chi_\lambda(t) \E[\varphi](s) \, \ell(\dd s).
\end{align*}
Further, the last term equals
\begin{align*}
			\int& (\be_1^\trans \exp(\lambda, -s, k_\lambda) \otimes \Ip) \chi_\lambda(t) \E[\varphi](s) \, \ell(\dd s) \\
			&= \iint (\be_1^\trans \exp(\lambda,t-x-s,k_\lambda) \otimes \bxi(\dd x)) \Av_\lambda \E[\varphi](s) \, \ell(\dd s) \1_{(-\infty,0)}(t) \\
			&\hphantom{=}~- \int (\be_1^\trans \exp(\lambda, t-s,k_\lambda) \otimes \Ip) \Av_\lambda \E[\varphi](s) \, \ell(\dd s) \1_{(-\infty,0)}(t).
\end{align*}
Putting things together,
\begin{equation*}
	\begin{split}
	\bxi*m^{\psi_\Lambda}(t) + \psi_\Lambda(t)
	&= \sum_{\lambda \in \intLambda} \iint (\be_1^\trans \exp(\lambda,t-x-s,k_\lambda) \otimes \bxi(\dd x)) \Av_\lambda \E[\varphi](s) \, \ell(\dd s) \\
	&\hphantom{=}~ - \sum_{\lambda \in \intLambda} \int (\be_1^\trans \exp(\lambda, t-s,k_\lambda) \otimes \Ip) \Av_\lambda \E[\varphi](s) \, \ell(\dd s) \1_{(-\infty,0)}(t).
		\end{split}
	\end{equation*}
	Note that the second sum is deterministic. This together with~\eqref{eq:xi*psidLambda} gives
	\begin{equation*}
		\Var[\bxi * m^{\varphi - \psi_\Lambda - \psi_\dLambda} + \varphi - \psi_\Lambda - \psi_\dLambda](t) = \Var[\bxi*f^{\varphi} + \varphi](t)
	\end{equation*}
	for
	\begin{align*}
		f^\varphi(t) &= m^\varphi(t) - \sum_{\lambda \in \intLambda} \int (\be_1^\trans \exp(\lambda,t-s,k_\lambda) \otimes \Ip) \Av_\lambda \E[\varphi](s) \, \ell(\dd s) \\
		&\hphantom{= m^\varphi(t) }- \sum_{\lambda\in\dLambda} \int (\be_1^\trans \exp(\lambda,t-s,k_\lambda) \otimes \Ip) \Av_\lambda \E[\varphi](s) \, \ell(\dd s) \1_{[0,\infty)}(t) \\
		&= r(t)\1_{[0,\infty)}(t) - \sum_{\lambda \in \intLambda} \int (\be_1^\trans \exp(\lambda,t-s,k_\lambda) \otimes \Ip) \Av_\lambda \E[\varphi](s) \, \ell(\dd s)\1_{(-\infty,0)}(t),
	\end{align*}
	which is equivalent to~\eqref{eq:hvarphi}.
	Let us now turn to the second term in~\eqref{eq:generalDecomp}. The characteristic $\chi$ may be represented as
	\begin{equation*}
		\chi(t) = \sum_{\lambda\in\dLambda} (R_\lambda(t) - \E[ R_\lambda(t)])\1_{[0,\infty)}(t)
	\end{equation*}
	for the random vectors
	\begin{equation*}
		R_\lambda(t) = \iint (\be_1^\trans \exp(\lambda,t-x-s,k_\lambda) \otimes \bxi(\dd x)) \Av_\lambda \E[\varphi](s) \, \ell(\dd s),
		\quad	\lambda \in \partial \Lambda.
	\end{equation*}
	Using the binomial formula, we get
\begin{equation}
\begin{split}
R_{\lambda}(t)
&= e^{\frac\alpha2t}\sum_{k=0}^{k_\lambda-1} e^{\imag \Im(\lambda) t} t^k
\sum_{m=0}^{k_\lambda-k-1} \binom{k+m}{k} \int e^{-\lambda x} (-x)^m \, \bxi(\dd x) \bb_{\lambda,k+m,\varphi} \\
&= e^{\frac\alpha2t} \sum_{k=0}^{k_\lambda-1} e^{\imag \Im(\lambda)t}t^k Y(\lambda,k).
\end{split}
\end{equation}

In particular, if $n=-1$, then $\chi$ equals the zero function a.\,s.\ and $\cZ_t^\varphi - H_\Lambda(t) - H_\dLambda(t) = \cZ_t^{\varphi - \psi_\Lambda - \psi_\dLambda}$ for all $t\geq 0$, a.\,s., which completes the proof in the case (i).

Otherwise, by \cite[Lemma 6.7]{Iksanov+al:2024}, $\chi$ satisfies the assumption of Theorem~\ref{thm:centeredChar}(ii) with $\theta = 2n$ and
\begin{equation*}
\frac{2n+1}{t^{2n+1}} \int_{[0,t]} e^{-\alpha s} \bu^\trans \Var[\chi](s) \, \ell(\dd s)
\to \sum_{i=1}^p \bu_i \sum_{\lambda\in\dLambda} \Var[Y^{(i)}(\lambda,n)] = \rho^2_n
\quad	\text{as } t \to \infty.
\end{equation*}
Therefore,
\begin{equation*}
\Big(\frac{e^{\alpha t} t^{2n+1} \rho_n^2}{2n+1}\Big)^{\!\!-1/2} \cZ_t^\chi \tod \Big(\frac{W}\beta\Big)^{\!1/2} \cN,
\end{equation*}
while~\eqref{eq:conv1} implies that $(e^{\alpha t} t^{2n+1})^{-1/2} \cZ_t^{\varphi-\psi_\Lambda-\psi_\dLambda}
\to 0$ in $\P$-probability as $t\to\infty$,
which completes the proof in case (ii).
\end{proof}

\begin{appendix}
	
	\section{Mean expansion in the lattice case}	\label{sec:mean expansion lattice}
	
	In this section, we present the calculation showing that the expansion
	\begin{equation*}	\tag{\ref{eq:mean expansion lattice}}
	m^\varphi_t
	= \sum_{\lambda\in\Lambda} \sum_{k=1}^{k_\lambda} B_{\lambda,k} e^{\lambda k} \sum_{s=0}^\infty e^{\lambda(t-s)} \E[\varphi](s) \sum_{l=0}^{k-1} (-1)^{k-l}  \binom{s}{l} \binom{t+k-l-1}{k-l-1} + O\big(e^{(\theta + \eps)t}\big)
	\end{equation*}
	for $t \to \infty$, $t \in \N_0$ can be written in the form~\eqref{eq:meanAsymptotics}.
	To this end, we first recall the following notation for the Stirling numbers of the first and second kind:
	\begin{equation*}
	\stirfst{n}{k}
	\qquad\text{and}\qquad
	\stirscd{n}{k},
	\end{equation*}
	the number of permutations of a set of $n$ elements having exactly $k$ cycles,
	and the number of partitions of a set of $n$ elements into exactly $k$ nonempty subsets, respectively.
	We shall use various formulae involving the Stirling numbers of the first and second kind
	and our main source for these is \cite[Section~6.1]{Graham+al:1994}.
	For instance, for $l,m,s,t \in \N_0$, using the notation $(s)_{(l)} \coloneqq s(s-1)\ldots(s-l+1)$ and $(t+1)^{(m)} = (t+1)(t+2)\ldots(t+m)$
	for the falling and rising factorial, respectively,
	we infer from \cite[Eqs.~(6.13) and (6.11)]{Graham+al:1994}
	\begin{align*}
	\binom{s}{l} &= \frac{(s)_{(l)}}{l!} = \frac1{l!}\sum_{j=0}^l (-1)^{l-j} \stirfst{l}{j} s^j,	\\
	\binom{t+m}{m} &= \frac{(t+1)^{(m)}}{m!} = \frac1{m!} \sum_{k=0}^m \stirfst{m}{k} (t+1)^k.
	\end{align*}
	This gives
	\begin{multline*}
	\sum_{l=0}^{k-1} (-1)^{k-l} \binom{s}{l} \binom{t+k-l-1}{k-l-1} \\
	\begin{aligned}
	&= \sum_{j=0}^{k-1} \sum_{i=0}^{k-1-j} \frac{(-1)^{k-j}}{(k-1)!} s^j (t+1)^i \sum_{l=j}^{k-1-i} \stirfst{l}{j} \stirfst{k-1-l}{i} \binom{k-1}{l} \\
	&= \sum_{j=0}^{k-1} \sum_{i=0}^{k-1-j} \frac{(-1)^{k-j}}{(k-1)!} s^j (t+1)^i \stirfst{k-1}{i+j} \binom{i+j}{j} \\
	&= \frac{(-1)^k}{(k-1)!} \sum_{j=0}^{k-1} \sum_{m=j}^{k-1} (-s)^j (t+1)^{m-j} \stirfst{k-1}{m} \binom{m}{j} \\
	&= \frac{(-1)^k}{(k-1)!} \sum_{m=0}^{k-1} \stirfst{k-1}{m} (t-s+1)^m \\
	&= \frac{(-1)^k}{(k-1)!} \sum_{j=0}^{k-1} (t-s)^j \sum_{m=j}^{k-1} \stirfst{k-1}{m} \binom{m}{j} \\
	&= \frac{(-1)^k}{(k-1)!} \sum_{j=0}^{k-1} \stirfst{k}{j+1} (t-s)^j,
	\end{aligned}
	\end{multline*}
	where in the third and in the last line, we have used \cite[Eq.~(6.29)]{Graham+al:1994} and \cite[Eq.~(6.16)]{Graham+al:1994}, respectively.
	Plugging the last expression into \eqref{eq:mean expansion lattice} yields
	\begin{equation*}
	\begin{split}
	m^\varphi_t
	&= \sum_{\lambda\in\Lambda} \sum_{k=1}^{k_\lambda} B_{\lambda,k} e^{\lambda k} \int e^{\lambda(t-s)} \E[\varphi](s) \frac{(-1)^k}{(k-1)!}
	\sum_{j=0}^{k-1} \stirfst{k}{j+1} (t-s)^j \, \ell(\dd s) + r(t) \\
	&= \sum_{\lambda\in\Lambda} \sum_{j=1}^{k_\lambda} \int e^{\lambda(t-s)} (t-s)^{j-1}
	\sum_{k=j}^{k_\lambda} \frac{(-1)^k}{(k-1)!} \stirfst{k}{j} B_{\lambda,k} e^{\lambda k} \E[\varphi](s) \, \ell(\dd s) + r(t),
	\quad	t \in \Z.
	\end{split}
	\end{equation*}
	This expression agrees with that in \eqref{eq:meanAsymptotics}, provided that the following identity holds:
	\begin{equation}	\label{eq:coefficients of expansions G and L}
	\sum_{k=j}^{k_\lambda} \frac{(-1)^k}{(k-1)!} \stirfst{k}{j} B_{\lambda,k} e^{\lambda k} = \frac1{(j-1)!} A_{\lambda,j}.
	\end{equation}
	We now establish \eqref{eq:coefficients of expansions G and L} via a rather technical and computation-heavy proof.
	
	To this end, notice that \eqref{eq:coefficients of expansions G and L} can be read entrywise
	and, therefore, follows from the following one-dimensional result.
	
	\begin{lem}	\label{lem:coefficients in Laurent series}
		Let $D \subseteq \C$ be open and $f,g: D \to \C$ be meromorphic in $D$.
		Assume that, for some $\lambda \in D$, $f$ has a pole in $\lambda$ of multiplicity $k_\lambda \in \N$
		with Laurent expansion
		\begin{equation*}
		f(z) = a_{k_{\lambda}} (z-\lambda)^{-k_\lambda} + \ldots + a_{1} (z-\lambda)^{-1} + h_1(z)
		\end{equation*}
		where $h_1: D \to \C$ is holomorphic in a neighborhood of $z=\lambda$ and $a_1,\ldots,a_{k_\lambda} \in \C$,
		$a_{k_\lambda} \not = 0$.
		Further, suppose that $f(z) = g(e^{-z})$ for $z \in D$.
		Then $e^{-\lambda}$ is a pole of $g$ of multiplicity $k_\lambda$ and 
		\begin{equation*}
		g(z) = b_{k_\lambda} (z-e^{-\lambda})^{-k} + \ldots + b_1 (z-e^{-\lambda})^{-1} + h_2(z)
		\end{equation*}
		for a function $h_2: D \to \C$ which is holomorphic in a neighborhood of $z = e^{-\lambda}$
		and for coefficients $b_1,\ldots,b_{k_\lambda}$ satisfying 
		\begin{equation}	\label{eq:coefficients of expansions f and g}
		\sum_{k=i}^{k_\lambda} \frac{(-1)^k}{(k-1)!} \stirfst{k}{i} b_k e^{\lambda k} = \frac1{(i-1)!} a_i.
		\end{equation}
	\end{lem}
	It is possible that this lemma can be found in the existing literature;
	however, we were not able to identify a suitable reference.
	
	\begin{proof}[Proof of Lemma \ref{lem:coefficients in Laurent series}]
		The fact that $e^{-\lambda}$ is a pole of $g$ of order $k_\lambda$ and that \eqref{eq:coefficients of expansions f and g} holds for $j=k_\lambda$	follows from \cite[Lemma 2.6]{Kolesko+al:2025}.
		Now let
		\begin{align*}
		F(z) &= (z-\lambda)^{k_\lambda} f(z),
		\end{align*}
		so that
		\begin{align}
		a_{k} &= \frac{1}{(k_\lambda - k)!} F^{(k_\lambda-k)}(\lambda) , \\
		b_{k} &= \frac{1}{(k_\lambda - k)!} \Big( (\Psi \cdot F)\circ (-\log) \Big)^{(k_\lambda-k)}(e^{-\lambda}), \label{eq:bk}
		\end{align}
		where $\log$ is a branch of the complex logarithm that is holomorphic in some neighbourhood of $e^{-\lambda}$ and
		\begin{equation*}
		\Psi(z) = \bigg( \frac{e^{-z} - e^{-\lambda}}{z-\lambda}\bigg)^{k_\lambda}.
		\end{equation*}
		We begin by applying Fa\'a di Bruno formula to \eqref{eq:bk} for $k < k_\lambda$. In what follows, $\mathbb{B}_{n,k}$ denote the ordinary partial Bell polynomials. In the following calculation we use well-known properties of $\mathbb{B}_{n,k}$, which may be found e.\,g.\ in \cite[Chapter 3.3]{Comtet:1974}.
		\begin{equation*}
		\begin{split}
		b_k 
		&= \frac{1}{(k_\lambda-k)!} \sum_{j=1}^{k_\lambda-k} (\Psi\cdot L)^{(j)} (\lambda) \mathbb{B}_{{k_\lambda-k},j}\big((-\log)'(e^{-\lambda}), \cdots, (-\log)^{(k_\lambda-k-j+1)}(e^{-\lambda})\big) \\
		&= \frac{1}{(k_\lambda-k)!} \sum_{j=1}^{k_\lambda-k} (\Psi\cdot L)^{(j)} (\lambda) \mathbb{B}_{{k_\lambda-k},j}\big(-e^\lambda, \cdots, (-e^\lambda)^{k_\lambda-k-j+1}(k_\lambda-k-j)!\big) \\
		&= \frac{1}{(k_\lambda-k)!} \sum_{j=1}^{k_\lambda-k} (\Psi\cdot L)^{(j)} (\lambda) (-e^\lambda)^{k_\lambda-k} \mathbb{B}_{{k_\lambda-k},j}\big(0!, \cdots, (k_\lambda-k-j)!\big) \\
		&= \frac{(-e^\lambda)^{k_\lambda-k}}{(k_\lambda-k)!} \sum_{j=1}^{k_\lambda-k} \stirfst{k_\lambda-k}{j} (\Psi\cdot L)^{(j)}(\lambda).		
		\end{split}
		\end{equation*}
		Note that we may include $j=0$ in the sum since the Stirling number is then $0$. Further,
		\begin{equation*}
		\frac{e^{-z} - e^{-\lambda}}{z-\lambda} = e^{-\lambda} \sum_{k=0}^\infty \frac{(z-\lambda)^k}{k!} \frac{(-1)^{k+1}}{k+1},
		\end{equation*}
		hence, using Fa\`a di Bruno formula again, we obtain, for $m \in \N$,
		\begin{equation*}
		\begin{split}
		\Psi^{(m)}(\lambda) &= \sum_{i=1}^m \frac{k_\lambda!}{(k_\lambda-i)!}(-e^{-\lambda})^{k_\lambda-i} \mathbb{B}_{m,i} \bigg(\frac{e^{-\lambda}}2, \frac{-e^{-\lambda}}3, \cdots, \frac{(-1)^{m-i+2}e^{-\lambda}}{m-i+2} \bigg) \\
		&= (-e^{-\lambda})^{k_\lambda} \sum_{i=1}^m \frac{k_\lambda!}{(k_\lambda-i)!} \mathbb{B}_{m,i} \bigg(\frac{-1}2, \frac13, \cdots, \frac{(-1)^{m-i+1}}{m-i+2} \bigg) \\
		&= (-e^{-\lambda})^{k_\lambda} \frac{k_\lambda! m!}{(k_\lambda + m)!} \mathbb{B}_{k_\lambda+m,k_\lambda}(1,-1,\cdots,(-1)^{m+1}) \\
		&= (-e^{-\lambda})^{k_\lambda} \frac{k_\lambda! m!}{(k_\lambda + m)!} (-1)^{m}\stirscd{k_\lambda + m}{k_\lambda}.
		\end{split}
		\end{equation*}
		Note that the last formula is valid also for $m=0$ since we have
		\begin{equation*}
		\Psi(\lambda) = (-e^{-\lambda})^{k_\lambda}.
		\end{equation*}
		Therefore,
		\begin{equation*}
		\begin{split}
		(\Psi \cdot L)^{(j)}(\lambda) &= \sum_{n=0}^j \binom{j}{n} L^{(n)}(\lambda) \Psi^{(j-n)}(\lambda) \\
		&= \sum_{n=0}^j \binom{j}{n} n! a_{k_\lambda-n} (-e^{-\lambda})^{k_\lambda}(-1)^{j-n} \frac{k_\lambda!(j-n)!}{(k_\lambda + j-n)!} \stirscd{k_\lambda + j -n}{k_\lambda} \\
		&= (-e^{-\lambda})^{k_\lambda} \sum_{n=0}^j a_{k_\lambda-n} (-1)^{j-n} \frac{k_\lambda! j!}{(k_\lambda + j -n)!} \stirscd{k_\lambda + j -n}{k_\lambda} \\
		&= (-e^{-\lambda})^{k_\lambda} \sum_{n=k_\lambda-j}^{k_\lambda} a_{n} (-1)^{k_\lambda+j+n} \frac{k_\lambda!j!}{(n+j)!} \stirscd{n+j}{k_\lambda}
		\end{split}
		\end{equation*}
		which means that
		\begin{equation*}
		\begin{split}
		b_{k} e^{\lambda k}&= \frac{(-1)^{k}k_\lambda!}{(k_\lambda-k)!} \sum_{j=0}^{k_\lambda-k} \stirfst{k_\lambda-k}{j} \sum_{n=k_\lambda-j}^{k_\lambda} a_{n} (-1)^{k_\lambda+j+n} \frac{j!}{(n+j)!}\stirscd{n+j}{k_\lambda}
		\end{split}
		\end{equation*}
		and thus
		\begin{multline*}
		\sum_{k=i}^{k_\lambda} \frac{(-1)^k}{(k-1)!} \stirfst{k}{i} b_{k} e^{\lambda k} \\
		\begin{aligned}
		&= \sum_{k=i}^{k_\lambda} \stirfst{k}{i} \frac{k_\lambda!}{(k_\lambda-k)!(k-1)!} \sum_{j=0}^{k_\lambda-k} \stirfst{k_\lambda-k}{j} \sum_{n=k_\lambda-j}^{k_\lambda} a_{n} (-1)^{k_\lambda+j+n} \frac{j!}{(n+j)!} \stirscd{n+j}{k_\lambda} \\
		&= k_\lambda \sum_{n=i}^{k_\lambda} a_{n} \sum_{j=k_\lambda-n}^{k_\lambda-i} \frac{j!(-1)^{k_\lambda+j+n}}{(n+j)!} \stirscd{n+j}{k_\lambda} \sum_{k} \stirfst{k}{i} \stirfst{k_\lambda-k}{j} \binom{k_\lambda-1}{k-1} \\&= k_\lambda \sum_{n=i}^{k_\lambda} a_{n} \sum_{j=k_\lambda-n}^{k_\lambda-i} \frac{j!(-1)^{k_\lambda+j+n}}{(n+j)!} \stirscd{n+j}{k_\lambda} \stirfst{k_\lambda}{i+j}\binom{j+i-1}{i-1} \\
		&= \frac{k_\lambda}{(i-1)!} \sum_{n=i}^{k_\lambda} a_{n} \sum_{j=0}^{n-i} (-1)^{j} \stirscd{k_\lambda + j}{k_\lambda} \stirfst{k_\lambda}{k_\lambda - n + i + j} \frac{(k_\lambda-n+i+j-1)!}{(k_\lambda+j)!}.
		\end{aligned}
		\end{multline*}
		Finally, the last sum equals $0$ whenever $n > i$ (see e.\,g.\ \cite[Exercise 6.19]{Graham+al:1994}). That is,
		\begin{equation*}
		\sum_{k=i}^{k_\lambda} \frac{(-1)^k}{(k-1)!} \stirfst{k}{i} b_{k} e^{\lambda k} = \frac{k_\lambda}{(i-1)!} a_{i} \stirscd{k_\lambda}{k_\lambda} \stirfst{k_\lambda}{k_\lambda} \frac{(k_\lambda-1)!}{k_\lambda!} = \frac1{(i-1)!} a_{i}
		\end{equation*}
		which ends the proof.
	\end{proof}
	
\end{appendix}

\bibliographystyle{amsplain}
\bibliography{Branching}

\end{document}